\documentclass{amsart}
\usepackage{amsmath}
\usepackage{amsthm}
\usepackage{listings}
\usepackage{xcolor}
\usepackage{mathtools}
\usepackage{paralist}
\usepackage{scalerel}
\usepackage{pbox}
\usepackage[noadjust]{cite}
\usepackage{thmtools,thm-restate}
\usepackage{thm-restate}

\newtheorem{theorem}{Theorem}[section]
\newtheorem{proposition}[theorem]{Proposition}
\newtheorem{lemma}[theorem]{Lemma}
\usepackage{amssymb}

\usepackage{enumitem}  
\usepackage{stmaryrd}

\setenumerate{label={\normalfont(\textit{\roman*})}, ref={(\textrm{\roman*})},
wide
}

\usepackage{ifthen}
\usepackage{hyperref}
\usepackage[mathscr]{euscript}
\usepackage{lineno}

\newtheorem{corollary}[theorem]{Corollary}

\theoremstyle{definition}
\newtheorem{definition}[theorem]{Definition}
\newtheorem{example}[theorem]{Example}
\newtheorem{remark}[theorem]{Remark}

\newtheorem{question}[theorem]{Question}
\newtheorem{problem}[theorem]{Problem}

\DeclareMathOperator*{\EEE}{\scalerel*{\mathbb{E}}{\textstyle\sum}}

\renewcommand{\Re}{\operatorname{Re}}
\renewcommand{\Im}{\operatorname{Im}}

\newcommand{\id}{\operatorname{id}}

\newcommand{\spanZ}{\operatorname{span}_{\ZZ}}
\newcommand{\spanlin}{\operatorname{span}}
\newcommand{\vol}{\operatorname{vol}}
\newcommand{\covol}{\operatorname{covol}}
\newcommand{\ord}{\operatorname{ord}}
\newcommand{\ab}{\operatorname{ab}}

\newcommand{\IP}{\mathrm{IP}}

\newcommand{\norm}[1]{\left\lVert #1 \right\rVert}

\newcommand{\normbig}[1]{\big\lVert #1 \big\rVert}
\newcommand{\normsml}[1]{\lVert #1 \rVert}

\newcommand{\normLip}[1]{\left\lVert #1 \right\rVert_{\mathrm{Lip}}}

\newcommand{\fp}[1]{\left\{ #1 \right\} }
\newcommand{\fpnormal}[1]{\{ #1 \} }

\newcommand{\ip}[1]{\left\lfloor #1 \right\rfloor }
\newcommand{\ipnormal}[1]{\lfloor #1 \rfloor }

\newcommand{\nint}[1]{\left\lfloor #1 \right\rceil}
\newcommand{\nintnormal}[1]{\lfloor #1 \rceil}

\newcommand{\fpa}[1]{\left\lVert #1 \right\rVert}\newcommand{\floor}[1]{\left\lfloor #1 \right\rfloor}
\newcommand{\bra}[1]{\left( #1 \right)}
\newcommand{\brabig}[1]{\big( #1 \big)}

\newcommand{\braif}[1]{\left\llbracket #1 \right\rrbracket}
\newcommand{\braBig}[1]{\Big( #1 \Big)}
\renewcommand{\tilde}{\widetilde}
\renewcommand{\bar}{\overline}

\newcommand{\abs}[1]{\left|#1\right|}

\newcommand{\absBig}[1]{\Big|#1\Big|}

\newcommand{\set}[2]{\left\{ #1 \ \middle| \ #2 \right\} }

\newcommand{\ceil}[1]{\left\lceil #1 \right\rceil}

\newcommand{\cmp}{\mathrm{ht}}

\newcommand{\pvec}[1]{#1}
\renewcommand{\vec}[1]{#1}

\newcommand{\e}{\varepsilon}
\renewcommand{\a}{\alpha}
\renewcommand{\b}{\beta}
\renewcommand{\c}{\gamma}

\newcommand{\NN}{\mathbb{N}}

\newcommand{\QQ}{\mathbb{Q}}
\newcommand{\PP}{\mathbb{P}}
\newcommand{\EE}{\mathbb{E}}
\newcommand{\ZZ}{\mathbb{Z}}
\newcommand{\RR}{\mathbb{R}}
\newcommand{\CC}{\mathbb{C}}

\DeclareMathAlphabet{\mathpzc}{OT1}{pzc}{m}{it}

\newcommand{\cO}{\mathcal{O}}

\newcommand{\cG}{\mathcal{G}}
\newcommand{\cH}{\mathcal{H}}

\newcommand{\cF}{\mathcal{F}}
\newcommand{\cS}{\mathcal{S}}
\newcommand{\cR}{\mathcal{R}}

\newcommand{\cP}{\mathcal{P}}

\newcommand{\fS}{\mathsf{P}}

\newcommand{\JK}[1]{}

\definecolor{checked}{RGB}{200,200,200}
\definecolor{final}{RGB}{ 255, 255, 255}
\definecolor{fresh}{RGB}{ 0, 128, 0}
\definecolor{wrong}{RGB}{ 255, 0, 0	}
\definecolor{Red}{RGB}{128,0,0}

\newcommand{\freq}{\operatorname{freq}}\newcommand{\cnt}{\pi}\newcommand{\bb}{\mathbf}
\newcommand{\fword}[1]{#1}

\renewcommand{\subset}{\subseteq}

\newcommand*\patchAmsMathEnvironmentForLineno[1]{\expandafter\let\csname old#1\expandafter\endcsname\csname #1\endcsname
  \expandafter\let\csname oldend#1\expandafter\endcsname\csname end#1\endcsname
  \renewenvironment{#1}{\linenomath\csname old#1\endcsname}{\csname oldend#1\endcsname\endlinenomath}}\newcommand*\patchBothAmsMathEnvironmentsForLineno[1]{\patchAmsMathEnvironmentForLineno{#1}\patchAmsMathEnvironmentForLineno{#1*}}\AtBeginDocument{\patchBothAmsMathEnvironmentsForLineno{equation}\patchBothAmsMathEnvironmentsForLineno{align}\patchBothAmsMathEnvironmentsForLineno{flalign}\patchBothAmsMathEnvironmentsForLineno{alignat}\patchBothAmsMathEnvironmentsForLineno{gather}\patchBothAmsMathEnvironmentsForLineno{multline}}

\newcommand{\fH}{\mathsf{H}}

\newcommand{\gp}{GP}\newcommand{\genpoly}{generalised polynomial}
\newcommand{\gpword}{bracket word}
\newcommand{\gpwords}{\gpword{}s}
\newcommand{\word}{word}
\newcommand{\pgp}{parametric GP}

\begin{document}

\author[B.\ Adamczewski]{Boris Adamczewski}

\author[J.\ Konieczny]{Jakub Konieczny}

\title[Bracket words]{Bracket words: a generalisation of Sturmian words arising from generalised polynomials}

\begin{abstract}
Generalised polynomials are maps constructed by applying the floor function, addition, and multiplication to polynomials. 
Despite superficial similarity, generalised polynomials exhibit many phenomena which are impossible for polynomials. 
In particular, there exist generalised polynomial sequences which take only finitely many values without being periodic; 
examples of such sequences include the Sturmian words, as well as more complicated sequences like $\ip{ 2\fp{ \pi n^2 + \sqrt{2}n\ip{\sqrt{3}n} }}$. 

The purpose of this paper is to investigate letter-to-letter codings of finitely-valued generalised polynomial sequences, which we call \emph{bracket words}, 
from the point of view of combinatorics on words. We survey existing results on generalised polynomials and their corollaries in terms of bracket words, 
and also prove several new results. Our main contribution is a polynomial bound on the subword complexity of bracket words.
\end{abstract}

\subjclass[2010]{}

\maketitle 

\setcounter{tocdepth}{1}
\tableofcontents
\newpage

\section{Introduction}

Generalised polynomial sequences are expressions built up from the usual polynomials with the use of addition, 
multiplication, and the floor function, such as
\[
	g(n) = 2n \ip{ \sqrt{2}n \ipnormal{\sqrt{3}n }^2+\sqrt{5}n^3} -  7 n^2\,.
\]
They have been investigated by many authors, with particular emphasis placed on problems involving uniform distribution and taking their source in 
Weyl's equidistribution theorem for classical polynomials. 
See, for instance, \cite{Haland-1993,Haland-1994, BergelsonHaland-1993,  BergelsonLeibman-2007,Leibman-2012,BergelsonHalandSon-2020} and the references therein.

An important feature that distinguishes generalised polynomials from ordinary polynomials is that they can be bounded without being constant. 
As a basic example, for any polynomial $p(x)\in\mathbb R[x]$, the sequence of fractional parts $\bra{\fp{p(n)}}_{n=0}^\infty$ is a generalised polynomial 
sequence with values in $[0,1)$, which is non-constant for most choices of $p$. 
In fact, it is also possible for non-constant generalised polynomials sequences to take only  finitely many values.  
A notable example of a class of such finitely-valued generalised polynomial sequences is provided by Sturmian words. 
Indeed, every Sturmian word $\bb a = (a_n)_{n=0}^\infty$ over $\{0,1\}$ can be defined by a generalised polynomial formula 
\begin{equation}\label{eq:intro-1}
a_n = {\ip{\a n+\b} - \ip{\a (n-1) + \b}} 
\end{equation}
 for some $\a \in [0,1) \setminus \QQ$ and $\beta \in [0,1)$ (possibly with the floor function $\ip{\cdot}$ replaced by the ceiling function $\ceil{\cdot}$). 

 Motivated by this fundamental example, we endeavour to investigate finitely-valued generalised polynomials from the perspective of combinatorics on words. 
 More precisely, we study letter-to-letter codings of finitely-valued generalised polynomial sequences, which we dub \emph{bracket words} 
 (cf.{} Definition \ref{def:def:bra-word}). Thus, a {\gpword} $\bb a = (a_n)_{n=0}^\infty$ over an alphabet $\Sigma$ takes the form $a_n = c(g(n))$, were 
 $g\colon \NN_0 \to \RR$ is  a finitely-valued generalised polynomial sequence and $c$ is a map from the finite set $g(\NN_0)$ to $\Sigma$. Throughout the paper, we let $\NN = \{1,2,\dots\}$ denote the set of positive integers and put $\NN_0 = \NN \cup \{0\}$. 
 Bracket words are thus first thought of as a broad arithmetical generalisation of Sturmian words.  

Among the several equivalent definitions of Sturmian words, one involves codings of circle rotations. 
Indeed, the word $\bb a$ defined by \eqref{eq:intro-1} can be described with the help of the rotation $R_\alpha \colon \RR/\ZZ \to \RR/\ZZ$, $x \mapsto x+\alpha$, as 
\begin{equation}\label{eq:intro-2}
	a_n = 
	\begin{cases}
		1 &\text{if } R_\alpha^n(\beta) \in [0,\alpha)\,,\\
		0 &\text{otherwise.}
	\end{cases}
\end{equation}
A seminal paper of Bergelson and Leibman \cite{BergelsonLeibman-2007} provides a dynamical representation of any bounded generalised polynomial sequence  
in terms of nilrotations, which are translations on nilmanifolds, hence linking the theory of generalised polynomials with nilpotent dynamics. 
As a consequence, each bracket word can be obtained as a coding of a nilrotation with respect to a semialgebraic partition.  
Conversely, each coding of a nilrotation which involves a semialgebraic partition gives rise to a bracket word. 

Nilsystems have received a considerable amount of attention in the past two decades. A major source of interest stems from their significance in the study 
of multiple ergodic averages, as demonstrated by Host and Kra \cite{HostKra-2005} and Ziegler \cite{Ziegler-2007}. 
Nilsystems also play a key role in additive combinatorics, specifically in the relatively modern branch of higher order Fourier analysis, 
initiated by Gowers in his work on an alternative proof of {S}zemer\'{e}di's theorem \cite{Gowers-2001}. 
The importance of nilsystems was revealed by Green, Tao, and Ziegler, who characterised Gowers uniformity in terms of correlations 
with nilsequences \cite{GreenTaoZiegler-2012}, see also \cite{GreenTao-2010}. For an introduction to higher order Fourier analysis, we refer to \cite{Tao-book}. 
Thus, the characterisation of bracket words mentioned above, which we discuss in more detail in Section \ref{sec:nil}, 
allows us to take advantage of deep pre-existing results.

A great number of combinatorial results concerning Sturmian words can be derived from 
the continued fraction expansion of the irrational parameter $\alpha$ and some related expansion, the so-called Ostrowski expansion, of the real parameter 
$\beta$ (these parameters being defined as in \eqref{eq:intro-1} or \eqref{eq:intro-2}).   
In contrast, the study of bracket words involves much more general Diophantine problems related to 
simultaneous rational approximations of real numbers. 
For this reason, one cannot expect to always obtain results as accurate as for Sturmian words.  
For instance,  Sturmian words can be charaterised in terms of their subword complexity. Recall that for an infinite word $\bb a$, 
the subword complexity $p_{\bb a}(N)$ is the count of distinct length-$N$ subwords which appear in $\bb a$, and that $\bb a$ is Sturmian 
if and only if $p_{\bb a}(N) = N+1$ for each $N \geq 1$, which is the slowest rate of growth possible for a word which is not eventually periodic. 
Contrary to Sturmian words, bracket words cannot be characterised in terms of their subword complexity, but we still prove the following  
polynomial upper bound, which is our main new result. It is also shown in Section \ref{sec:subword}  that this result is essentially the best possible 
(c.f. Propositions \ref{prop: comp1} and \ref{prop: comp2}).

\begin{restatable}{alphatheorem}{thmsubword}
\label{thm:A}Let $\bb a$ be a {\gpword}. Then there exists a constant $C > 0$ such that $p_{\bb a}(N) = O(N^C)$ for all $N \in \NN$.
\end{restatable}

In another direction, Sturmian words can also be characterised in terms of balance: a word $\bb a$ over $\{0,1\}$ is Sturmian if and only if it is not eventually periodic and for each $N \geq 1$ and each pair $u,v$ of length-$N$ subwords of $\bb a$, the number occurrences of the symbol $1$ in $u$ and in $v$ is either the same or differs by $1$. 
For bracket words, we have weaker estimates on frequencies of symbols, which we explore in Sections \ref{sec:facts} and \ref{sec:growth}.

Since Sturmian words are {\gpwords}, all of our results apply in particular to Sturmian words. In all instances, the special case involving Sturmian words was shown earlier than the general case, or follows quickly from earlier results. However, there are several facts which we suspect might not be very widely known. These facts include existence of frequencies along the primes (Theorem \ref{thm:facts:prime-freq-exists}) and the $\mathrm{IP}^*$ recurrence (Theorem \ref{thm:nil:IP*-rec}).

A different source from which we draw inspiration is the theory of automatic sequences, and computation theory in general. Recall that a word $\bb a= (a_n)_{n=0}^\infty$ is $k$-automatic if there exists a deterministic finite automaton which, given the base-$k$ expansion of $n$ as input, produces $a_n$ as output (for extensive introduction, see \cite{AlloucheShallit-book}). 
Although $k$-automatic words are generally not {\gpwords} (cf.\ Remark \ref{rmk:NOT-auto} below), we find this comparison helpful because the classes of $k$-automatic sequences and the {\gpwords} enjoy similar closure properties with respect to Cartesian products, codings, finite modifications and certain types of rearrangements, as discussed at length in Section \ref{sec:cl}. 
For instance, bracket words taking values in a ring, equipped with coordinatewise addition and multiplication, form a ring. In contrast, the class of Sturmian sequences is considerably too small for analogous closure properties to be true. 

Access to the closure properties mentioned above allows us to look at bracket words from a computational perspective. For example, given three bracket words 
$\bb a^{(0)},\bb a^{(1)} \in \Sigma^\infty$, and $\bb b \in \{0,1\}^\infty$, we can construct a new bracket word $\bb a$ given by 
\[
	a_n =
	\begin{cases}
		a^{(0)}_n &\text{if } b_n = 0\,,\\
		a^{(1)}_n &\text{if } b_n = 1\,.		
	\end{cases}
\]
As a consequence, bracket words can encode any finite computation including real constants, the basic arithmetic operations $+,\times$, the operation of taking the integer part $\ip{\cdot}$ and conditional statements involving $=$ and $<$ (c.f.{} Proposition \ref{prop:cl:g-in-I}).

\begin{remark}\label{rmk:NOT-auto}
	We stress that, even though we use automatic sequences as a motivating example, automatic words in general are not bracket words. In fact, there are no automatic words that are also bracket words, except for eventually periodic ones (see Theorem \ref{thm:BK}). Rather, our interest in automatic sequences stems from the fact that they share with bracket words certain properties, as discussed above. Other classes of sequences which  we discuss in this context are morphic sequences and regular sequences (in the sense of Allouche and Shallit, \cite{AlloucheShallit-1992}). 
\end{remark}

In light of the discussion above, it is natural to ask which properties of bracket words can be tested algorithmically. We will say that a problem is \emph{decidable} if there exists an algorithm which solves it on all inputs, and \emph{undecidable} otherwise. Many properties are known to be decidable for $k$-automatic sequences, see e.g., \cite{CharlierRampersadShallit-2012} and references therein; a practical implementation is discussed in \cite{GocHenshallShallit-2013} and \cite{Mousavi-2016}. Similarly, there are many results on decidability for $k$-regular sequences (see e.g.\ \cite{AlloucheShallit-1992}, \cite{KrennShallit-2022}) and for morphic sequences (see e.g.\ \cite{Durand-2013,Durand-2013b}).

For problems involving generalised polynomial sequences, we assume that the algorithm is provided with a formula involving only polynomials, addition, multiplication and the floor function which represents the sequence. In \cite{Leibman-2012}, Leibman constructed a ``canonical'' representation of a bounded generalised polynomial, which is essentially unique. As a consequence, the problem of determining if a given generalised polynomial is zero almost everywhere is decidable. Likewise, it is decidable whether two given bracket words are equal almost everywhere. Here, a statement $\varphi(n)$, involving a parameter $n \in \NN_0$, is said to hold almost everywhere if the set $\set{n \in \NN_0}{\neg \varphi(n)}$ of positions where it is false has asymptotic density zero. Somewhat surprisingly, the problem of verifying equality everywhere turns out to be undecidable, as shown in Section \ref{sec:comp}. (For terminology used, see Section \ref{ssec:GP-def} and Remark \ref{rmk:alg-coeff}.)

\begin{restatable}{alphatheorem}{thmdecide}
\label{thm:comp}
It is undecidable if two given bracket words $\bb a$ and $\bb b$ with algebraic coefficients
	defined over a finite alphabet $\Sigma$ are equal. 
\end{restatable}

Finally, we discuss examples of ``naturally occurring'' words for which we can show that they are, or that they are not, bracket words. In the positive direction, we note that the characteristic word $\bb 1_{F} = 1111010010010001\cdots$ of the Fibonacci numbers is a bracket word, which can be traced back to the observation that the golden ratio 
$(1+\sqrt{5})/2$ is a Pisot unit and the group of units of the field it generates has rank $1$. In Section \ref{sec:pisot} we discuss generalisations of this example, corresponding to other Pisot and Salem numbers. Similar but slightly weaker results were obtained in \cite{ByszewskiKonieczny-2018-TAMS}. 

In the negative direction, several criteria for proving that an infinite word is not a {\gpword} follow from results discussed in the reminder of the paper. Several other techniques were developed in a series of papers by Byszewski and the second-named author \cite{ByszewskiKonieczny-2018-TAMS,ByszewskiKonieczny-2020-CJM,Konieczny-2021-JLM}, leading to a proof that automatic sequences which are not eventually periodic are not bracket words. As explicit applications of the aforementioned methods, we mention the characteristic words of primes and of squares are not bracket words. The same applies to many words coming from number theory, such as $(\varphi(n) \bmod q)_{n=0}^\infty$, where $\varphi$ denotes the totient function and $q \geq 3$.

We point out that a significant part of the paper is devoted to a survey of known results concerning generalised polynomial sequences and their interpretation in terms of {\gpwords}.  Beyond that, we prove several new results, including Theorems \ref{thm:A} and \ref{thm:comp}. The paper is organised as follows. 

\subsection*{Organisation of the paper} 

Sections \ref{sec:def} and \ref{sec:ex} are concerned with setting up the terminology and providing examples of {\gpwords}.
In Section \ref{sec:nil}, we discuss the connection between dynamics on nilmanifolds and generalised polynomials, which is one of the key tools used in subsequent sections. 
For the sake of readability, we delegate some related material to Appendix \ref{app:nil}.
In Sections \ref{sec:constr} and \ref{sec:cl}, we discuss closure properties of {\gpwords} and other ways in which {\gpwords} can be constructed. These results allow us to perform many basic operations on {\gpwords} later in the paper. 
Sections \ref{sec:facts}--\ref{sec:comp} each concern a different facet of {\gpwords} and can mostly be read independently from one another. In Sections \ref{sec:facts} and \ref{sec:growth}, we discuss frequencies of symbols and subwords in {\gpwords}; qualitative results are included in \ref{sec:facts} and quantitative --- in  \ref{sec:growth}. 
In Section \ref{sec:comp}, we discuss the canonical representation of {\gpwords}, based on \cite{Leibman-2012}, and its consequences in terms of decidability; next, we prove several undecidability results, including Theorem \ref{thm:comp}. 
In Section \ref{sec:pisot}, we consider a class of {\gpwords}, consisting of characteristic words  of certain sets of integers described by a linear recurrence. In Section \ref{sec:neg}, we consider the problem of proving that a given word is not a {\gpword}. We collect several criteria from previous sections and from \cite{ByszewskiKonieczny-2018-TAMS,Konieczny-2021-JLM}, and give a number of new examples and applications.
In Section \ref{sec:subword}, we discuss subword complexity of {\gpwords} and lay the lay the groundwork for Theorem \ref{thm:A}, which we prove in Sections \ref{sec:sc-induction}--\ref{sec:sc-proof}.

\subsection*{Acknowledgements} 
The second-named author works within the framework of the LABEX MILYON (ANR-10-LABX-0070) of Universit\'e de Lyon, within the program "Investissements d'Avenir" (ANR-11-IDEX- 0007) operated by the French National Research Agency (ANR). We are also grateful to Val\'{e}rie Berth\'{e}, Jakub Byszewski, and Sam Chow for helpful comments, and to the anonymous referee for careful reading of our paper and valuable corrections.

\section{Definitions and notation}\label{sec:def}

\subsection{Combinatorics on words} 

An alphabet $\Sigma$ is a finite set of symbols, also called letters. A finite word over $\Sigma$ is a 
finite sequence of letters in $\Sigma$  or, equivalently, an element of $\Sigma^*= \bigcup_{\ell = 0}^\infty \Sigma^{\ell}$, 
the free monoid  generated by $\Sigma$ 
with respect to the concatenation of finite words.  
The length of a finite word $w$, that is, the number 
of symbols in $w$, is denoted by $\vert w\vert$. 
We let $\epsilon$ denote the empty word, the neutral element of $\Sigma^*$. 
An infinite word $\bb a= (a_n)_{n = 0}^\infty$ over $\Sigma$ is an element of $\Sigma^{\infty}$, or, equivalently, an infinite sequence with values in $\Sigma$, 
({\it i.e.}, a map from $\NN_0$ to $\Sigma$). It is sometimes represented as $\bb a=a_0a_1\cdots$. 
Throughout, we use bold symbols $\bb a, \bb b, \dots$ to denote infinite words.

Let $\Sigma$ and $\Pi$ be two alphabets. A \emph{morphism} 
is a map  $\sigma \colon \Sigma^* \to \Pi^*$ that obeys the identity $\sigma(\fword{u} \fword{v}) = \sigma(\fword{u})\sigma(\fword{v})$ for 
all words $\fword{u},\fword{v}\in \Sigma^*$. Note that a morphism $\sigma$ is uniquely determined by the knowledge of $\sigma(x)$ 
for all $x \in \Sigma$. 
A map from $\Sigma$ to $\Pi^*$ naturally (and uniquely) extends as a morphism from $\Sigma^*$ to $\Pi^*$.   
A morphism is  said to be \emph{non-erasing} if $\sigma(x)\not= \epsilon$ for all $x\in \Sigma$. 
A morphism $\sigma$ over $\Sigma^*$ is said to be $k$-\emph{uniform} if 
$\vert \sigma(a)\vert =k$ for every letter $a$ in~$\Sigma$, and just 
\emph{uniform} if it is $k$-uniform for some~$k$. A $1$-uniform morphism is called a $\emph{coding}$. 
Furthermore, there is a natural way to extend the action of a non-erasing morphism to infinite words, that is, as a map 
from $\Sigma^\infty$ to  $\Pi^\infty$ defined by $\sigma(\bb a) = \sigma(a_0)\sigma(a_1) \cdots$ for $\bb a \in \Sigma^\infty$. 
Such a map is still called a morphism and denoted by $\sigma$. 

\subsection{Generalised polynomials and bracket words}\label{ssec:GP-def}
For $x \in \RR$, we let $\ip{x} \in \ZZ$ denote the integer part of $x$ (also known as the floor), which is the unique integer with $\ip{x} \leq x < \ip{x}+1$. Similarly, we let $\fp{x} = x - \ip{x} \in [0,1)$ denote the fractional part, $\ceil{x} = -\ip{-x} \in \ZZ$ --- the ceiling, $\nint{x} = \ip{x+1/2} \in \ZZ$ --- the nearest integer, and $\fpa{x} = \abs{ x - \nint{x}} \in [0,1/2]$ --- the distance to the nearest integer. 
We extend the notions introduced above to  $x = \bra{x_i}_{i=1}^d \in \RR^d$ for $d \geq 2$ coordinate-wise, meaning that $\ip{x} = \bra{\ip{x_i}}_{i=1}^d \in \ZZ^d$, etc. Similarly, for a map $f \colon X \to \RR$ (where $X$ is any set) we define $\ip{f} \colon X \to \ZZ$ by $\ip{f}(x) = \ip{f(x)}$.

Let $d \in \NN$. We define \emph{{\gp} maps} (or {\genpoly} maps) $\RR^d \to \RR$ as the smallest family such that
\begin{enumerate}
\item\label{it:def:i} each polynomial map $\RR^d \to \RR$ is a {\gp} map;
\item\label{it:def:ii} if $g,h \colon \RR^d \to \RR$ are {\gp} maps then $g+h$ and $g \cdot h$ are {\gp} maps;
\item\label{it:def:iii} if $g \colon \RR^d \to \RR$ is a {\gp} map then $\ip{g}$ is a {\gp} map.
\end{enumerate}
Note that $x \mapsto \fp{x},\ceil{x},\nint{x}$ are {\gp} maps, and the definition of a {\gp} map does not change if in \ref{it:def:iii} we replace $\ip{\cdot}$ with $\fp{\cdot}$, $\ceil{\cdot}$ or $\nint{\cdot}$. The distance for the nearest integer can also be expressed by a generalised polynomial formula, such as $\fpa{x} = \bra{x-\nint{x}}\cdot\bra{2\ceil{x-\nint{x}}-1}$. (Since all maps $\RR^d \to \RR$ obtained from polynomials using addition, multiplication and $\fpa{\cdot}$ is continuous, replacing $\ip{\cdot}$ with $\fpa{\cdot}$ in \ref{it:def:iii} yields a strictly smaller family of maps.)

\begin{remark}\label{rem:composition}
If $g:\RR \to \RR$ and $h: \RR^d \to \RR$ are {\gp} maps, then $g\circ h$ is also a {\gp} map.
\end{remark}

A {\gp} map on a domain $\Omega \subset \RR^d$ (e.g.{} $\Omega = \ZZ^d$ or $\NN_0^d$) is simply the restriction of a {\gp} map on $\RR^d$. In particular, each {\gp} map on $\Omega$ can be extended to $\RR^d$, but the extensions is usually not unique. In this paper, we are particularly interested in finitely-valued {\gp} sequences (i.e.{}, {\gp} maps $\NN_0 \to \RR$), as seen from the perspective of combinatorics on words. This motivates us to pose the following definition.\footnote{Several authors have expressed the sentiment that a better name for ``generalised polynomials'' would have been ``bracket polynomials'', and the main reason to not adopt the latter name is that it is already used in knot theory (see e.g.{} \cite{Leibman-2012}). Fortunately, similar considerations do not apply to the term ``{\gpword}''.}

\begin{definition}\label{def:def:bra-word}
	A \emph{\gpword} over a (finite) alphabet $\Sigma$ is an infinite word $\bb a = (a_n)_{n=0}^\infty \in \Sigma^\infty$ of the form $a_n = c(g(n))$ for all $n \in \NN_0$, where $g \colon \NN_0 \to \RR$ is a finitely-valued {\gp} map and $c \colon g(\NN_0) \to \Sigma$ is an arbitrary map. \end{definition}

\begin{remark}\label{rmk:coding}
	The inclusion of the coding in Definition \ref{def:def:bra-word} does not significantly increase the level of generality. Indeed, if $\bb a = (a_n)_{n=0}^\infty$ is a bracket word over an alphabet $\Sigma$ and $\phi \colon \Sigma \to \RR$ is an arbitrary map then the map $g \colon \NN_0 \to \RR$ given by $g(n) = \phi(a_n)$ is a {\gp} map as a consequence of Corollary \ref{cor:constr:gpword-nice}. However, from the perspective of combinatorics on words, it would be unnatural to restrict our attention to words over alphabets contained in $\RR$. We also point out that in Section \ref{sec:cl} it is frequently more natural to consider words over more general alphabets, especially when it comes to closure under Cartesian products in Proposition \ref{prop:cl:prod}. 
\end{remark}

\begin{remark}\label{rmk:bw-multidim}
	In analogy with Definition \ref{def:def:bra-word}, one can define \emph{$d$-dimensional {\gpwords}} over a finite alphabet $\Sigma$ to be $d$-dimensional infinite words \[ \bb a = (a_{n_1,n_2,\dots,n_d})_{n_1,n_2,\dots,n_d=0}^\infty \in \Sigma^\infty \times \Sigma^\infty \times \cdots \times\Sigma^\infty\]
	of the form $a_{n_1,n_2,\dots,n_d} = c(g(n_1,n_2,\dots,n_d))$ where $g \colon \NN_0^d \to \RR$ is a finitely-valued {\gp} map and $c \colon g(\NN_0^d) \to \Sigma$ is an arbitrary map. We limit the discussion to the $1$-dimensional case for the sake of clarity, but many of the results have their multidimensional analogues. In particular, we point out that multidimensional Sturmian sequences (as defined in \cite[Def.\ 4]{Fernique-2006}), and more generally rotation sequences (as defined in \cite[Def.\ 5]{Fernique-2006}) are multidimensional {\gpwords} (cf.\ \cite[Prop.\ 2]{Fernique-2006}). 
\end{remark}

{
\begin{remark}\label{rmk:alg-coeff}
For a ring $\ZZ \subset A \subset \RR$ we define \emph{{\gp} maps with coefficients in $A$} in a fully analogous way, except that in \ref{it:def:i} we only include polynomial maps with coefficients in $A$. We will be especially interested in {\gp} maps with algebraic coefficients, i.e., $A = \bar\QQ$. One can also define the set of coefficients of a {\genpoly} (see e.g.{} \cite{Haland-1994}), but we avoid using this notion since it depends on the choice of a representation, which is usually not unique (see Section \ref{sec:comp} for further discussion). Slightly informally, we will say that an infinite word $\bb a$ over $\Sigma$ is a \emph{{\gpword} arising from a {\gp} map with algebraic coefficients}, or simply a  \emph{{\gpword} with algebraic coefficients}, if $a_n = c(g(n))$ for every $n \in \NN_0$, where $g \colon \NN_0 \to \RR$ is a finitely-valued {\gp} map with coefficients in $\bar{\QQ}$ and $c \colon g(\NN_0) \to \Sigma$ is a map. 
\end{remark}
}

For $k \geq 2$, a map $ g = (g_i)_{i=1}^k \colon \RR^{d} \to \RR^k$ is {\gp} if for each $i$, $1 \leq i \leq k$, the coordinate map $g_i$ is {\gp}. In most cases, we find it simpler to speak of $k$-tuples of {\gp} maps instead.

A \emph{{\gp} subset of $\Omega \subset \RR^d$} (or simply \emph{a {\gp} set}, if $\Omega$ is clear from the context) is the zero locus of a {\gp} map, that is, a set $E \subset \Omega$ which takes the form 
\(E = \set{\vec x \in \Omega}{ g(\vec x) = 0}\)
for some {\gp} map $g \colon \Omega \to \RR$. As we will see in Section \ref{sec:constr}, $E \subset \NN_0$ is a {\gp} set if and only if $\bb 1_E$ is a {\gpword}, cf.\ Proposition \ref{prop:cl:g-in-I}. (Here and elsewhere, $\bb{1}_E \in \{0,1\}^\infty$ is given by $(\bb{1}_E)_n = 1$ if $n \in E$ and $(\bb{1}_E)_n = 0$ otherwise).  We stress that this notion depends on $\Omega$; in particular, a {\gp} subset of $\NN_0$ will usually not be a {\gp} subset of $\ZZ$.

\subsection{Other notation}\label{sec:Notation}

We briefly summarise some other pieces of notation we use. 

For $N \in \NN_0$, we let $[N] = \{0,1,\dots,N-1\}$. For a quantity $X$, we let $O(X)$ denote any quantity bounded in absolute value by $CX$, where $C$ is a constant. When the $C$ is additionally allowed to depend on a parameter $Y$, we write $O_Y(X)$ instead. Similarly, assuming that $X > 0$, we let $\Omega(X)$ denote any quantity bounded from below by $cX$, where $c > 0$ is a constant. If $X = O(Y)$ and $Y = \Omega(X)$, we write $Y = \Theta(X)$. We occasionally also use the notation $Y \ll X$ when $Y = O(Y)$.
Lastly, we write $o_{n\to\infty}(X)$ for any quantity $Y$ with $\lim_{n \to \infty} Y/X = 0$; if the parameter $n$ is clear from the context, we write $o(X)$ instead. 

In what follows, we will use the Iverson bracket notation. 
For a sentence $\varphi$ we put $\braif{\varphi} = 1$ if $\varphi$ is a true and $\braif{\varphi} = 0$ otherwise. 
By a slight abuse of notation, we also use the Iverson bracket to define infinite words over the alphabet $\{0,1\}$. 
For instance, if $X,Y$ are sets with $Y \subset X$ and $f \colon \NN_0 \to X$ is a map then 
$\braif{f \in Y} \in \{0,1\}^\infty$ is given by $\braif{f \in Y}_n = \braif{f(n) \in Y}$ for all $n \in \NN_0$.

\color{black} 
\section{Examples}\label{sec:ex}

Let us now present several examples of {\gpwords}. In some cases, the fact that the sequence under consideration 
indeed is a {\gpword} will follow directly from the definition, while in other cases it may be more surprising.

\begin{example}\label{ex:ex:1_0}
	Let $\alpha \in \RR \setminus \QQ$ and let $\bb a$ be the {\gpword} defined by $a_n = \ip{1-\fp{n\a}}$ for $n \in \NN_0$. 
	Then $\bb a = \bb{1}_{\{0\}} = 100\cdots$.  More generally, all eventually constant sequences are {\gpwords}. 
	Explicitly, if $\bb a$ is a word over $\Sigma$ with $a_n = b$ for all $n \geq N$ then $a_n = c(g(n))$, 
	where $g \colon \ZZ \to \{-1,0,1,\dots,N-1\}$ is the {\gp} map given by
	\[
		g(n) = -1+ \sum_{m = 0}^{N-1} (m+1) \ip{1-\fp{(n-m)\a}},
	\]
	and $c \colon \{-1,0,1,\dots,N-1\} \to \Sigma$ is given by $c(n) = a_n$ if $n \neq -1$ and $c(-1) = b$.
\end{example}

\begin{example}\label{ex:ex:per}
	Let $Q \in \NN$. Then $\bb a \in [Q]^\infty$ defined by $a_n = n \bmod Q = Q\fp{n/Q}$ is a {\gpword}. 
	More generally, all eventually periodic words are {\gpwords}. For details, see Section \ref{sec:cl}.
\end{example}

\begin{example}\label{ex:ex:sturm}
	One way to define Sturmian words is by an explicit formula. Namely, an infinite word $\bb a$ over $\{0,1\}$ is Sturmian if it is the discrete derivative of a Beatty sequence, meaning that it takes one of the following forms:
	\begin{align}
\label{eq:802:1}		a_n &= \ip{n\a + \b} - \ip{\a (n-1) + \b}, \quad  \text{or}\\
\label{eq:802:2}		a_n &= \ceil{n\a + \b} - \ceil{\a (n-1) + \b}, 
	\end{align}
	for some $\a,\b \in [0,1)$ with $\a$ irrational. (Note that \eqref{eq:802:1} and \eqref{eq:802:2} differ for at most one value of $n$.) 
	Thus, Sturmian words are {\gpwords} (while Beatty sequences are unbounded {\gp} sequences). We also point out that Sturmian words 
	arise from codings of rotations, which gives another way to see that they are {\gpwords}; we explore this point of view further in Section \ref{sec:nil}.
	
	 More concretely, setting $\a = \frac{\sqrt{5}-1}{2}$ and $\b = 0$ we obtain the Fibonacci word, whose initial values are:
	\begin{align*}
1 0 1 0 1 1 0 1 0 1 1 0 1 1 0 1 0 1 1 0 1 0 1 1 0 1 1 0 1 0 1 1 0 1 1 0 1 0 1 1 0 1 0 1 1 0 1 1 0 1 0 1 1 0 1 0 \cdots
	\end{align*}

\end{example}

\begin{example}\label{ex:ex:poly}
	As a generalisation of Example \ref{ex:ex:sturm}, let $p \colon \RR \to \RR$ be a polynomial and let $I \subset [0,1)$ 
	be an interval (or a finite union thereof). Let $\bb a$ be the infinite word over $\{0,1\}$ defined as	
	\[
		a_n =
		\begin{cases}
			1 & \text{ if } \fp{p(n)} \in I,\\
			0 & \text{ otherwise}.
		\end{cases} 
	\]
	Then $\bb a$ is a {\gpword}. As a concrete illustration, let $p(x) = \varphi x^2$, where $\varphi = \frac{\sqrt{5}+1}{2}$ is the golden ratio, 
	and $I = [0,1/4) \cup (3/4,1)$. Then 
	\[
		a_n =
		\begin{cases}
			1 & \text{ if } \fpa{ \varphi n^2} < 1/4,\\
			0 & \text{ if } \fpa{ \varphi n^2} \geq 1/4 .
		\end{cases} 
	\]
The initial values of $\bb a$ are:
	\begin{align*}
&1000101001111011101110111010011111011101111000111110111\cdots
	\end{align*}
\end{example}

\begin{example}\label{ex:ex:rec}
	Let $F = \{0,1,2,3,5,8,13,\dots\}$ be the set of all Fibonacci numbers. Then $\bb{1}_F$ is a {\gpword}.	
	Similarly, let $(t_i)_{i=0}^\infty$ be the sequence given by $t_0 = 0$, $t_1= t_2 =1$ and 
	$t_{i+3} = t_{i+2} + t_{i+1} + t_{i}$ for all $i \in \NN_0$, sometimes called the Tribonacci numbers, 
	and let $T = \set{t_i}{i \in \NN_0}$. Then $\bb{1}_T$ is a {\gpword}. These are special cases of Proposition \ref{prop:pisot:rec}. 
	In an upcoming preprint by the second-named author and Byszewski \cite{ByszewskiKonieczny-upcoming}, it is shown that, more generally, for each $E \subset F$, $\bb{1}_E$ is a {\gpword}.
\end{example}

\begin{example}\label{ex: littlewood}
	Let $\a,\b \in \RR$ and $\e > 0$. Then the infinite word $\bb a$ over $\{0,1\}$ defined by
	\[
		a_n =
		\begin{cases}
			1 & \text{ if } n \neq 0 \text{ and } \fpa{\a n} \cdot \fpa{\b n} < \e/n,\\
			0 & \text{ otherwise},
		\end{cases} 
	\]
	is a {\gpword}, as follows from Proposition \ref{prop:cl:g-in-I}.  
	We point out that a famous conjecture in Diophantine approximation, the Littlewood conjecture, 
	is equivalent to the statement that, for each choice of $\a,\b,\e$, the {\gpword} $\bb a$ defined above is not identically zero. 
	Indeed, in its usual formulation, Littlewood's conjecture asserts that
\begin{equation}\label{eq:232:1}
	\liminf_{n \to \infty} n \cdot \fpa{\a n} \cdot \fpa{\b n} = 0
\end{equation}
for all $\a,\b \in [0,1)$.  A landmark result toward its resolution is due to  Einsiedler, Katok and Lindenstrauss \cite{EinsiedlerKatokLindenstrauss-2006}: 
the set of possible exceptions ({\it i.e.}, the set of pairs $(\a,\b)$ for which \eqref{eq:232:1} is false) has Hausdorff dimension zero. 
\end{example}

\begin{example}
	Let $(n_i)_{i=0}^\infty$ be a sequence of positive integers with $n_{i+1} \geq n_{i}^2$ for all $i$, such as $n_i = 2^{2^i}$. 
	Put $E = \set{n_i}{i \in \NN_0}$. Then $\bb{1}_E$ is a {\gpword}. This is a special case of Proposition \ref{prop:constr:super-sparse}.
\end{example}

\section{Dynamical representation}\label{sec:nil}

In this section, we discuss a dynamical description of bounded {\gp} sequences and {\gpwords}. 
Specifically, we briefly introduce basic facts about nilmanifolds and nilsystems and explain their relation to generalised polynomials, 
as established by Bergelson and Leibman \cite{BergelsonLeibman-2007}.

\subsection{Nilsystems and generalized polynomials}

Classical theory of nilpotent dynamics can be found in \cite{AuslanderGreenHahn-book}. 
In order to maintain the introductory nature of this section, we delegate some of the more technical results to Appendix \ref{app:nil}.  
We also refer to cited references, such as \cite{BergelsonLeibman-2007}, for precise definitions and a more detailed discussion. 

\subsubsection{Nilpotent Lie groups}

Let $G$ be a group. The \emph{lower central series} $(G_i)_{i\geq 0}$ is the chain of subgroups of $G$ inductively defined by 
$G_0= G_1 = G$ and $G_{i+1} = [G,G_i]$ for $i\geq 1$. Here, we let $[G,H]$ denote the group generated by the commutators $[g,h] = ghg^{-1}h^{-1}$ for $g \in G, h \in H$. 
The group $G$ is nilpotent if there exists $s$ such that $G_{s+1}=\{\id_G\}$. 
The smallest such $s$ is called the nilpotency class of $G$ and $G$ is said to be nilpotent of class $s$ or a $s$-step nilpotent group. 
We recall that a Lie group is a smooth manifold obeying the group properties and that satisfies the additional condition that the group operations are differentiable. 
A nilpotent Lie group is a Lie group that is nilpotent. 

\subsubsection{Nilrotations}
{
A \emph{nilmanifold} is a quotient space $G/\Gamma$ where $G$ is a nilpotent Lie group and $\Gamma$ is a discrete cocompact subgroup. 
A \emph{nilsystem} is a dynamical system of the form $(G/\Gamma, T_g),$ where $G/\Gamma$ is a nilmanifold and $T_g$ is a \emph{nilrotation}. 
That is,  there exists some $g \in G$ such that $T_g(h\Gamma) = gh\Gamma$ for all $h \in G$.  
In general, there is no guarantee that $G$ is connected, and we let $G^{\circ}$ denote the connected component of $\id_G$. 
We may assume without loss of generality that $G^{\circ}$ is simply connected.
The simplest example of a nilsystem is the aforementioned rotation on the torus, where we take $G = \RR$, $\Gamma = \ZZ$ (see Section \ref{ssec:nil:sturm}). 

\subsubsection{Mal'cev basis}

Let $G$ be a connected and simply connected $s$-step nilpotent Lie group and let $\Gamma < G$ be a discrete cocompact subgroup. 
In this case, for $g \in G$ and $t \in \RR$, one can use the Lie algebra of $G$ to define $g^t \in G$.
 We let $\dim G$ denote the dimension of $G$ as a Lie group.  
A \emph{Mal'cev basis} of $G$ is a sequence $h_1,h_2, \dots, h_d \in \Gamma$ satisfying the following conditions. 
\begin{enumerate}
\item Every $g \in G$ has a unique representation $h_1^{t_1} h_2^{t_2} \cdots h_d^{t_d}$ with $t_i \in \RR$. 

\item There exists an increasing sequence of natural numbers $$1 = k_1 < k_2 < \dots < k_{s} = d+1$$ such that for each $1 \leq j \leq s$, 
the quotient $G_{j}/G_{j+1}$ is spanned by $h_{k_j}, \dots, h_{k_{j+1}-1}$. 
\end{enumerate}
Existence of such bases was established by Mal'cev  \cite{Malcev-1949,Malcev-1949-tran}.
Given a Mal'cev basis, we let $\tilde \tau \colon G \to \RR^d$ denote the coordinate map, characterised by the property that 
\begin{equation}\label{eq:Malcev-G}
	\tilde\tau\bra{ h_1^{t_1} h_2^{t_2} \cdots h_d^{t_d} } = (t_1,t_2,\dots, t_d)\,, \qquad t_i \in \RR\,.
\end{equation}
This induces also the coordinate map $\tau \colon G/\Gamma \to [0,1)^d$, similarly characterised by
\begin{equation}\label{eq:Malcev-X}
	\tau\bra{ h_1^{t_1} h_2^{t_2} \cdots h_d^{t_d} \Gamma} = (t_1,t_2,\dots, t_d)\,, \qquad t_i \in [0,1)\,.
\end{equation}

Thus, the nilmanifold $G/\Gamma$ can be identified with a cube $[0,1)^{\dim G}$ via 
\emph{Mal'cev coordinates} $\tau \colon G/\Gamma \to [0,1)^{\dim G}$. The coordinate map $\tau$ is a bijection; $\tau^{-1}$ is continuous, 
its restriction to $ (0,1)^{\dim G}$ is a diffeomorphism. The nilmanifold $G/\Gamma$ carries a natural probability measure, the Haar measure,  
which we denote by $\mu_{G/\Gamma}$.  

\subsubsection{Semialgebraic sets}\label{ssec:semialgebraic}

An algebraic variety in $\RR^d$ is a set defined by a finite number of polynomial equations. More generally, a semialgebraic set is a set defined by a finite number of polynomial equations and inequalities, or a finite union of sets of this form. 
 A map $f \colon \RR^d \to \RR$ is \emph{piecewise polynomial} if there exists a partition $\RR^d = S_1 \cup S_2 \cup \cdots \cup S_r$ 
 into semialgebraic pieces such that, for every $i$,  $1 \leq i \leq r$, the restriction $f|_{S_i}$ is a polynomial map. 
  A map $f \colon G/\Gamma \to \RR$ is piecewise polynomial if it takes the form $f = \tilde f \circ \tau$ for a piecewise polynomial 
  map $\tilde f \colon \RR^{\dim G} \to \RR$.
The notion of a piecewise polynomial map is independent of the choice of Mal'cev basis (see \cite[Sec.{} 0.18]{BergelsonLeibman-2007}). 
 
\subsubsection{Connectedness}  
Let us now return to the general case, where $G$ may be disconnected. 
If $G/\Gamma$ is connected, then it remains true that $G/\Gamma = G^{\circ}/\Gamma \cap G^{\circ}$, 
and the previous discussion applies verbatim (note, however, that not every translation $T_g$, with $g \in G$, 
can be represented as $T_h$, with $h \in G^{\circ}$). If $G/\Gamma$ is disconnected, then it can be decomposed as a finite union 
of connected components which again are nilmanifolds, and we can apply the discussion above to each component separately. 
A map $f \colon G/\Gamma \to \RR$ is piecewise polynomial if its restriction to each connected component of $G/\Gamma$ is piecewise polynomial.

Finally, we recall that a topological dynamical system $(X,T)$ is \emph{minimal} if there is no closed subset $Y \subset X$ with $T(Y) \subset Y$. 
We have now introduced all the terminology which is needed to state the $1$-dimensional case of the main result of \cite{BergelsonLeibman-2007}.

\begin{theorem}[{\cite[Thm.\ A]{BergelsonLeibman-2007}}]\label{thm:BL-mini}
	Any bounded {\gp} map $g \colon \ZZ \to \RR$ admits a representation $g(n) = f(T^n(x))$, where $(X,T)$ is a minimal nilsystem, 
	$f \colon X \to \RR$ is piecewise polynomial, and $x \in X$. Conversely, for any nilsystem $(X,T)$, any piecewise polynomial map
	$f \colon X \to \RR$, and any $x \in X$, the map $n \mapsto f(T^n(x))$ from $\ZZ$ to $\RR$ is a bounded {\gp} map.
\end{theorem}

For future reference, we record the following special case of Theorem \ref{thm:BL-mini} applicable to {\gpwords}.

\begin{theorem}\label{thm:BL-word}
	For any {\gpword} $\bb a$ over an alphabet $\Sigma$, there exists a minimal nilsystem $(X,T)$, a point $x \in X$, and a partition 
	$X = \bigcup_{i \in \Sigma} S_i$ into pairwise disjoint semialgebraic pieces, such that for each $n \in \NN$ and $i \in \Sigma$, 
	\begin{align}\label{eq:418:1}
		a_n = i && \text{ if and only if }&& T^n(x) \in S_i \,.
	\end{align}
Conversely, for any nilsystem $(X,T)$, any point $x \in X$ and any partition $X = \bigcup_{i \in \Sigma} S_i$ into pairwise disjoint semialgebraic pieces, 
\eqref{eq:418:1} defines a {\gpword}.
\end{theorem}

\begin{proof}
	Pick a representation $a_n = c(g(n))$  ($n \in \NN_0$), where $g$ is a finitely-valued {\gp} map and $c$ is a coding. Let $g(n) = f(T^n(x))$ be the representation of $g$ as in Theorem \ref{thm:BL-mini}. Then, for each $i \in \Sigma$, $S_i = f^{-1}(c^{-1}(i))$ is a semialgebraic set such that \eqref{eq:418:1} holds.
	The converse implication follows along similar lines. 
\end{proof}

Now, we give two emblematic examples of nilsystems and the corresponding {\gp} maps. While we do not use them elsewhere in the paper, we believe they help to illustrate Theorem \ref{thm:BL-mini}.

\subsection{One dimensional torus and Sturmian words}\label{ssec:nil:sturm}

As already mentioned in the introduction,  Sturmian words can be dynamically represented as codings of irrational translations on the one-dimensional torus. 
In this case, we simply take $G=\mathbb R$, $\Gamma=\mathbb Z$, and the nilrotation $T=T_\a$ is just an irrational translation on $\RR/\ZZ$.  More precisely, 
let $\bb a$ be the Sturmian word defined by 
\eqref{eq:802:1} (or, respectively, by \eqref{eq:802:2}).  
Let $T_\a \colon \RR/\ZZ \to \RR/\ZZ$ denote the translation by $\a$ on $\RR/\ZZ$, 
meaning that $T_\a(x) = x+\a$ ($x \in \RR/\ZZ$). Let $I \subset \RR/\ZZ$ be the interval $[0,\a)$ (resp.{} $(0,\a]$). 
Then, \eqref{eq:802:1} (resp. \eqref{eq:802:2}) is equivalent to:
\[
	a_n = 1_I(T_\a^n(\beta))
	= \begin{cases}
	1 &\text{ if } T_\a^n(\beta) \in I\, ,\\
	0 &\text{ otherwise.}
	\end{cases}	
\]

\subsection{Heisenberg group}
Another helpful example to keep in mind is the Heisenberg nilsystem. It is a standard example, and appears e.g.\ in 
\cite[Sec.\ 0.14]{BergelsonLeibman-2007},
\cite[Sec.\ 5]{GreenTao-2012}, 
\cite[Sec.\ 1]{GreenTao-2010},
\cite[p.\ 555]{GreenTao-2012-Mobius},
\cite{GreenTaoZiegler-2012} (as a running example).
 Pick any $\a,\b,\c \in \RR$, and set
\[
	G =
	\begin{bmatrix}
	1 & \RR & \RR \\
	0 & 1 & \RR \\
	0 & 0 & 1
	\end{bmatrix}, \;\;\;
	\Gamma =
	\begin{bmatrix}
	1 & \ZZ & \ZZ \\
	0 & 1 & \ZZ \\
	0 & 0 & 1
	\end{bmatrix}, \;\;\mbox{ and }\;\;
	h =	\begin{bmatrix}
	1 & \b & \c + \a\b/2\\
	0 & 1 & \a \\
	0 & 0 & 1
	\end{bmatrix} \,.
\]
One possible choice for Mal'cev coordinates is given by
\[
	\tau\bra{
	\begin{bmatrix}
	1 & y & z \\
	0 & 1 & x \\
	0 & 0 & 1
	\end{bmatrix}
	} = (x,y,z)\, .
\]
Then we can compute that
\begin{align*}
	h^n \Gamma &= 
	\begin{bmatrix}
	1 & n\b & n\c + n^2 \a\b/2 \\
	0 & 1 & n\a \\
	0 & 0 & 1
	\end{bmatrix}\Gamma\\& = 
	\begin{bmatrix}
	1 & \fp{n\b} & \fp{ n\b \fp{n\a} - n^2 \a\b/2 + \c n} \\
	0 & 1 & \fp{n\a} \\
	0 & 0 & 1
	\end{bmatrix}\Gamma \,.
\end{align*}
Hence, using the nilsystem $(G/\Gamma,T_h)$ and taking $f = \tau_3$ (\emph{i.e.}, the third entry of the coordinate map $\tau$), 
we obtain a representation of the bounded {\gp} map
\[ g(n) = \fp{ n\b \fp{n\a} - n^2 \a\b/2 + \c n}.\]
A slightly more complicated but similar construction involving matrices in dimension $4$ discussed in \cite[Sec.{} 0.14]{BergelsonLeibman-2007} 
provides a dynamical representation of the sequence $\fp{ n\b \fp{n\a}}$.

\section{Representations and constructions}\label{sec:constr}

In this section we discuss methods by which {\gpwords} can be represented and constructed.

\subsection{Representations} It is clear from the definition that {\gp} maps $\colon \NN_0 \to \RR$ are precisely those maps which can be expressed using (classical) polynomials, addition, multiplication and the floor function. For finitely-valued {\gp} maps and {\gpwords}, the situation becomes more complicated. As we will see in Section \ref{sec:comp}, it is not always possible to decide if a given {\gp} map $g \colon \NN_0 \to \RR$ is finitely-valued 
(and hence relevant to the study of {\gpwords}) or not. Thus, it is not always possible to see if a given formula represents a {\gpword}. On the other hand, for suitably constructed {\gp} maps, it is easy to see that they must be finitely-valued. 
For instance, if $g(n) = \ip{2\fp{h(n)}}$ for some {\gp} map $h \colon \NN_0 \to \RR$, then evidently $g(\NN_0) \subset \{0,1\}$. In general, one can always represent a finitely-valued {\gp} map in a form which makes it easy to estimate the cardinality of the image. We stress that Proposition \ref{lem:constr:reg-rep} provides a concrete way to generate 
all finitely valued \gp{} maps from $\NN_0$ to $\RR$, and hence all bracket words. 

\begin{proposition}\label{lem:constr:reg-rep}
	Let $g \colon \NN_0 \to \RR$ be a finitely-valued {\gp} map. Then $g$ can be written in the form $g(n) = f\bra{ \ip{N { \fp{h(n)}}}}$, where 
	$f$ and $h$ are {\gp} maps from $\NN_0$ to $\RR$ and $N = \abs{g(\NN_0)}$. Conversely, any map $g \colon \NN_0 \to \RR$ of the form 
	$g(n) = f\bra{ \ip{N { \fp{h(n)}}}}$, where $f$ and $h$ are {\gp} maps from $\NN_0$ to $\RR$ and $N \in \NN$, is a {\gp} map which
	takes at most $N$ distinct values.
\end{proposition}

In the course of the proof of Proposition \ref{lem:constr:reg-rep}, we will need the following simple fact, which clarifies the relation between {\gpwords} and finitely-valued {\gp} maps.

\begin{lemma}\label{lem:constr:gp-word-vs-fun}
	Let $f \colon \NN_0 \to \RR$ be a map taking finitely many values. Then the two following properties are equivalent. 
	\begin{enumerate}
	\item\label{it:685:a} The word $\bb f=\bra{f(n)}_{n=0}^\infty$ is a {\gpword}. 
	\item\label{it:685:b} The map $f \colon \NN_0 \to \RR$ is a {\gp} map.
	\end{enumerate}
\end{lemma}

\begin{proof}
	It is clear that \ref{it:685:b} implies \ref{it:685:a}: in Definition \ref{def:def:bra-word}, we can take $g=f$ and $c = \mathrm{id}$.  
	For the converse implication, suppose that $f(n) = c(g(n))$, where $g \colon \NN_0 \to \RR$ has finite image and 
	$c \colon g(\NN_0) \to \RR$. There exists a polynomial map $p \colon \RR \to \RR$ such that $p(x) = c(x)$ for all $x \in g(\NN_0)$. 
	Hence, $f = p \circ g$ is a {\gp} map.
\end{proof}

\begin{proof}[Proof of Proposition \ref{lem:constr:reg-rep}]
	Let $c \colon g(\NN_0) \to [N]$ be any bijective map. Then the word $\bb a = (c(g(n))_{n=0}^\infty$ is a {\gpword} taking values in $[N]$. Hence, by Lemma \ref{lem:constr:gp-word-vs-fun}, $c \circ g$ is a {\gp} map. Let $h$ be the {\gp} map defined by $h(n) = c(g(n))/N$, and let $f \colon \NN_0 \to \RR$ be a polynomial map such that $f(c(x)) = x$ for each $x \in g(\NN_0)$ (such $f$ exists by polynomial interpolation). It is straightforward to check that 
	\[ f(\ip{N\fp{h(n)}}) = f(c(g(n))) = g(n) \]
	for all $n \in \NN_0$, as needed.
	The converse direction holds trivially. 
\end{proof}

As a consequence of Proposition \ref{lem:constr:reg-rep}, for each {\gpword} we can construct a particularly convenient representation.

\begin{corollary}\label{cor:constr:gpword-nice}
	Let $\bb a$ be a {\gpword} defined over a finite alphabet $\Sigma$ and let $N = \abs{\Sigma}$. 
	Then there exist a {\gp} map $g \colon \NN_0 \to [N]$ and a map $c \colon [N] \to \Sigma$ such that $a_n = c(g(n))$ for all $n \in \NN_0$.
\end{corollary}

\begin{proof}
	Let $a_n = c'(g'(n))$, where $g' \colon \NN_0 \to \RR$ is a finitely-valued {\gp} map and $c' \colon g'(\NN_0) \to \Sigma$. Let $g'(n) = f(\ip{N\fp{h(n)}})$ be the representation of $g'$ from Proposition \ref{lem:constr:reg-rep}. It remains to define $g$ and $c$ by $g(n) = \ip{N\fp{h(n)}}$ and $c(x) = c'(f(x))$.
\end{proof}

\subsection{Constructions}

Next, we discuss a basic tool, which can be used to construct potentially interesting examples of {\gpwords}. In what follows, we will use the Iverson bracket notation, where $\braif{g \in I}$ denotes the word over the alphabet $\{0,1\}$ which records the positions $n$ such that $g(n) \in I$ (see Section \ref{sec:Notation} for exact definition). 

\begin{proposition}\label{prop:cl:g-in-I}
	Let $g \colon \NN_0 \to \RR$ be a {\gp} map and $I \subset \RR$ be an interval (possibly infinite or degenerate). Then $\braif{g \in I}$ is a {\gpword}.
\end{proposition}

\begin{proof}
The case where $I$ is bounded is covered by {\cite[Lemma 1.2]{ByszewskiKonieczny-2018-TAMS}}, while the case where $I$ is unbounded follows from \cite[Lemma B.3]{Konieczny-2021-JLM}. 
\end{proof} 

\begin{remark}
We point out that the analogous result with $\NN_0$ replaced by $\ZZ$ is true 
for bounded $I$ and false for unbounded $I$. For instance, $1_{\NN}(n) = \braif{n \in (0,\infty)}$ is not a {\gp} map on $\ZZ$ (c.f.\ \cite[Ex.\ B.2]{Konieczny-2021-JLM}). This is one of the reasons why we focus on one-sided {\gpwords} $(a_n)_{n=0}^\infty$, 
rather than on their two-sided analogues $(a_n)_{n \in \ZZ}$.
\end{remark} 

We will use later the following slight refinement of Proposition \ref{prop:cl:g-in-I} in the case $I = \{0\}$. 

\begin{lemma}\label{lem:cl:[g=0]}
	Let $g \colon \NN_0 \to \RR$ be a {\gp} map with coefficients in some field $K$ with $\QQ \subsetneq K \subseteq \RR$. Then $\braif{g = 0}$ is a {\gpword} with coefficients in $K$.
\end{lemma}

\begin{proof}
	Pick any $\a \in K \setminus \QQ$. Note that the only solution to $\fpnormal{x} = \fpnormal{\a x} = 0$ in $\RR$ is $x = 0$. Hence,
	\[
		\ip{1 - \frac{1}{2} \fpnormal{x} - \frac{1}{2} \fpnormal{\a x}  } =  \braif{ x = 0}\,, \qquad (x \in \RR)\,.
	\]
	It follows that 
	\[
		\braif{g(n) = 0} =  \ip{1 - \frac{1}{2} \fpnormal{g(n)} - \frac{1}{2} \fpnormal{\a g(n)}  }\,, \qquad n \in \NN_0\,. \qedhere
	\]
\end{proof}

One of the main reasons for interest in Lemma \ref{lem:cl:[g=0]} is that it gives an alternative definition of {\gpwords}, 
phrased in terms of fibres.

\begin{corollary}\label{cor:constr:fibre}
	Let $\bb a$ be an infinite word defined over a finite alphabet $\Sigma$. Then the two following properties are equivalent.
	\begin{enumerate}
	\item The word $\bb a$ is a {\gpword}. 
	\item For every $x \in \Sigma$, the fibre $\set{n \in \NN_0}{a_n = x}$ is a {\gp} subset of $\NN_0$.
	\end{enumerate}
	In particular, for a set $E \subset \NN_0$, $E$ is a {\gp} set if and only if $\bb{1}_E$ is a {\gpword}.
\end{corollary}

\begin{proof}
Let us assume that  $\bb a$ is a {\gpword} defined over $\Sigma$ and let $x\in \Sigma$. 
Let $a_n=c(g(n))$ be a representation of $\bb a$, where $g: \NN_0\to \RR$ is a finitely valued \gp{} map and $c$ is a map from $g(\NN_0)$ to $\Sigma$. 
Then $h= \prod_{a\in c^{-1}(x)}(g-a)$ is a \gp{} map and it follows from Lemma \ref{lem:cl:[g=0]} that $\braif{h(n) = 0}=\{n\in \NN_0 \mid a_n=x\}$ is a \gp{} subset of $\NN_0$.

Conversely, let us assume that, for every $x \in \Sigma$, the fibre $F_x=\set{n \in \NN_0}{a_n = x}$ is a {\gp} subset of $\NN_0$. Hence $\bb{1}_{F_x}$ is a \gp{} map. 
Let us assume that $\vert \Sigma\vert=N$ and let $x_1,\ldots,x_N$ denote an enumeration of the elements of $\Sigma$. Set $h= \sum_{i=1}^N i \bb{1}_{F_{x_i}}$ and let $c$ be the map defined by $c(i)=x_i$, $1\leq i\leq N$. Then $h$ is a \gp{} map and thus $a_n=c(h(n))$ is a {\gpword}. 
\end{proof}

With some basic algebraic manipulations, one can extend Proposition \ref{prop:cl:g-in-I} to apparently more complicated conditions, 
as shown by the following example.

\begin{example}
	Let $g$ and $h$ be {\gp} maps from $\NN_0$ to $\RR$ and assume that $h(n) > 0$ for all $n \in \NN_0$. Then $\braif{g < 1/h} = \braif{gh < 1}$ is a {\gpword}. More generally, for each rational exponent $\lambda = p/q \in \QQ_{>0}$, also 
	$\braif{g < 1/h^{\lambda}} = \braif{g^q h^p < 1}$ is a {\gpword}.
	As an explicit application, for every pair $(\a,c) \in \RR^2$ and $\lambda \in \QQ$, with $c$ and $\lambda$ positive, 
	the formula $a_n = \braif{ \fpa{\a n} < c/n^{\lambda} }$ defines a {\gpword} $\bb a$ which detects denominators of good rational approximations to $\a$.
\end{example}

\section{Closure properties}\label{sec:cl}

We will now discuss ways in which known instances of {\gpwords} can be used to construct new ones. 
Compared to the earlier section, the results discussed here have a more computational flavour. 
For instance, we point out that, with the sole exception of Proposition \ref{prop:cl:orb-cl}, all results in this section 
are analogues of standard results about automatic sequences, in the sense that they remain true if the term ``{\gpword}'' is replaced with ``$k$-automatic sequence'' and the term ``{\gp} set'' is replaced with ``$k$-automatic set'' for fixed $k \geq 2$ (see \cite[Sec. 5]{AlloucheShallit-book}).

\subsection{Codings and products} 

It is an almost immediate consequence of Definition \ref{def:def:bra-word} that {\gpwords} are preserved under coding.

\begin{lemma}\label{lem:cl:code}
	Let $\bb a$ be a {\gpword} defined over a finite alphabet $\Sigma$ and let $\varphi \colon \Sigma \to \Pi$ be a map to some other finite alphabet $\Pi$. 
	Then $\bra{\varphi(a_n)}_{n=0}^\infty$ is a {\gpword} over $\Pi$.
\end{lemma}

\begin{proof}
	If $a_n = c( g(n))$  is the representation of $\bb a$ as in Definition \ref{def:def:bra-word}, 
	then $\varphi(a_n)= c'( g(n))$ where $c' = \varphi \circ c$.
\end{proof}

Next, we note that the direct product of two {\gpwords} is again a {\gpword}.

\begin{proposition}\label{prop:cl:prod} 
Let $\bb a$ and $\bb a'$ be {\gpwords} respectively defined over some finite alphabets $\Sigma$ and $\Sigma'$. 
Then $\bb a \times \bb a' = \bra{(a_n,a_n')}_{n=0}^\infty$ is a {\gpword} over $\Sigma \times \Sigma'$.
\end{proposition}

\begin{proof}
	Recall that, by definition of a {\gpword}, $\bb a$ has a representation $a_n = c( g(n))$, where $g \colon \NN_0 \to A$ is a 
	{\gp} map taking values in some finite set $A \subset \RR$, and $c \colon A \to \Sigma$ is an arbitrary map. Let  $a_n' = c'( g'(n))$ 
	be an analogous representation of $\bb a'$.
	
	Replacing $g$ with $C g$ for a sufficiently large positive real number $C > 0$ (and modifying $A$ and $c$ accordingly), 
	we may assume that the only solutions to $x+x' = y+y'$ with $(x,y) \in A^2$ and $(x',y)' \in A'^2$ are the trivial ones: $x = x'$ and $y = y'$. 
	Let $B$ denote the sumset $A+A' = \set{x+x'}{x \in A,\ x' \in A'}$, let $h$ denote the {\gp} map  $g+g' \colon \NN_0 \to B$, and let 
	$d \colon B \to \Sigma \times \Sigma'$ denote the unique map such that $d(x+x') = \bra{ c(x),c'(x') }$ for all $x \in A$ and $x' \in A'$. 
	Then $(a_n, a'_n) = d(h(n))$ for all $n \in \NN_0$.
\end{proof}

\begin{remark}
	It follows that replacing, in Definition \ref{def:def:bra-word},  {\gp} maps $\NN_0 \to \RR$ with {\gp} maps $\NN_0 \to \RR^d$ 
	for arbitrary $d \in \NN$, leaves unchanged the set of words so defined. 
\end{remark}

In practice, Proposition \ref{prop:cl:prod} is mostly used via the following corollary.

\begin{corollary}\label{cor:cl:prod}
Let $\bb a$ and $\bb a'$ be {\gpwords} respectively defined over some finite alphabets $\Sigma$ and $\Sigma'$, 
and let $f \colon \Sigma \times \Sigma' \to \Pi$ be a map to some other finite alphabet $\Pi$. 
Then $\bra{f(a_n,a_n')}_{n=0}^\infty$ is a {\gpword} over $\Pi$.
\end{corollary}

\begin{proof}
	This follows directly from Proposition \ref{prop:cl:prod} and Lemma \ref{lem:cl:code}.
\end{proof}

\begin{remark}
	In particular, for any ring $R$, {\gpwords} taking values in $R$, equipped with coordinatewise addition and multiplication, form a ring. 
\end{remark}

As a consequence of the two previous results, we see that {\gpwords} can be defined in a ``case-by-case'' manner.

\begin{proposition}\label{prop:cl:cases}
	Let $\Sigma$ be a finite alphabet, $\NN_0 = \bigcup_{i=1}^r S_i$ be a partition of $\NN_0$ into pairwise disjoint {\gp} subsets, 
	and $\bb a^{(i)}$, $1 \leq i \leq r$, be {\gpwords} over $\Sigma$. Let $\bb a$ be defined by $a_n = a^{(i)}_n$ if $n \in S_i$.  
	Then $\bb a$ is a {\gpword}.
\end{proposition}

\begin{proof}
We first infer from a recursive use of Proposition \ref{prop:cl:prod} that the word $$(\bb{1}_{S_1},\ldots,\bb{1}_{S_r},\bb a^{(1)},\ldots,\bb a^{(r)})$$ 
defined over the alphabet 
$\{0,1\}^r\times \Sigma^r$ is a bracket word. Now, letting $f$ be any function from $\{0,1\}^r\times \Sigma^r$ to $\Sigma$ such that 
$$
f(\varepsilon_1,\ldots,\varepsilon_r,x_1,\ldots,x_r)=x_i \mbox{ if } \varepsilon_i=1 \mbox{ and } \varepsilon_j=0 \mbox { when } j\not=i \,,
$$
we infer from Corollary \ref{cor:cl:prod} that  $f(\bb{1}_{S_1},\ldots,\bb{1}_{S_r},\bb a^{(1)},\ldots,\bb a^{(r)})$ is a bracket word. 
It remains to see that  $\bb a=f(\bb{1}_{S_1},\ldots,\bb{1}_{S_r},\bb a^{(1)},\ldots,\bb a^{(r)})$.
\end{proof}

\begin{example}
	Let $\bb a$ be the infinite word over the alphabet $\{-2,-1,0,\dots,10\}$ given by
	\[
		a_n = 
		\begin{cases}
			\ip{n\fp{\sqrt{2}n}} &\text{if } {n\fp{\sqrt{2}n} } \leq 10\,,\\
			-1 &\text{if } {n\fp{\sqrt{2}n} } > 10 \text{ and } n^2 \fp{\sqrt{2}n\ip{\sqrt{3}n}} - n \fp{\sqrt{5}n} + 7 > 0\,,\\
			-2 &\text{if } {n\fp{\sqrt{2}n} } > 10 \text{ and } n^2 \fp{\sqrt{2}n\ip{\sqrt{3}n}} - n \fp{\sqrt{5}n} + 7 \leq 0\,.
		\end{cases}
	\]
	Then $\bb a$ is a {\gpword}.
\end{example}

Since each eventually constant sequence is a {\gpword}, it follows that {\gpwords} are also closed under finite modifications.
\begin{corollary}
	Let $\bb a,\bb a'$ be infinite words over a finite alphabet. If $\bb a$ is a {\gpword} and $a'_n = a_n$ for all but finitely many $n \in \NN_0$ then $\bb a'$ is a {\gpword}.
\end{corollary}
\begin{proof}
	The result follows directly from Proposition \ref{prop:cl:cases} and the fact that all finite subsets of $\NN_0$ are {\gp}.
\end{proof}

\subsection{Rearrangements and morphisms}

Many natural operations on infinite words can be described in terms of extracting or inserting entries in a regular manner. We record a simple observation, which can be used to find examples of operations of the aforementioned type which preserve {\gpwords}.

\begin{lemma}\label{lem:cl:subseq}
	Let $\bb a$ be a {\gpword} and let $h \colon \NN_0 \to \NN_0$ be a {\gp} map. Then $\bra{ a_{h(n)}}_{n=0}^\infty$ is a {\gpword}.
\end{lemma}

\begin{proof}
	If $a_n = c( g(n))$ is the representation of $\bb a$ as in Definition \ref{def:def:bra-word}, then $a_{h(n)} = c( g'(n))$, 
	where $g' = g \circ h$ is a {\gp} map (see Remark \ref{rem:composition}).
\end{proof}

Below, we list some applications of this result. 
Recall that, for $A \in \NN$, the map $\NN_0 \to \NN_0$ given by $n \mapsto n \bmod A = A\fp{n/A}$ is a {\gp} map.

\begin{corollary}\label{cor:cl:rearrange}
	Let $\bb a$ be a {\gpword} defined over a finite alphabet $\Sigma$, let $A \in \NN$, $B \in \NN_0$, let $\diamondsuit$ be a symbol not belonging to $\Sigma$, and let $\pi \colon [A] \to [A]$ be a map. 
	Then the following infinite words are also {\gpwords}. 
	\begin{itemize}
	\item[{\rm (i)}] $\bra{ a_{An+B}}_{n=0}^\infty$.
	\item[{\rm (ii)}] $\bb a'$ defined over $\Sigma \cup \{\diamondsuit\}$  by $a'_n = a_{n/A}$ if $A \mid n$ and $a'_n = \diamondsuit$ otherwise.
	\item[{\rm (iii)}] $\bra{ a_{ \floor{n/A} + \pi(n \bmod A) }}_{n=0}^\infty$.
	\end{itemize}
\end{corollary}

\begin{proof}
	The first item follows directly from Lemma \ref{lem:cl:subseq}. The second and third items follow from Lemma \ref{lem:cl:subseq} 
	and Proposition \ref{prop:cl:cases}.
\end{proof}

\begin{lemma}\label{lem:cl:subs}
	Let $\Sigma$ and $\Pi$ be two alphabets. Let $\bb a$ be a {\gpword} over $\Sigma$ and $\sigma$ be a morphism of constant length from 
	$\Sigma^*$ to $\Pi^*$. Then $\sigma(\bb a)$ is a {\gpword}.
\end{lemma}

\begin{proof}
	Let us assume that $\sigma$ has constant length $k$, and set $\bb a' = \sigma(\bb a)$. 
	For each $i \in [k]$ and $n \in \NN_0$, we have $a'_{kn+i} = \sigma(a_n)_i$, where we let $ \sigma(a_n)_i$ denote the $i$th letter occurring in  
	$\sigma(a_n)$. 
	It remains to apply Proposition \ref{prop:cl:cases} and Lemmas \ref{lem:cl:subseq} and \ref{lem:cl:code}.
\end{proof}

By similar techniques, we can show a result in the reverse direction: {\gpwords} are preserved under grouping blocks of constant length, 
or, in other words, if the image of a word $\bb a$ by an injective morphism of constant length is a {\gpword}, then $\bb a$ itself is also a {\gpword}.

\begin{lemma}
	Let $\bb a$ be an infinite word defined over a finite alphabet $\Sigma$ and let $k \in \NN$. 
	Consider the infinite word ${\bb a}'$ over $\Sigma^k$ given by $ a_n' = a_{kn} a_{kn+1} \cdots a_{kn+k-1}$. 
	Then ${\bb a}'$ is  a {\gpword} if and only if $\bb a$ is a {\gpword}.
\end{lemma}

\begin{proof}
	The proof is similar to the one of Lemma \ref{lem:cl:subs}.
\end{proof}

\subsection{Orbit closure}

Given an infinite word $\bb a$ defined over a finite alphabet $\Sigma$, we let $\cO(\bb a)$ denote the orbit of $\bb a$ under the shift, 
that is, the set of all infinite words $ \bb a' $ given by $a'_n = a_{n+m}$ for some $m \in \NN_0$.

\begin{proposition}\label{prop:cl:orb-cl-0}
	Let $\bb a$ be a {\gpword} and let $\bb a' \in  \cO(\bb a)$. Then $\bb a'$ is a {\gpword}. 
\end{proposition}

\begin{proof}
This follows from (i) of Corollary \ref{cor:cl:rearrange} in the special case where $A=1$. 
\end{proof}

More generally, we can consider the orbit closure $\bar{\cO(\bb a)}$, that is, the closure of $\cO(\bb a)$ with respect to the product topology 
on $\Sigma^{\infty}$. Explicitly, $\bb a'$ belongs to $\bar{\cO(\bb a)}$ if, for every $N \in \NN$, there exists $m$ such that $a'_n = a_{n+m}$ for all $n \in [N]$. 

\begin{proposition}\label{prop:cl:orb-cl}
	Let $\bb a$ be a {\gpword} and let $\bb a' \in \bar{\cO(\bb a)}$. Then $\bb a'$ is a {\gpword}. 
\end{proposition}

\begin{proof}
This is a direct consequence of Lemmas \ref{lem:nil:shift} and \ref{lem:cl:limit} given in Appendix \ref{app:nil}.
\end{proof}
 
\begin{example}
	Let $\bb a$ be the {\gpword} given by $a_n = \braif{ \fp{ \sqrt{2} n\fp{\sqrt{3} n}} < 1/4}$. 
	Then each $\bb a' \in \cO(\bb a)$ can be written in the form 
	\[ a'_n = \braif{  \fp{ (\sqrt{2} n+\a)\fp{\sqrt{3} n+\b} + \gamma n + \delta } < 1/4},\] 
	where $\a,\b,\gamma,\delta \in [0,1)$. Conversely, $\bar{\cO(\bb a)}$ contains all sequences $\bb a'$ of the aforementioned form,
	as well as all sequences obtained from them by replacing any instances of $\fp{x}$ with $\fp{x}' = 1-\fp{-x}$, 
	or by replacing the strict inequality $<$ with $\leq$.	
\end{example}

\subsection{{\gp} sets}\label{ssec:gp-sets}

The closure properties of {\gpwords} directly translate into closure properties of {\gp} subsets of $\NN_0$ thanks to Corollary \ref{cor:constr:fibre}. 
In fact, we have already used this connection at several places. For ease of reference, we gather these properties here.

\begin{proposition}\label{prop:cl:sets}
	The family of {\gp} subsets of $\NN_0$ is a field of sets. In other words, it contains the empty set and if $E,F \subset \NN_0$ are {\gp} sets 
	then so are $\NN_0 \setminus E$, $E \cup F$, and $E \cap F$.
\end{proposition}

\begin{proof}
The result follows from Proposition \ref{prop:cl:cases}. 
\end{proof}

\begin{proposition}\label{prop:cl:sets-2}
	Let $E \subset \NN_0$ be a {\gp} set. Then the following sets are also {\gp} sets. 
	\begin{itemize}
	\item $m \cdot E = \set{nm}{n \in E}$, where $m \in \NN_0$. 
	\item $E / m = \set{n \in \NN_0}{ nm \in E}$, where $m \in \NN$.
	\item $E - m = \set{n \in \NN_0}{ n+m \in E}$, where  $m \in \ZZ$. 	
	\end{itemize}
\end{proposition}
\begin{proof}
	This is a consequence of Corollary \ref{cor:cl:rearrange}.
\end{proof}

\section{Frequencies and recurrence}\label{sec:facts}

In this section, we collect some facts concerning frequencies with which symbols and subwords appear in {\gpwords}, as well as additive structure of the set of positions at which a given symbol appears in a {\gpword}. Most of these facts are consequences of results concerning uniform 
distribution and recurrence of orbits of point in nilsystems, and hence are related to the dynamical characterization of bracket 
words discussed in Section \ref{sec:nil}. We point out that, similarly, many standard properties of Sturmian words can be inferred from dynamical properties of circle rotations.

\subsection{Uniform frequency}

Let $\Sigma$ be a finite alphabet. Given an infinite word $\bb a\in\Sigma^{\infty}$ and a letter $x \in \Sigma$, the (asymptotic)
\emph{frequency} of $x$ in $\bb a$ is defined as
\[    \freq(\bb a,x) = \lim_{N \to \infty} \abs{\set{n \in [N]}{a_n= x}}/N\,,\]
provided that the limit exists. 
More generally, given a finite word $\fword{w} \in \Sigma^\ell$, the frequency of $\fword{w}$ in $\bb a$ is defined as
\[
        \freq(\bb a, \fword{w}) = \lim_{N \to \infty} \abs{\set{n \in
[N]}{\bb a_{[n,n+\ell)} = \fword{w}}}/N\,,
\]
provided that the limit exists. 
Here and elsewhere, $\bb a_{[n,n+\ell)}$ denotes the finite subword
$a_{n}a_{n+1}\cdots a_{n+\ell-1}$. We define also the \emph{recurrence function} of $\fword{w}$ in 
$\bb a$ as the least length of a segment of $\bb a$ that is guaranteed to contain an occurrence of $\fword{w}$:
\[
        \mathrm{rec}(\bb a, \fword{w}) = \inf \set{ r \in \NN }{
        \parbox{17em}{
        {for each } $m$ { there exists } $n \in [m,m+r)$ { such that } $\bb
a_{[n,n+\ell)} = \fword{w}$ }
        } \in \NN \cup \{\infty\} \,.
\]
For {\gpwords}, we have a strong result about the existence of 
frequencies, which is a rephrasing of \cite[Thm.{} B]{BergelsonLeibman-2007}.

\begin{theorem}\label{prop:facts:freq-exists}
Let $\bb a$ be a {\gpword} defined over a finite alphabet $\Sigma$ and $\ell$ 
be a positive integer. Then for each $\fword{w} \in \Sigma^\ell$, the frequency $\freq(\bb
a, \fword{w})$ exists and, moreover, one has
        \[
                \abs{\set{n \in [M,M+N)}{\bb a_{[n,n+\ell)} = \fword{w}}}/N \to
\freq(\bb a, \fword{w})
        \]
uniformly in $M$ as $N \to \infty$.
\end{theorem}

\begin{proof}
By definition, there exist a finitely valued \gp{} map $g:\NN_0 \to
\RR$ and a coding $c:g(\NN_0)\to \Sigma$ such that $a_n=c(g(n))$. 
We deduce from \cite[Thm.{} B]{BergelsonLeibman-2007} that the desired 
conclusion holds for the sequence $g(n)$. 
The same conclusion remains true after applying a coding. 
\end{proof}

\begin{corollary}\label{cor:facts:freq-exists}
For $\bb a$ and $\fword{w}$ as above, either $\freq(\bb a, \fword{w})
= 0$ or $\mathrm{rec}(\bb a, \fword{w}) < \infty$.
\end{corollary}

\begin{remark}
It can happen that $\freq(\bb a, \fword{w}) = 0$ but still 
$\fword{w}$ appears in $\bb a$ infinitely often, see e.g.\ Example \ref{ex:ex:rec}. Hence, {\gpwords} are 
not guaranteed to be uniformly recurrent\footnote{A sequence $\bb a$ is uniformly
recurrent if, for every finite word $\fword{w}$ which appears in $\bb a$,
        there exists $\ell \in \NN$ such that $w$ appears in the
length-$\ell$ segment $\bb a|_{[n,n+\ell)}$ for all $n$.}.
\end{remark}

The same result can be stated in terms of densities of {\gp} sets. 
Recall that the  \emph{asymptotic density} of a set $E \subset \NN_0$ 
is defined by 
\[ d(E) = \lim_{N\to\infty} \abs{E\cap [N]}/N\,,
\]
provided that the limit exists. In general, we let $\bar{d}(E)$ and 
$\underline{d}(E)$ denote the upper and lower asymptotic densities, 
obtained by replacing $\lim$ with $\limsup$ and $\liminf$ 
respectively. Additionally, we define the 
\emph{upper and lower uniform (Banach) densities} by 
$$d^*(E) = \limsup_{N\to\infty} \sup_{M} \abs{E\cap [M,M+N)}/N$$
and
$$
d_*(E) = \liminf_{N\to\infty} \inf_{M} \abs{E\cap [M,M+N)}/N \,.
$$
In general, we have the chain of  inequalities:
\(
        d_*(E) \leq \underline d(E) \leq \overline d(E) \leq d^*(E).
\)

\begin{theorem}\label{thm:facts:dens-exists}
Let $E$ be a {\gp} subset of $\NN_0$. Then $d(E)$ exists. Moreover, 
$d^*(E) = d_*(E) = d(E)$. 
\end{theorem}

\subsection{Distribution along subsequences}\label{ssec:facts:equidist}

\renewcommand{\t}{t}

Many of the desirable properties of bracket words are preserved under passing to subsequences. As a first instance of this principle, we
consider the simple case of polynomial subsequences.

\begin{proposition}\label{cor:facts:poly-freq-exists}
Let $\bb a$ be a {\gpword} defined over a finite alphabet $\Sigma$, $p: \NN_0 \to \NN_0$ be a polynomial, and $\ell$ be a positive integer. 
Then for each $\fword{w} \in \Sigma^\ell$, 
the frequency $\freq \bra{ \bra{a_{p(n)}}_{n=0}^\infty, \fword{w}}$ exists.
\end{proposition}

\begin{proof}
Let $a_n=c(g(n))$ be a representation of $\bb a$, where $g: \NN_0 \to \RR$ is a \gp{} map. Then $g\cdot p : \NN_0 \to \NN_0$ is also a \gp{} map for \gp{} maps 
are closed under composition, and $a_{p(n)}=c(g\circ p(n))$. The results follows now from Theorem \ref{prop:facts:freq-exists}. 
\end{proof}

We point out that even in the simplest cases, passing to a subsequence
can alter frequencies with which symbols occur. For instance, if $\bb
a = 101010\cdots$ then $\freq(\bb a, 1) = 1/2$ but
$\freq((a_{2n})_{n=0}^\infty, 1) = 1$.

Next, let us consider distribution along the primes. In
\cite{GreenTao-2012-Mobius}, Green and Tao obtained quantitative
estimates on correlations between the M\"{o}bius function and
nilsequences, which have important consequences for the question at
hand, (cf.{} \cite[Sec.{} 5]{GreenTao-2012-Mobius}). Extending the
techniques developed by Green and Tao, and the connection between
{\gp} maps and nilsystems, Bergelson, H{\aa}land-Knutson and Son
\cite{BergelsonHalandSon-2020} showed that bounded {\gp} maps have
asymptotic distribution along the primes. The following result is
obtained by specialising \cite[Theorem 5.1]{BergelsonHalandSon-2020}
to finitely-valued sequences. Below, we let $p_n$ denote the $n$-th
prime. We also mention related work of Eisner \cite{Eisner-2019} concerning
convergence of ergodic averages along primes in nilsystems. In the case of Sturmian words, this result is essentially due to Davenport \cite{Davenport-1937}. 

\begin{theorem}\label{thm:facts:prime-freq-exists}
        Let $\bb a$ be a {\gpword} defined over a finite alphabet $\Sigma$. 
        Then for each $x \in \Sigma$, the frequency
        $\freq \bra{ \bra{a_{p_n}}_{n=0}^\infty, x}$ exists.
\end{theorem}
\begin{proof}
	This follows from {\cite[Thm.\ 5.1]{BergelsonHalandSon-2020}} using the same argument as in the proof of Theorem \ref{prop:facts:freq-exists}.
\end{proof}

Considering $\bb a = 101010\cdots$ again, we see that the frequencies $\freq \bra{ \bra{a_{p_n}}_{n=0}^\infty, x}$ need not be equal to $\freq \bra{ \bb a, x}$.  
One can also inquire if it possible to generalise Theorem \ref{thm:facts:prime-freq-exists} from frequencies of symbols $x \in \Sigma$ to frequencies of 
words $w \in \Sigma^\ell$, $\ell \geq 2$. While it seems plausible that this generalisation is true, it is out of reach of the current techniques. 
Indeed, already in the case where $\bb a = (n \bmod q)_{n=0}^\infty$ is a periodic word with period $q \geq 3$, such a generalisation would require 
us to understand the asymptotic behaviour of
\begin{equation}\label{eq:358:1}
	\frac{1}{N} \abs{ \set{n < N}{ p_n \equiv w_1, p_{n+1} \equiv w_2,  \dots, p_{n+\ell-1} = w_{n+\ell-1} \bmod{q}}}.
\end{equation}
It is conjectured that the expression in \eqref{eq:358:1} converges to $1/\varphi(q)^\ell$, where $\varphi$ denotes the totient function; in fact, the Main Conjecture in \cite{LemkeSoundararajan-2016} gives a more precise asymptotic expression. However, such estimate remains unknown for any $\ell \geq 2$, as discussed in  \cite{LemkeSoundararajan-2016}.

Lastly, we consider a class of subsequences which preserve the frequencies of symbols. 
As already alluded to earlier, Theorem \ref{thm:BL-mini} leads to a close connection between 
equidistribution results for nilsystems and statements about frequencies of symbols in {\gpwords}. Recall 
that each nilmanifold $X$ comes equipped with the  Haar measure $\mu_{X}$. A sequence $(x_n)_{n=0}^\infty$ is 
\emph{equidistributed} in $X$ if for each continuous map 
$f \colon X \to \RR$ we have 
\begin{equation}\label{eq: equidistributed}
        \lim_{N \to \infty} \frac{1}{N} \sum_{n = 0}^{N-1} f(x_n) = \int_{X}
f d\mu_{X}\,.
\end{equation}

\begin{definition}\label{def:good-equidist}
        Let $\t \colon \NN_0 \to \NN_0$. We say that $t$ is \emph{good for
equidistribution in nilsystems} if, 
        for all minimal nilsystems $(X,T)$ and all $x \in X$, 
        the sequence $(T^{\t(n)}(x))_{n=0}^\infty$ is equidistributed in $X$.
\end{definition}

\begin{proposition}\label{prop:nil:equidistribution}
        Let $\bb a$ be a {\gpword} defined over a finite alphabet $\Sigma$, and let 
$t \colon \NN_0 \to \NN_0$ be good for equidistribution in nilsystems.
        Then the frequencies of letters in $\bb a$ along the 
subsequence $t(n)$ exist and remain the same as in $\bb a$: 
        \[
        \freq\bra{ \bra{a_{t(n)}}_{n=0}^\infty, x} = \freq(\bb a,x) \;\;\quad
\forall x \in \Sigma\,.
\]
\end{proposition}

\begin{proof}
This is a consequence of the representation of $\bb a$ coming from Theorem \ref{thm:BL-word}. 
Indeed, though the indicator functions ${\bb 1}_{S_i}$ of the 
semialgebraic sets $S_i$ from Theorem \ref{thm:BL-word} are not 
continuous, they can be efficiently approximated by  continuous functions to which Equality \eqref{eq: equidistributed} can be applied.
\end{proof}

Hopefully, the following result of Frantzikinakis
\cite{Frantzikinakis-2009} provides a plentiful source of sequences
that satisfy Definition
\ref{def:good-equidist} and make Proposition 
\ref{prop:nil:equidistribution} relevant.

\begin{theorem}[{\cite[Thm.\ 1.1]{Frantzikinakis-2009}}]\label{thm:Fran}
        Let us assume that $f \colon \RR_{>0} \to \RR_{>0}$ satisfies the
following properties.

\begin{enumerate}

\item\label{it:574:1} The function $f$ belongs to some Hardy field.

\item\label{it:574:2} There exists $C > 0$ such that ${f(x)} = O(x^C)$ as $x \to \infty$.

\item\label{it:574:3} For every $p \in \ZZ[x]$ and $c \in \RR$, we have
\(\displaystyle \lim_{x \to \infty} \frac{\abs{ f(x) - c p(x)}}{\log
x} = \infty.\)
\end{enumerate}
Then $\ip{f}$ is good for equidistribution in nilsystems.
\end{theorem}

For a definition of a Hardy field, we refer the reader to 
\cite{Frantzikinakis-2009}. Here, we just point out that one example of a Hardy field
is the logarithmic-exponential functions, that is, real-valued 
functions defined on some interval $[c,\infty)$, $c \in \RR$, that can be constructed from polynomials using addition, multiplication, and the functions 
$\exp$ and $\log$. Hence, the following formulae define sequences that 
are good for equidistribution in nilsystems: $\ip{n^{3/2}}$, $\ip{n^{2} 
\log^2 n}$, $\ip{n^{2} + \log^2 n}$, and $\ip{n^{2} + \sqrt{2}n}$. 
Conversely, because of item \ref{it:574:3}, there is no guarantee that 
the same should apply to $\ip{n^2 + \log n}$ or $\ip{\sqrt{2}n}$. 
In fact, one can compute that 
\[
        \frac{1}{\sqrt{2}} \ip{\sqrt{2}n} = n + \frac{\fp{\sqrt{2}n} }{\sqrt{2}}
\equiv \frac{\fp{\sqrt{2}n} }{\sqrt{2}} \bmod{1},
\]
so $\ip{\sqrt{2}n}$ is not good for equidistribution for the circle 
rotation by ${1}/{\sqrt{2}}$. 
Extensions of some results in \cite{Frantzikinakis-2009} were 
recently obtained by Richter \cite{Richter-2020}.

It is not possible to directly generalise Proposition \ref{prop:nil:equidistribution}  to frequencies of words instead of symbols. 
For instance, the sequence $t(n) = \ip{\sqrt{n}}$ is good for equidistribution in nilsystems and constant on each interval $[N^2,(N+1)^2)$, $N \in \NN$. 
Hence, already  for the periodic word $\bb a = 101010\cdots$ and $w \in \{0,1\}^2$ we have
	\[
		\freq\bra{ (a_{t(n)})_{n=0}^\infty, w } = 
		\begin{cases}
			1/2 &\text{if } w \in \{00,11\},\\
			0 & \text{if } w \in \{01,10\}.
		\end{cases}
	\] 
	Thus, the frequencies of words in $(a_{t(n)})_{n=0}^\infty$ bear little resemblance to the frequencies of symbols and words in $\bb a$.	
	
	On the other hand, under additional growth conditions, one can obtain positive results. The key technical component is the following theorem of Bergelson, Moreira and Richter. Below, we let $f^{(j)}$ denote the $j$-th derivative of $f$.
	
	\begin{theorem}[{\cite[Thm. 5.6, special case]{BergelsonRichterMoreira-2020}}]\label{thm:BMR}
		Let $(X,T)$ be a minimal nilsystem, let $z \in X$, and let $f_1,f_2,\dots,f_k$ be functions belonging to the same Hardy field. Suppose further that the following two conditions hold. 
	\begin{enumerate}
		\item\label{it:BRM:1} For each $h \in \spanlin\set{ f_i^{(j)}}{1 \leq i \leq k, j \geq 0}$ and each $p(x) \in \RR[x]$, we have either $\abs{h(x) - p(x)} = O(1)$ or $\abs{h(x) - p(x)}/\log x \to \infty$ as $x \to \infty$.
		\item\label{it:BRM:2} For each non-zero $h \in \spanlin\set{ f_i }{ 1 \leq i \leq k}$ and each $p(x) \in \RR[x]$, we have $\abs{h(x) - p(x)} \to \infty$ as $x \to \infty$.
	\end{enumerate}
	Then $\bra{\bra{T^{\ip{f_1(n}}(z),T^{\ip{f_2(n}}(z),\dots, T^{\ip{f_k(n}}(z) }}_{n=0}^\infty$ is equidistributed in $X^k$.
	\end{theorem}
	
	Let now $f$ be a function belonging to a Hardy field, and put $t = \ip{f}$ and $f_i(x) = f(x+i)$ for every $i$ such that  $1 \leq i \leq k$. Using standard techniques, such as the Taylor expansion and estimates on derivatives of functions in Hardy fields \cite[Lem.\ 2.1]{Frantzikinakis-2009}, one can verify that conditions \ref{thm:BMR}.\ref{it:BRM:1} and \ref{thm:BMR}.\ref{it:BRM:2} are satisfied if 
	\begin{align*}
		\frac{f(x)}{x^k \log x} &\to \infty & \text{and} && \frac{f(x)}{x^{k+1}} &\to 0 & \text{as }&& x \to \infty.
	\end{align*}
	Combining the conclusion of Theorem \ref{thm:BMR} with Theorem \ref{thm:BL-word}, one could show that  for each {\gpword} $\bb a$ and each word $w \in \Sigma^k$ we have
	\begin{equation}\label{eq:265:1}
		\freq\bra{ \bra{a_{t(n)}}_{n=0}^\infty, w} = \prod_{i=0}^{k-1} \freq\bra{ \bb a, w_i}\,.
	\end{equation}
	We emphasise that the frequencies in \eqref{eq:265:1} are, in general, not equal to $\freq(\bb a,w)$; rather, they are the values which one would expect if $t$ exhibited random-like behaviour. We leave the details to the interested reader.

\subsection{IP recurrence}

Apart from asymptotic density, one can also inquire into combinatorial 
richness of {\gp} subsets of $\NN$. 
In order to state the relevant results, we introduce some terminology. 
A set $E \subset \NN$ is an \emph{$\IP$-set} if it contains all finite 
sums of a sequence of positive integers, that is, if there exists a 
sequence of positive integers $(n_i)_{i=1}^\infty$ such that 
$\sum_{i \in I} n_i \in E$ for all finite set $I \subset \NN$. 
A set $F \subset \NN$ is $\IP^*$ if $E \cap F \neq \emptyset$ for 
every $\IP$-set $E$, or, equivalently, if for every sequence of  positive integers 
$(n_i)_{i=1}^\infty$ there exists a finite set $I \subset \NN$ with 
$\sum_{i \in I} n_i \in F$. We also define shifted variants of these 
notions. A set $E \subset \NN$ is $\IP_+$ if $(E + n) \cap \NN$ is an $\IP$-set 
for some $n \in \ZZ$. 
Likewise, a set $F \subset \NN$ is $\IP^*_+$ if $(F + n) \cap \NN$ is 
an $\IP^*$-set for some $n \in \ZZ$. 
Each $\IP^*$-set is \emph{syndetic}, meaning that for every 
$\IP^*$-set $F$ there exists an integer $N$ such that $F \cap [n,n+N) 
\neq \emptyset$ 
for all $n \in \NN_0$. In particular, all $\IP^*$-sets have positive 
lower uniform density. 
Moreover, since for all $m \in \NN$, the multiples of $m$ form an 
$\IP$-set, all $\IP^*$-sets contain a multiple of $m$. 
The class of $\IP^*$-sets is closed under intersection, meaning that 
for all pairs $F,F' \subset \NN$ of $\IP^*$-sets, $F \cap F'$ is also 
an $\IP^*$-set. 
In what follows, we say that a statement $\varphi(n)$ holds for 
\emph{almost all} $n \in \NN_0$ if $d\bra{\set{n \in \NN_0}{\neg \varphi(n)}} = 0$.

\begin{theorem}[{\cite[Thm.{}D]{BergelsonLeibman-2007}}]\label{thm:nil:IP*-rec} 
        Let $\bb a$ be a {\gpword} over $\Sigma$. Then for almost all $n_0
\in \NN_0$ and all sequences of positive integers $(n_i)_{i=1}^\infty$, 
        there exists a finite set $I \subset \NN$ such that $a_{ n_0 + 
\sum_{i \in I} n_i } = a_{n_0}$. 
        In other words, if $E \subset \NN$ is a {\gp} set with $d(E) > 0$, 
then $E$ is $\IP^*_+$. Furthermore, $(E-n) \cap \NN$ is $\IP^*$ for 
almost all $n \in E$. 
\end{theorem}

\begin{remark}
	For Sturmian words, a much simpler proof is possible, see e.g.\ a special case of \cite[Cor.\ 7.3]{Bergelson-2010}.
\end{remark}

\begin{remark}
One can obtain more precise versions of this result. Specifically, it 
follows from \cite{BergelsonLeibman-2018} that one 
can additionally require that $I \subset [r]$ for some integer $r$ 
that only depends on $\bb a$ and $n_0$. 
In a different direction, it follows from \cite{Konieczny-2017-IJM} 
that one can additionally require that the gaps between consecutive 
elements of $I$ are bounded from above by some integer $d$ that only 
depends on $\bb a$. 
\end{remark}

\section{Counting function and discrepancy}\label{sec:growth}

\newcommand{\discr}{\Delta}

Let $\bb a$ be a {\gpword} defined over a finite alphabet $\Sigma$.
In addition to the asymptotic frequency $\freq(\bb a, x)$ with which a symbol $x \in \Sigma$ appears in $\bb a$, 
one can inquire into more quantitative estimates on the count of occurrences of $x$ in $\bb a$. This leads us to consider the 
counting function 
\begin{equation}\label{eq:growth:freq}
	\cnt(\bb a, x; N) = \abs{\set{ n \in [N]}{ a_n = x}} \,,
\end{equation}
as well as the closely related notion of \emph{symbolic discrepancy} considered in  \cite{Adamczewski-2004},
\begin{equation}\label{eq:growth:discr}
	\discr(\bb a; N) = \max_{x \in \Sigma}\absBig{ \cnt(\bb a, x; N) -N \freq(\bb a,x)} \,.
\end{equation}
For example, if $\bb a$ is the Sturmian word given by $a_n = \ip{\a n + \b} - \ip{\a (n-1)+\b}$, then we have the estimates
\begin{align}\label{eq:growth:Sturm}
	\cnt(\bb a, 1; N) &= \a N + O(1) && \text{and}& \discr(\bb a, N) = O(1)\,,
\end{align}
which are significantly stronger than $\freq(\bb a, 1) = \a$.

In this section, our aim is to study the possible rates of growth of $\cnt(\bb a,x;N)$ and $\discr(\bb a; N)$ when $\bb a$ is a bracket word. We note that, as a direct consequence of the definitions, for each $x \in \Sigma$ and $N \in \NN$ we have
\[ 0\leq \cnt(\bb a, x; N),\ \discr(\bb a;N) \leq N\,.\]
Since frequencies of all symbols exist, we also have
\[
	\cnt(\bb a, x; N)/N \to \freq(\bb a,x) \text{ and } \discr(\bb a;N)/N \to 0 \ \text{as } N \to \infty\,.
\]
This quantitative approach is especially relevant in the case where $\freq(\bb a, x) = 0$. We point out that if there exists $z \in \Sigma$ with $\freq(\bb a,z) = 1$ then 
\[ \discr(\bb a; N) = N - \cnt(\bb a,z ;N) = \sum_{x \in \Sigma \setminus \{z\}}\cnt(\bb a,x ;N) \,.\]
For instance, if we let $F=\{0,1,2,3,5,8,\ldots\}$ denote the set of Fibonacci numbers, $\varphi = (1+\sqrt 5)/2$ be the golden ratio, and 
$E = \set{ n \in \NN_0}{ \fpa{\varphi n} < 1/\sqrt{n}}$, 
then the frequencies $\freq(\bb{1}_{E}, 1)$ and $\freq(\bb{1}_F,1)$ are both zero, but 
\[ \cnt(\bb{1}_F,1;N) = \discr(\bb{1}_F;N) = \Theta(\log N)\] grows much more 
slowly (cf.{} Proposition \ref{prop:growth-N^lambda}) than 
\[ \cnt(\bb{1}_E,1;N) = \discr(\bb{1}_E;N) =\Theta(\sqrt{N})\,.\] 

For context, we mention that similar questions have been studied for other classes of sequences of combinatorial interest. For instance, it is well-known (see e.g.\ \cite[Sec.\ 1.2]{Rigo-book-2}) that for each $k$-automatic word $\bb a$ over a finite alphabet $\Sigma$ and each $x \in \Sigma$, the count of integers $n$ with $a_n = x$ whose base-$k$ expansion has length $L \geq 1$, $\cnt(\bb a, x; k^{L}) - \cnt(\bb a, x; k^{L-1})$, satisfies a linear recurrence in $L$. Thus, either $\cnt(\bb a,x;N) = \Theta(\log^d N)$ for some $d \in \NN_0$ or $\cnt(\bb a,x; N) = \Omega(N^c)$ for some $c \in (0,1]$ \cite[Thm.\ 1.48]{Rigo-book-2}. In particular, if $\cnt(\bb a,x; N) = N^{o(1)}$ then $\cnt(\bb a,x; N) = \log^{O(1)}N$. Another convenient consequence is that if $\freq(\bb a, x) = 0$ then $\cnt(\bb a, x;N) = O(N^{c})$ for some $c \in (0,1)$.

\subsection{Linear growth and low discrepancy} 
 As we have already observed, the rate of growth of $\cnt(\bb a,x;N)$ can be linear and $\discr(\bb a;N)$ can be bounded. In this context, it is helpful to recall the notion of a $C$-balanced word. An infinite word $\bb a$ over an alphabet $\Sigma$ is \emph{$C$-balanced} if the number of occurrences of a given symbol in any two subwords of $\bb a$ of equal length differs by at most $C$. 
For instance, as we have already pointed out in the introduction, Sturmian words are $1$-balanced. It was shown in \cite[Prop.\ 7]{Adamczewski-2003} that an infinite word $\bb a$ is $C$-balanced for a finite value $C > 0$ if and only the discrepancy $\discr(\bb a; N)$ is bounded as $N \to \infty$.

\begin{lemma}\label{lem:growth:linear}
	For every $\a \in (0,1]$, there exists a {\gpword} over $\{0,1\}$ with $\cnt(\bb a, 1; N) = \a N + O(1)$, and thus $\freq(\bb a, 1) = \a$ and $\discr(\bb a; N) = O(1)$. 
\end{lemma}
\begin{proof}
	Follows immediately from \eqref{eq:growth:Sturm}.
\end{proof}

The example above stems from the fact that $[0,\alpha)$ is a bounded remainder set for the rotation by $\a$ on the torus $\RR/\ZZ$. In general, given $d \in \NN$ and $\a = (\a_1,\a_2,\dots,\a_d)$ a measurable set $S \subset \RR^d/\ZZ^d$ is a \emph{bounded remainder set} for the rotation by $\a$ on $\RR^d/\ZZ^d$ if for almost all $x \in \RR^d/\ZZ^d$, the \emph{discrepancy}
\begin{equation*}
	\Delta(S,x;N) = \abs{ \sum_{n = 0}^{N-1}1_S(x + n \a) - \lambda(S)N } 
\end{equation*}
is bounded by a constant $C(S,\a)$ independent of $x$ and $N$. (Above, $\lambda$ denotes the Lebesgue measure.) A complete classification of Riemann measurable bounded remainder sets for rotations on tori was obtained in \cite{GrepstadLev-2015}. In particular, by \cite[Thm.\ 1]{GrepstadLev-2015}, if $v_1,v_2,\dots,v_d \in \ZZ \a + \ZZ^d$ are linearly independent vectors such that the parallelepiped
\[
	P = \set{ \sum_{i=1}^d t_i v_i }{ t_1,t_2,\dots,t_d \in [0,1)}
\]
is contained in $[0,1)^d$ then $P$ (or, strictly speaking, its projection to $\RR^d/\ZZ^d$) is a bounded remainder set. 

\begin{example}\label{ex:growth:BRS}
	Let $d \in \NN$, let $\a = (\a_1,\a_2,\dots,\a_d) \in \RR^d$ be a $d$-tuple of real numbers such that $1,\a_1,\a_2,\dots,\a_d$ are linearly independent over $\QQ$, and let $S \subset \RR^d/\ZZ^d$ be a bounded remainder set for the rotation by $\a$ on $\RR^d/\ZZ^d$ which is also a semialgebraic set. Then $\bb a = \bra{1_S(n\a)}_{n=0}^\infty$ is a bracket word with $\freq(\bb a, 1) = \lambda(S)$ and $\discr(\bb a; N) = O(1)$ as $N \to \infty$.
\end{example}

A different way to generalise Lemma \ref{lem:growth:linear} is to consider iterated discrete derivatives of polynomials of higher degrees.
\begin{example}\label{ex:growth:diff-quad}
	Let $\a \in (0,1)$ and let $\bb a$ be the {\gpword} given by 
	\begin{align*}
	a_n &= \ip{\a (n+2)^2} - 2 \ip{\a(n+1)^2} + \ip{\a n^2} + 1 
	\\ &= - \fp{\a (n+2)^2} + 2 \fp{\a(n+1)^2} - \fp{\a n^2} + 1+2\a.
	\end{align*}	 
It follows directly from the formula above that $a_n \in \ZZ$ and $-1 < a_n < 5$ for all $n \in \NN_0$, so $\bb a$ is a word over the alphabet $\{0,1,2,3,4\}$. Consider the word $\bb a'$ over the alphabet $\{0,1\}$ given by 
\[
	a_{4n+i}' = 
	\begin{cases}
		1 & \text{if } a_n > i, \\
		0 & \text{otherwise.}
	\end{cases}
	\qquad \text{for } n \in \NN_0,\ i \in \{0,1,2,3\} \,.
\]
Then we can compute that 
\begin{align*}
	\cnt(\bb a',1;N) &= \sum_{n=0}^{N-1} a_n' = \sum_{n=0}^{N/4-1} a_n + O(1) 
	\\ &= N/4 + \ip{\a(N/4+1)^2} - \ip{\a(N/4)^2} + O(1) = (1/4+\a/8)N + O(1).
\end{align*}
Hence, $\cnt(\bb a',1;N) = (1/4+\a/8)N + O(1)$ and $\discr(\bb a';N) = O(1)$.   
\end{example}

We recall the words in Examples \ref{ex:growth:BRS} and \ref{ex:growth:diff-quad} are $C$-balanced for some finite constant $C$. 
Thanks to Theorem \ref{thm:BL-word}, one could use Example \ref{ex:growth:diff-quad} to construct an explicit example of a bounded remainder set for a certain $2$-step nilsystem. 

\subsection{Slow growth and small discrepancy}
At the other extreme in terms of the possible rates of growth of $\cnt(\bb a, x; N)$, it is possible for $\cnt(\bb a, x; N)$ to tend to $\infty$ arbitrarily slowly. In the case where $\abs{\Sigma} = 2$, $\discr(\bb a; N) = \cnt(\bb a, x; N)$ tends to $\infty$ at the same rate.
As a first example, we mention a result \cite[Theorem C]{ByszewskiKonieczny-2018-TAMS} 
asserting that any ``sufficiently sparse'' subset of $\NN_0$ is  a {\gp} set. 

\begin{proposition}[{\cite[Thm. C]{ByszewskiKonieczny-2018-TAMS}}]\label{prop:constr:super-sparse}
	Let $E = \set{n_i}{ i \in \NN_0} \subset \NN_{\geq 2}$ be a set with
	\[
		\liminf_{i \to \infty} \frac{\log n_{i+1}}{\log n_i} > 1.
	\]
	Then ${\bb 1}_E$ is a {\gpword}.
\end{proposition}

The proof of Proposition \ref{prop:constr:super-sparse} relies on an application of Proposition \ref{prop:cl:g-in-I} for a {\gp} map of the form 
$g(n) = \braif{ \fpa{\a n} < 1/n^d}$ with large $d \in \NN$ and a carefully constructed $\a \in \RR$.  
Of course, this result provides examples of bracket words 
over $\{0,1\}$ for which the growth of $\cnt(\bb a, 1; N)$ tends to infinity as slowly as wanted. 

Now, we prove the following general result that also covers growth order of type $\log N$. 

\begin{proposition}\label{prop:growth:sub-log}
	Let $f \colon \NN_0 \to \RR_{>0}$ be a non-decreasing map and assume that there exists a positive real number $C$ such that 
	$f(2n) \leq f(n) + C$ for all $n \in \NN$. Then there exists a {\gpword} $\bb a$ over $\{0,1\}$ such that $\cnt(\bb a,1;N) = f(N) + O(1)$.
\end{proposition}

\begin{proof}
	Let $F$ be the set of all Fibonacci numbers, and let $F' = F+[C]$. Note that for all sufficiently large $n \in \NN$, we have $\abs{F' \cap [n,2n)} \geq C$. 
	In \cite{ByszewskiKonieczny-upcoming}, it is shown that any subset $E$ of $F$ is a {\gp} set, and hence the same also applies to each $E \subset F'$. 
	One can inductively construct a set $E \subset F'$ such that $\abs{E \cap [N]} = f(N)+ O(1)$ for all $N$: for each $n \in F'$, assuming that $E \cap [n]$ 
	has already been constructed, we let $n \in E$ if and only if $f(n) > \abs{E \cap [n]}$. It remains to set $\bb a = \bb{1}_E$.
\end{proof}

\begin{example}
	For every $\a \in \RR_{>0}$, there is a {\gpword} $\bb a$ over $\{0,1\}$ such that $\cnt(\bb a,1;N) = \a \log N + O(1)$.
\end{example}

For the sake of completeness, we mention that for even slower rates of growth, we can reformulate \cite[{Thm.{} C}]{ByszewskiKonieczny-2018-TAMS} 
to obtain a slightly more precise result which does not involve an error term.

\begin{proposition}
	Let $f \colon \RR_{>0} \to \RR_{>0}$ be a continuously differentiable increasing function with $\sup_x\abs{f'(x)} < \infty$.
	Then there exists a {\gpword} $\bb a$ over $\{0,1\}$ with $\cnt(\bb a,1;N) = \floor{ f(\log \log N) }$ for all sufficiently large $N$.
\end{proposition}
\begin{proof}
	Since $f(\log \log x)$ is increasing in $x$ and its derivative tends to $0$ as $x \to \infty$, we can construct a set $E  = \{m_1 < m_2 < \cdots \}$ with $\abs{E \cap [N]} = \floor{ f(\log \log N) }$ for all sufficiently large $N$. Thus, for sufficiently large $n$ we have
	\[ f(\log \log m_n) < n \leq f(\log \log (m_n+1))\,.\]
As a consequence, we find that
\[
	f^{-1}(n- 1/2) \leq \log \log m_n \leq f^{-1}(n)\,.
\]
Applying this bound to $n$ and $n+1$, we obtain
\[
	\frac{ \log m_{n+1} }{\log m_{n}} \geq \exp\bra{f^{-1}(n+1/2) - f^{-1}(n) }\,.
\]	
Since $f'(x)$ is bounded as $x \to \infty$, we have
\[
	\liminf_{n \to \infty} \bra{ f^{-1}(n+1/2) - f^{-1}(n) } > 0\,,
\]	
and consequently
\[
	\liminf_{n \to \infty}  \frac{ \log m_{n+1} }{\log m_{n}} > 1\,.
\]
Thus, $\bb{1}_E$ is a {\gpword} by Proposition \ref{prop:constr:super-sparse}.
\end{proof}

\begin{example}
There are {\gpwords} $\bb a$ over $\{0,1\}$ such that $\cnt(\bb a,1;N)$ is any of the following: $\floor{ \a \log \log N}$ 
for any $\a \in \RR_{> 0}$, $\floor{ (\log \log N)^c}$ for any $c \in [0,1)$, $\floor{ \log \log \log N}$, and so on.
\end{example}

\subsection{Moderate behaviours}

Bearing in mind the examples mentioned above, we are left with the question of determining which rates of growth 
between linear and logarithmic are possible for $\cnt(\bb a,x;N)$. Towards this end, we will make use of results concerning the asymptotic behaviour of Diophantine expressions of the form $\prod_{i=1}^d \fpa{\a_i n}$, which have received considerable attention in connection with the Littlewood conjecture (see Example \ref{ex: littlewood}). While the conjecture remains unsolved, one can obtain considerably more precise estimates for a generic choice of $(\a_1,\a_2,\dots,\a_d)$. For future reference, let us define the set
\[
	E(\lambda,\a) = \set{n \in \NN}{ \prod_{i=1}^d \fpa{n\a_i} < n^{-1+\lambda} },
\]
where $d \in \NN$, $\lambda \in [0,1]$, and $\a = (\a_1,\a_2,\dots,\a_d) \in \RR^d$.
Note that $E(\lambda,\a)$ is a {\gp} set if $\lambda \in \QQ$.

\begin{proposition}\label{prop:growth-log^dN}
For every $\lambda \in [0,1) \cap \QQ$ and $c \in \NN_0$, there exists a {\gpword} $\bb a$ over $\{0,1\}$ with $\cnt(\bb a, 1; N) = \Theta(N^\lambda \log^c N)$. 
\end{proposition}

\begin{proof}
	It follows from a variant of the main result in \cite{WangYu-1981}, as cited in \cite[Thm.{} 1.6]{ChowTechnau-2022} and specialised to $\psi(n) = n^{-1+\lambda}$, 
	that for each $d \in \NN$ and almost all $\a \in \RR^d$, we have the asymptotic formula
	\begin{equation}\label{eq:WY-001}
		\abs{E(\lambda,\a) \cap [N]} = 
		\begin{cases}
		\Theta(x^{\lambda} \log^{d-1} x) &\text{if } \lambda \neq 0,\\
		\Theta(\log^{d} x) &\text{if } \lambda = 0.
	\end{cases}
	\end{equation}
	If $\lambda = c = 0$ then one can simply take $\bb a = \bb 1_{\{0\}}$, so suppose this is not the case. Let $d = c+1$ if $\lambda \neq 0$ and $d = c$ if $\lambda = 0$. It follows from \eqref{eq:WY-001} that for almost all $\a \in \RR^d$ we may take $\bb a = \bb 1_{E(\lambda,\a)}$.
\end{proof}

Unfortunately, the main result in \cite{WangYu-1981} does not provide any explicit example of $\a \in \RR^d$ for which \eqref{eq:WY-001} holds, and as a consequence our proof of Proposition \ref{prop:growth-log^dN} does not provide any explicit example of {\gpwords} $\bb a$ such that $\cnt(\bb a, 1; N)$ has prescribed asymptotic behaviour. However, such examples can be constructed in the special case where (using notation from Proposition \ref{prop:growth-log^dN}) $d = 0$. Below, we let $\varphi$ denote the golden ratio.

\begin{proposition}\label{prop:growth-N^lambda}
Let $\lambda \in (0,1) \cap \QQ$, and let $\bb a = \bb 1_{E}$, where
\[
	E = \set{ n \in \NN }{ \fpa{n\varphi} \leq n^{-1+\lambda}} \,.
\]
Then $\cnt(\bb a, 1; N) = \Theta(N^\lambda)$.
\end{proposition}
\begin{proof}
	Recall that every positive integer $n$ has a unique expansion 
	\[
		n = \sum_{i=2}^\infty \epsilon_i(n) f_i \,,
	\]
	where $f_i$ is the $i$-th Fibonacci number, $\epsilon_i(n) \in \{0,1\}$ and $(\epsilon_i(n),\epsilon_{i+1}(n)) \neq (1,1)$ for all $i \geq 2$. 
	Let $\nu(n)$ denote the least index $i \geq 1$ such that $\epsilon_{i+1}(n) = 1$, and let $\mu(n) = - \log_{\varphi} \fpa{n \varphi}$. 
	It follows from \cite[Lemma 1]{DrmotaMullerSpiegelhofer-2018} that the difference between $\nu(n)$ and $\mu(n)$ is bounded; in fact, 
	$\abs{\mu(n) - \nu(n) } \leq 3$ for all $n \in \NN$. 
	
	It is a standard exercise to show that for each $\ell \in \NN$, the number of sequences in $\{0,1\}^\ell$ with no pair of consecutive 
	$1$s is $f_{\ell+2} = \Theta(\varphi^\ell)$. As a consequence, for each sufficiently large $\ell \in \NN$ and each $c \in (-10,10)$, we have
\begin{equation}\label{eq:226:1}
 \abs{ \set{ n \in [f_{\ell-1}, f_{\ell}) }{  \nu(n) \geq (1-\lambda) \ell + c} } = \Theta\bra{ \varphi^{\lambda \ell}}\,.
\end{equation}
Since $\abs{\mu(n) - \nu(n) } \leq 3$ and $\abs{\log_{\varphi} n - \ell} < 2$ for all $n \in  [f_{\ell-1}, f_{\ell})$, 
it follows that
\begin{equation}\label{eq:226:2}
 \abs{ \set{ n \in [f_{\ell-1}, f_{\ell}) }{  \mu(n) \geq (1-\lambda) \log_{\varphi} n } } = \Theta\bra{ \varphi^{\lambda \ell}}\,.
\end{equation}
Let $N$ be a large integer, and let $i$ be the index with $f_{i-1} \leq N < f_{i}$. Summing \eqref{eq:226:2} over all $\ell \leq i$ (resp.{} $\ell \leq i-1$), 
we conclude that 
\[ \abs{E \cap [N]} =  \abs{ \set{ n \in [N] }{  \mu(n) \geq (1-\lambda) \log_{\varphi} n } } = \Theta\bra{ \varphi^{\lambda i}}=\Theta\bra{ N^{\lambda}}\, . \qedhere
\]	
\end{proof}

In general, it is not clear how to construct $\a \in \RR^d$ which satisfy \eqref{eq:WY-001} for $d \geq 2$. However, some progress in that direction follows from recent work of Chow and Technau. For instance, it follows as a special case of \cite[Theorem 1.9]{ChowTechnau-2022}, that for each $\lambda \in (0,1)$ and each $\alpha_1 \in \RR$ which is not a Liouville number (\textit{i.e.}, $\fpa{n\a_1}\gg 1/n^C$ for some constant $C$), for almost all $\alpha_2$ we have the we have the one-sided variant of  \eqref{eq:WY-001},
	\[
		\abs{E(\lambda,(\a_1,\a_2)) \cap [N]} \gg N^{\lambda} \log N\,.
	\]

\subsection{Almost linear growth}

It might come as a surprise that there exist {\gpwords} $\bb a$ over $\{0,1\}$ such that the growth rate $\cnt(\bb a,1; N)$ is slower than, 
but arbitrarily close to, linear. This is closely related to a classical result of Khinchine asserting that there exist pairs $\a_1,\a_2 \in \RR$, 
with $1,\a_1,\a_2$ linearly independent over $\QQ$, such that $\min_{\abs{n} < N} \fpa{n_1\a_1 + n_2\a_2}$ tends to $0$ arbitrarily fast 
as $N \to \infty$, see e.g. \cite[Ch. V, Thm. XIV]{Cassels-book}. We will need the following elementary lemma.

\begin{lemma}\label{lem:exist-swell-approx}
	Let $(N_i)_{i = 1}^\infty$ be a sequence of positive integers and $(\e_i)_{i = 1}^\infty$ be a sequence with values in $(0,1)$  such that
	\[ N_{i+1} \geq 2 N_i/ \e_i\, .\]
	Then there exists $\a \in \RR \setminus \QQ$ with $\fpa{N_i\a} \leq \e_i$ for all $i \in \NN$. 
\end{lemma}

\begin{proof}
	For every $j \in \NN$,  we set 
	\[
		A_j = \set{ \a \in \RR/\ZZ }{ \fpa{N_i\a} \leq \e_i  \text{ for all } 1 \leq i \leq j}.
	\]
	It will suffice to show that the set $A= \bigcap_{j=1}^\infty  A_j$ is uncountable.
	
	Note that, for each $j \in \NN$, $A_{j+1}$ is the intersection of $A_{j}$ with the periodic set $\set{ \a \in \RR/\ZZ }{ \fpa{N_{j+1}\a} \leq \e_{j+1}}$, which is the union of the intervals
	 \[
	 	I_{j+1}^{(m)} = \left[ \frac{m-\e_{j+1}}{N_{j+1}}, \frac{m+\e_{j+1}}{N_{j+1}} \right], \qquad  m \in [N_{j+1}].
	 \]
	 This motivates us to recursively define a family of sets $A_j'$ by $A_1' = A_1$ and
	\[
		A_{j+1}' = \bigcup \set{I_{j+1}^{(m)} }{ m \in [N_{j+1}] \text{ and }I_{j+1}^{(m)} \subset A_{j}' }.
	\]
	A routine inductive argument shows that $A'_j \subset A_j$ for each $j \in \NN$, and it follows directly from the definition that $A'_j$ is a union of intervals with lengths ${2\e_{j}}/{N_j}$. Set $A'= \bigcap_{j=1}^\infty  A_j' \subset A$.  
	
	Since ${2\e_{j}}/{N_i} \geq 4/N_{j+1}$, for each $m \in [N_j]$ there exists $m' \in [N_{j+1}]$ such that 
	\begin{equation}\label{eq:294:1}
		I_{j+1}^{(m')}, I_{j+1}^{(m'+1)} \subset \left[ \frac{m'-1}{N_{j+1}}, \frac{m'+2}{N_{j+1}} \right] \subset I_{j}^{(m)}.
	\end{equation}
Applying this observation inductively and using Cantor's intersection theorem, we conclude that $I_j^{(m)} \cap A' \neq \emptyset$ for each $j \in \NN$ and $m \in [N_j]$ such that $I_j^{(m)} \subset A_j$. In fact, since the left hand side of \eqref{eq:294:1} includes two disjoint intervals, we can use a similar argument to produce an injective map from $\{0,1\}^\NN$ to $A'$. Thus, $A'$ is uncountable, and so is $A$.
\end{proof}

\begin{proposition}\label{prop:growth-slow-o(N)}
	Let $f \colon \NN \to (0,1]$ be any function with $f(N) \to 0$ as $N \to \infty$. There exists a {\gpword} $\bb a$ over $\{0,1\}$ such that $\freq(\bb a, 1) = 0$ 
	but $\cnt(\bb a,1;N) = \discr(\bb a; N) \geq f(N)N$ for all $N \in \NN$.
\end{proposition}

\begin{proof}
	We may assume without loss of generality that $f(n) \geq 1/n$ for all $n \in \NN$ and that $f$ is non-increasing; if this is not the case, we can freely replace $f(n)$ with $\max\bra{1/n,f(n),f(n+1),f(n+2),\dots}$.
	For each $N$, let $h(N)$ denote the least positive integer with $f(h(N)) < 1/N$, and let $(N_i)_{i =1}^\infty$) be the sequence defined by 
	$N_1 = 1$ and
	\[
		N_{i+1} = 2 h(N_i)^2\, .
	\]
	Put also $\e_i = N_{i}/h(N_{i+1})^2 = 2N_{i}/N_{i+2}$. Note that $h(N) > N$ so $N_{i+1} > 2N_i^2$, and in particular the sequence $N_i$ is increasing.
	
	By Lemma \ref{lem:exist-swell-approx}, there exist $\a_0, \a_1 \in \RR \setminus \QQ$ such that $\fpa{ \a_{j} N_{i} } \leq \e_{i}$ 
	for all $j \in \{0,1\}$ and $i \in \NN$ with $i \equiv j \bmod{2}$. 
	For $j \in \{0,1\}$, set 
	\[
		E_j = \set{ n \in \NN }{ \fpa{n \a_j} \leq 1/n}\, ,
	\]  
	$E = E_0 \cup E_1 \cup [h(N_1)]$, and $\bb a = \bb{1}_E$.	We claim that $\bb a$ satisfies the required conditions. 
	Since $\a_0$ and $\a_1$ are irrational, we have 
	\[ \freq(\bb a,1) = d(E) \leq d(E_0) + d(E_1) = 0\, .\]
	
	It remains to show that $\cnt(\bb a,1;N) = \abs{E \cap [N]} \geq f(N)N$ for all $N$. If $N \leq h(N_1)$ then $[N] \subset E$, 
	so we may assume that $N > h(N_1)$. Let $i$ denote the unique index such that $h(N_i) \leq N < h(N_{i+1})$, 
	and put $\a = \a_{i \bmod 2}$. Since $\fpa{N_i \a} \leq \e_i$, 
	we have
\[
	\fpa{mN_i \a} \leq m \fpa{N_i \a} \leq m \e_i \leq \frac{1}{mN_i}
\]
and thus $m N_i \in E$ for all integers $m$ with $1 \leq m \leq h(N_{i+1})/N_{i}$. It follows that
\[
	\abs{ E \cap [N] } \geq N/N_i \geq f(h(N_{i}))N \geq f(N)N. \qedhere
\]
\end{proof}

\begin{example}
	There exists a {\gpword} $\bb a$ over $\{0,1\}$ with $\freq(\bb a,1) = 0$ and $\cnt(\bb a,1;N) \geq N/\log\log\log N$ for all sufficiently large $N$. 
	Note, however, that Proposition \ref{prop:growth-slow-o(N)} does not ensure the existence of a bracket word $\bb a$ with 
	$\cnt(\bb a,1;N) = \Theta(N/\log\log\log N)$.
\end{example}

\subsection{Algebraic coefficients}

The construction in Proposition \ref{prop:growth-slow-o(N)} relies on finding a pair of real numbers $(\a_0,\a_1)$ 
with some rather unusual Diophantine properties. If we restrict our attention to {\gpwords} with algebraic coefficients (c.f. Remark \ref{rmk:alg-coeff}), 
such constructions are no longer possible and some gap appears in the possible growth rates.

\begin{proposition}\label{prop:growth-alg-gap}
	Let $\bb a$ be a {\gpword} defined over a finite alphabet $\Sigma$ which {arises from a {\gp} map with algebraic coefficients} 
	and let $x \in \Sigma$. If $\freq(\bb a,x) = 0$, then there exists $c > 0$ such that  $\cnt(\bb a, x; N) = O(N^{1-c})$.
\end{proposition}

\begin{proof}
In this proof, we will make extensive use of the material discussed in Appendix \ref{app:nil}. We also recall that definitions of algebraic varieties and semialgebraic sets are given in Section \ref{ssec:semialgebraic}.

	By Theorem \ref{thm:BL-poly}, there exists a nilmanifold $X = G/\Gamma$ with $G$ connected and simply connected, 
	as well as a polynomial sequence $g \colon \ZZ \to G$ and a semialgebraic set $S \subset X$ such that for each $n \in \NN_0$ 
	we have the equivalence: $a_n = x \Longleftrightarrow g(n)\Gamma \in S$. 
	The nilmanifold $X$ can be equipped with a Mal'cev coordinates $\tau \colon X \to [0,1)^d$, where we put $d = \dim G$. Recall also from \eqref{eq:Malcev-G} that we have a closely related coordinate map $\tilde\tau \colon G \to \RR^d$. 
	Furthermore, inspecting the construction in \cite{BergelsonLeibman-2007}, we see that $X$ can be equipped with a Mal'cev coordinates $\tau$ such that 
	$\tilde \tau \circ g$ is a polynomial with algebraic coefficients and $\tau(S)$ is defined by equations and inequalities with algebraic coefficients.  
	Since $\freq(\bb a,x) = 0$, the measure of $S$ is zero. As a consequence (see e.g.\ \cite[Sec.\ 2.8]{BochnakCosteRoy-book}), $S$ is contained in a proper algebraic variety $V \subsetneq \RR^d$, which is defined over $\QQ$. 
	
	Now, we pick a non-zero polynomial $p \colon \RR^d \to \RR$ that vanishes on $V \cup \partial [0,1]^d$, which exists because $V \cup \partial [0,1]^d$ is contained in a proper algebraic sub-variety of $\RR^d$ (specifically, in the union of the algebraic variety $V$ and $2d$ hyperplanes). 
	Throughout the argument, we allow all implicit constants to depend on $X$, $g$, and $S$. 
	There are now two cases to consider, depending on the distribution of $g(n)\Gamma$. 
	
	Suppose first that the sequence $(g(n)\Gamma)_{n=0}^\infty$ is equidistributed in $X$. Then by Lemma \ref{lem:quant-equi-algebraic}, 
	there exists $c_1 > 0$ such that for each $N$ the sequence $\bra{ g(n)\Gamma }_{n=0}^{N-1}$ is $O(N^{-c_1})$-equidistributed for all $N \in \NN$.  
	For $\delta > 0$, consider the function $H_{\delta} \colon X \to [0,1]$ defined by
\[
	H_\delta(x) = \max\braBig{ 1-\frac{\abs{p(\tau(x))}}{\delta}, 0 } \,.
\]
	Then $\norm{H_{\delta}}_{\mathrm{Lip}} = O(1/\delta)$ and $\int H_{\delta} d\mu_{X} = O( \delta^{c_2})$ for some $c_2 > 0$. 
We can now estimate
\[
	\cnt(\bb a,x;N) \leq \sum_{n=0}^{N-1} H_\delta(g(n)\Gamma) \ll \frac{N^{1-c_1}}{\delta} + \frac{N}{\delta^{c_2} } \,\cdot
\] 
	Letting $\delta = N^{-c/c_2}$ for a sufficiently small $c > 0$, we obtain $\cnt(\bb a,x;N) = O(N^{1-c})$.
	
	Let us assume now that the sequence $(g(n)\Gamma)_{n=0}^\infty$ is not equidistributed in $X$. In that case,  
	there exists a horizontal character $\eta \colon X \to \RR/\ZZ$ such that $\eta \circ g$ is constant. 
	Let $G'$ be the connected component of $\ker \eta$ and let $\Gamma' = \Gamma \cap G'$. Then, $g(n)$ can be decomposed as 
	$g(n) = g'(n)\gamma(n)$, where $g'$ takes values in $G'$ and $\gamma$ is periodic.  This can be shown by adapting the proof of 
	\cite[Prop.{} 9.2]{GreenTao-2012}, or using this result as a black-box and passing to the limit $\delta \to 0$; periodicity of $\gamma$ 
	follows from Lemma A.12 therein. Let $q$ denote a period of $\gamma$. Each of the sequences $\bb a^{(i)}$, $0 \leq i < q$, defined by 
	$a^{(i)}_n = a_{qn+i}$ can be represented using the nilmanifold $G'/\Gamma'$ rather than $X$. 
	Reasoning by induction with respect to the dimension $d$, we may assume that for each $i$, $0 \leq i < q$, there exists 
	$c_i$ such that $\cnt(\bb a^{(i)},x;N) = O( N^{1-c_i} )$. Letting $c = \min_{0 \leq i < q} c_i$, we see that $\cnt(\bb a,x;N) = O(N^{1-c})$.
\end{proof}

\subsection{Concluding remarks}

We close this section by mentioning several rates of growth about which we do not know if they can be realised. 
It is also not clear if restriction to {\gpwords} with algebraic coefficients influences the answer.

\begin{question}
Does there exist a {\gpword} $\bb a$ over $\{0,1\}$ such that one of the following rate of growth holds?
\begin{enumerate}
\item $\cnt(\bb a, 1; N) = \Theta(N^{\lambda})$, where $\lambda \in (0,1) \setminus \QQ$\,.
\item $\cnt(\bb a, 1; N) = \Theta(\log^c N)$, where $c \in (1,\infty) \setminus \NN$\,.
\item $\cnt(\bb a, 1; N) = \Theta\bra{ e^ {\sqrt{\log N}}}$\,.
\end{enumerate}
\end{question}

\section{Computability}\label{sec:comp}

As we have already seen, a single {\gpword} can be represented by many different formulae. 
This happens, for instance, in Lemma \ref{lem:cl:[g=0]}. An even simpler example is provided by the {\gpword} $\bb a = \bb 1_{0}$, 
which can be represented as $a_n = \ip{1-\fp{\a n}}$ for any $\a \in \RR \setminus \QQ$. 
A more surprising representation of the same {\word} is attributed to H{\aa}land Knutson in \cite[p.\ 2]{GrahamOBryant-2010}:
\[
	a_n = \ip{ \ipnormal{\sqrt{2}n}2 \sqrt{2}n}- \ip{\sqrt{2}n}^2 - 2n^2 + 1\,.
\]

Existence of multiple representations of a single {\gpword} is not, in and of itself, surprising or worrying. 
After all, the same phenomenon occurs for usual polynomials, say with real coefficients.  
The reason why this ambiguity does not lead to problems in that case is that each polynomial $p(x) \in \RR[x]$ has 
a canonical representation $p(x) = \sum_{i=0}^d \a_i x^d$, where $\a_i \in \RR$ and $\a_d \neq 0$.

Leibman \cite{Leibman-2012} constructed a similar canonical representation for {\gp} sequences. 
The full statement of the main result, \cite[Thm.{} 0.1 \& 0.2]{Leibman-2012}, is rather technical, so we state only a simplified version. 
Specifically, we restrict our attention to {\gpwords} (Leibman's result concerns arbitrary bounded {\gp} maps from $\ZZ^d$ to $\RR$) and do not 
discuss the details of the construction (Leibman explicitly describes the families $\cF_M$ appearing below). 

\begin{theorem}\label{thm:Leib}
	Let $\bb a$ be a {\gpword} defined over a finite alphabet $\Sigma$. Then there exist families $\cF_M$, $M \in \NN$, of {\gp} maps from $\NN_0$ to $\RR$, depending only on polynomials which appear in a given representation of $\bb a$, 
	with the following properties. 
\begin{enumerate}
	\item\label{it:L:A} For each $M,d \in \NN$, and each $d$-tuple of different maps $v_1,v_2, \dots, v_d \in \cF_M$, 
	the sequence $(\fp{v_1(n)},\fp{v_2(n)},\dots,\fp{v_d(n)})_{n=0}^\infty$ is equidistributed in $[0,1)^d$.
	\item\label{it:L:B} There exist integers $M,Q,d \in \NN$, maps $v_1,v_2,\dots,v_d \in \cF_M$, and a partition of 
	\[ \textstyle \{0,{1}/{Q},\dots,(Q-1)/{Q}\}  \times [0,1)^d \]
	into pairwise disjoint semialgebraic pieces $S_i$, $i\in \Sigma$, such that for each $n \in \NN$ and $i \in \Sigma$, 
	\begin{align}\label{eq:419:1}
		a_n = i && \text{ if and only if }&& (\fp{n/Q},\fp{v_1(n)},\fp{v_2(n)},\dots,\fp{v_d(n)}) \in S_i\,.
	\end{align}
	\end{enumerate}   
\end{theorem}

\begin{remark}
\begin{enumerate}
\item The families $\cF_M$ are explicitly constructed in \cite{Leibman-2012}. 
We start with polynomial maps $p_1(x),p_2(x),\ldots \in \RR[x]$ that span the $\QQ$-algebra generated by the polynomials which 
appear in some fixed representation of $\bb a$. As the first step, $p_1/M,p_2/M, \ldots \in \cF_M$. In subsequent steps, we add to $\cF_M$ 
some elements of the form $u \fp{v}$ where $u,v \in \cF_M$ have already been constructed. For details on exactly which of such elements should be added, 
we refer to the original paper.
\item The representation in \ref{thm:Leib}\ref{it:L:B} can be explicitly computed for a given representation of the {\gpword} $\bb a$. 
For fixed $M$, $Q$, and $d$, the sets $S_i$ are determined uniquely up to a set of measure zero. 
\end{enumerate}
\end{remark}

As a consequence of Theorem \ref{thm:Leib}, for a given bracket word $\bb a$ over an alphabet $\Sigma$ and $i \in \Sigma$, one can check if $a_n = i$ 
for almost all $n \in \NN_0$. Indeed, it is enough to verify if $S_i$ has full measure. Since, for any two {\gpwords} $\bb a$ and $\bb b$ over the same alphabet, 
we can construct the {\gpword} $(\braif{a_n=b_n})_{n=0}^\infty$, one can determine whether $a_n = b_n$ for almost all $n \in \NN_0$.  
The situation is totally different 
when we insist on exact equality, even when restricting our attention to bracket words with algebraic coefficients.

\begin{theorem}\label{thm:compbis}
	 Given {\gp} map $g \colon \NN_0 \to \{0,1\}$ 	with algebraic coefficients, the following problems are undecidable.
\begin{enumerate}
	\item\label{it:compbis-A} Determine if $g(n) = 0$ for at least one $n \in \NN_0$.
	\item\label{it:compbis-B} Determine if $g(n) = 0$ infinitely many $n \in \NN_0$.
	\item\label{it:compbis-B'}Determine if  $g(n)=0$ for all but finitely many $n\in\NN_0$.
	\item\label{it:compbis-A'}Determine if  $g(n) = 0$ for all $n \in \NN_0$.
\end{enumerate}	
\end{theorem}

\begin{corollary}
It is undecidable if a given {\gp} map $g \colon \NN_0 \to \RR$ with algebraic coefficients takes only finitely many values.
\end{corollary}

\begin{proof}
Take an arbitrary {\gp} map $g \colon \NN_0 \to \{0,1\}$ and consider the {\gp} map $h\colon \NN_0 \to \RR$ defined by $h(n)=g(n)n$. Then 
$h$ takes finitely many values if and only if $g(n)=0$ for all but finitely many $n\in\NN_0$.  
\end{proof}

\begin{remark}
\begin{enumerate}
\item	In Theorem \ref{thm:compbis}, item \ref{it:compbis-A} is reminiscent of the undecidability of Hilbert's tenth problem, concerning the existence of integer solutions to polynomial equations. In fact, our argument proceeds by a reduction to this problem.
	
\item	In the case of $k$-automatic sequences, the analogues of properties in 	\ref{thm:compbis}\ref{it:compbis-A}--\ref{it:compbis-A'} are easily seen to be decidable.
	
\item	 Another related result, due to Allouche and Shallit \cite[Thm.\ 5.2]{AlloucheShallit-1992}, asserts that it is undecidable if a given $k$-regular sequence with integer values has at least one vanishing term. We point out that, conversely to Theorem \ref{thm:compbis}\ref{it:compbis-A'}, it is decidable if a $k$-regular sequence is identically zero \cite[Thm.\ A]{KrennShallit-2022}.
\end{enumerate}

\end{remark}

A key component in the proof of Theorem \ref{thm:compbis} is the existence of a surjective {\gp} map from $\NN_0$ to $\ZZ^d$ for each $d \in \NN$.

\begin{proposition}\label{prop:comp:surjective-Z^d}
	For each $d \in \NN$, there exists a {\gp} map $h_d \colon \NN_0 \to \ZZ^d$ with algebraic coefficients such that for each $x \in \ZZ^d$ there exist infinitely many $n \in \NN_0$ with $h_d(n) = x$. In particular, $h_d$ is surjective.
\end{proposition}

Before we proceed to construct the maps mentioned in Proposition \ref{prop:comp:surjective-Z^d}, let us show how their existence can be used to deduce Theorem \ref{thm:compbis} from classical undecidability results.

\begin{proof}[Proof of Theorem \ref{thm:compbis} assuming Proposition \ref{prop:comp:surjective-Z^d}]
	Let $p$ be an arbitrary polynomial in $\mathbb Z[x_1,\ldots,x_d]$ and let $h \colon \NN_0 \to \ZZ^d$ be the map constructed in Proposition 
	\ref{prop:comp:surjective-Z^d}. Define $g_p$ by $g_p(n) = \braif{ p(h(n)) = 0}$ for $n\in \NN_0$. Then  $g_p$ is a 
	{\gp} map from $\NN_0$ to $\{0,1\}$ with algebraic coefficients, and the following conditions are equivalent. 
\begin{enumerate}
\item[{\rm (a)}] There exists $n \in \NN_0$ with $g_p(n) = 0$.
\item[{\rm (b)}] There exist infinitely many $n \in \NN_0$ with $g_p(n) = 0$.
\item[{\rm (c)}] There exist $x_1,x_2,\dots,x_d \in \ZZ$ with $p(x_1,x_2,\dots,x_d) = 0$.
\end{enumerate}	
Hilbert's tenth problem is known to be undecidable (see, for instance, \cite{Matiyasevich-1993}). Hence, comparing (a) and (c), we conclude that there is no algorithm to determine whether an element of the set 
	 $\{g_p \mid d\geq 1, \,p\in\mathbb Z[x_1,\ldots,x_d]\}$   has a zero. Similarly,  comparing (b) and (c) we conclude that there is no algorithm to determine whether an element of the aforementioned set has infinitely many zeroes. This finishes the proof in cases \ref{it:compbis-A} and \ref{it:compbis-B}.
	 
	 To derive case \ref{it:compbis-B'} from \ref{it:compbis-B}, it is enough to notice that, for any map $g \colon \NN_0 \to \{0,1\}$, the condition that $g(n) = 0$ for all but finitely many $n \in \NN_0$ is equivalent to the condition that there are finitely many $n \in \NN_0$ such that $1-g(n) = 0$. Likewise,  \ref{it:compbis-A'} follows from \ref{it:compbis-A} 
	 and the observation that, $g(n) = 0$ for all $n \in \NN_0$ is equivalent to the condition that there does not exist $n \in \NN_0$ such that $1-g(n) = 0$. 
\end{proof} 

We now turn to the proof of Proposition \ref{prop:comp:surjective-Z^d}. The main difficulty lies in the construction of a surjective {\gp} map $\NN_0 \to \NN_0^2$; once such a map is constructed, it will not be difficult to use it to construct surjective maps $\NN_0 \to \ZZ^d$ for all $d \in \NN$. As a source of motivation, let us consider a (random) map $\bb f \colon \NN_0 \to \NN_0^2$ given by $\bb f(n) = \bra{\ip{\bb X_n^2 n},\ip{\bb Y_n^2 n}}$, where $\bb X_n,\bb Y_n$ for $n \in \NN_0$ are jointly independent random variables, uniformly distributed in $[0,1)$. One can explicitly compute that, for fixed $k,l \in \NN$ and sufficiently large $N$, we have \[\PP(\bb f(n) = (k,l)) = \frac{\bra{\sqrt{k+1}-\sqrt{k}}\bra{\sqrt{l+1}-\sqrt{l}}}{n}\,\cdot\] Hence, by the second Borel–Cantelli lemma, almost surely there exists infinitely many $n \in \NN_0$ with $\bb f(n) = (k,l)$. This prompts us to consider generalised polynomials of the form $\bra{ \ip{ \fpnormal{\sqrt{2} n}^2 n }, \ip{ \fpnormal{\sqrt{3} n}^2 n }}$, or, more generally 
\[
	\bra{ \ip{ \fpnormal{\sqrt{2} n}^A n }, \ip{ \fpnormal{\sqrt{3} n}^A n }}
\]
for $A \geq 2$, and exploit equidistribution properties of the sequence $\bra{\fpnormal{\sqrt{2} n},\fpnormal{\sqrt{3} n}}$ in $[0,1)^2$. As a first step in that direction, we mention the following quantitative equidistribution estimate.

\begin{lemma}\label{lem:comp:equidist-linear}
	There exist $N_0 > 0$ and $c > 0$ such that for each $M \in \ZZ$, $N \geq N_0$, and $(x,y) \in [0,1)^2$, 
	there exists $n \in [M,M+N)$ such that 
	$$\max\left\{\fpa{\sqrt{2}n - x}, \fpa{\sqrt{3}n-y}\right\} \leq 1/N^c\,.$$
\end{lemma}

\begin{proof}
	A more precise version follows from \cite[Thm 1.(ii)]{Chen-2000} combined with standard estimates on the quality of approximate rational relations 
	between algebraic numbers. Specifically, we have the well-known estimate $$\min_{(n,m) \in [N]^2 \setminus \{(0,0)\}} \fpa{n\sqrt{2} + m \sqrt{3}} \gg N^{-3}\,,$$
	which is easily obtained by noticing that the norm $\prod_{\sigma,\rho \in \{\pm 1\} }\bra{k +\sigma n\sqrt{2} + \rho m \sqrt{3}}$ is an integer for all $k,n,m \in \ZZ$. (Stronger estimates follow from Schmidt's subspace theorem \cite{Schmidt-1972}, but are not needed for our purposes.)
	 
	The theorem is applied with $n = 2$, $m=1$, $\a_1 = x$, $\a_2 = y$, $T_1 = M$, $T_1'=M+N$, $\delta_1=\delta_2=1$, and $M = N^{1/7}$; 
	then $\Delta = 2$, $\Lambda_2 \gg N^{-3}$, and $\norm{S(T)} = N$. The conclusion, after elementary manipulations, asserts that there exists 
	$n \in [M,M+N)$ such that 
	$$\max\left\{\fpa{\sqrt{2}n - x}, \fpa{\sqrt{3}n-y}\right\} \ll N^{-2/7}\,.$$ Alternatively, one can also derive this estimate as a special case of 
	Theorem \ref{thm:GT}.
\end{proof}

\begin{lemma}\label{lem:comp:surjective}
	For $A \in \NN$, define the map $g_A \colon \NN_0 \to \NN_0^2$ by 
	\[
		g_A(n) =  \bra{ \ip{ \fpnormal{\sqrt{2} n}^A n }, \ip{ \fpnormal{\sqrt{3} n}^A n }}\,.
	\] 
	There exists a constant $A_0$ such that for all $A > A_0$, for all $(k,l) \in \NN_0$ the set
	\[
		\set{ n \in \NN_0 }{ g_A(n) = (k,l) }
	\]
	is infinite. In particular, $g_A$ is surjective.
\end{lemma}

\begin{proof}
	Let $c$ be the constant from Lemma \ref{lem:comp:equidist-linear} and let $A > 3/c$. 
	Fix $(k,l) \in \NN_0^2$, let $N$ be a sufficiently large integer (to be determined in the course of the argument), and put $M = N^2$. 
	It will suffice to show that there exists $n \in [M,M+N)$ with $g_A(n) = (k,l)$. 
	
	Pick $(x,y)\in [0,1)^2$ such that $x^A M = k$ and $y^A M = l$ 
	(that is, $x = (k/M)^{1/A}$ and $y = (l/M)^{1/A}$). By Lemma \ref{lem:comp:equidist-linear}, we can find $n \in [M,M+N)$ such that 
	$x \leq \fp{\sqrt{2}n} < x+1/N^{c}$ and $y \leq \fp{\sqrt{3}n} < y+1/N^{c}$. It remains to show that
	\begin{equation}\label{eq:785:1}
		k \leq \fpnormal{\sqrt{2} n}^A n < k+1 \;\; \mbox{ and } \;\; l \leq \fpnormal{\sqrt{3} n}^A n < l+1\,.
	\end{equation}	
We only consider the first of these two conditions, the second one is fully analogous. The lower bound is immediate:
\[
	\fpnormal{\sqrt{2} n}^A n \geq x^A n = kn/M \geq k\,.
\]
For the upper bound, we first obtain that 
\begin{align*}
	\fpnormal{\sqrt{2} n}^A n < (x+N^{-c})^A (M+N) \,.
\end{align*}
If $k = 0$, then $x = 0$ and thus $(x+N^{-c})^A (M+N) < 2 N^{2-Ac} < 1$ (recall that $A > 3/c$). 
If $k \geq 1$, then $x > N^{-2c/3}$ and hence
\begin{align*}
(x+N^{-c})^A (N_0+N) &= x^A N_0 \bra{1+N^{-c}/x}^A(1+N/M) 
\\ &\leq k \brabig{1+N^{-c/3}}^A(1+1/N) \leq k \exp\bra{ (A+1) N^{-c/3} } \,.
\end{align*}
Thus, we can find $N_0 = N_0(k) = O_{A}(k^{3/c})$ such that, for all $N \geq N_0$, we have
\[
	\fpnormal{\sqrt{2} n}^A n\leq  k \exp\bra{ (A+1) N^{-c/3} } < k+1 \,,
\]
as needed.
\end{proof}

\begin{proof}[Proof of Proposition \ref{prop:comp:surjective-Z^d}]
Let $g_2 \colon \NN_0 \to \NN_0^2$ denote the map constructed in Lemma \ref{lem:comp:surjective}. 
For $d \geq 2$, define  $g_{d} \colon \NN_0 \to \NN_0^{d}$ inductively by 
\[ g_{i+1}(n) = \bra{ g_2(n)_1, g_i(g_2(n)_2)}.\] 
An inductive argument shows that for each $x \in \NN_0^d$, there are infinitely many $n \in \NN_0$ with $g_d(n) = x$. Next, define $h_d \colon \NN_0 \to \ZZ^d$ by 
\[ h_d(n) = \bra{ g_{2d}(n)_1 - g_{2d}(n)_{d+1}, g_{2d}(n)_2 - g_{2d}(n)_{d+2}, \dots, g_{2d}(n)_{d} - g_{2d}(n)_{2d}}\,.\] 
Since the map $\NN_0^2 \to \ZZ$, $(n,m) \mapsto n-m$ is surjective, for each $x \in \ZZ^d$, there are infinitely many $n \in \NN_0$ with $h_d(n) = x$. 
\end{proof}

\subsection{Consequences for bracket words}

As a consequence of Theorem \ref{thm:compbis}, we deduce Theorem \ref{thm:comp}, which we restate below for the reader's convenience.

\thmdecide*

\begin{proof}
Given an arbitrary {\gp} map $g \colon \NN_0 \to \{0,1\}$, define $\bb a$ by $a_n = g(n)$ for all $n \in \NN_0$, and let $\bb b$ be defined by $b_n=0$ 
for all $n\in \NN_0$.  Then $\bb a = \bb b$ if and only if $g(n) = 0$ for all $n \in \NN_0$.  
Thus, the results follows from Theorem \ref{thm:comp}.
\end{proof}

\section{Linear recurrences}\label{sec:pisot}

In earlier sections, we have encountered examples of {\gpwords} related to linear recurrence sequences, 
such as the Fibonacci numbers. Here, we discuss these results in more detail and provide some new arguments. 

\subsection{Results}

We recall that a \emph{Pisot number} is a real algebraic integer $\beta > 1$ such that all Galois conjugates of $\beta$ have absolute value strictly less than $1$. 
Similarly, a \emph{Salem number} is a real algebraic integer $\beta > 1$ whose Galois conjugates  all have absolute value no greater than $1$, and at 
least one of which has absolute value exactly $1$. The minimal polynomial for a Salem number must be reciprocal, which implies that $1/\beta$ is a Galois conjugate of $\beta$, and that all other roots have absolute value exactly one. As a consequence, a Salem number is a unit in the ring of algebraic integers. 

\begin{theorem}\label{prop:pisot:rec}
	Let $\beta > 1$ be an algebraic unit with minimal polynomial $p(x) = x^d - \sum_{i=1}^d x^{d-i} a_i$ for some $d \in \NN$ and $a_1,a_2, \dots, a_d \in \ZZ$. 
	Let $(n_i)_{i=0}^\infty$ be a sequence of non-negative integers satisfying the linear recurrence 
	\[ n_{i+d} = \sum_{i=1}^d a_i n_{i+d-1}\, , \qquad i \in \NN_0 \,,\]	
	and put $E = \set{n_i}{i \in \NN_0}$. Then $\bb{1}_E$ is a {\gpword} if one of the following holds. 
	\begin{enumerate}
	\item\label{it:pisot:rec-1} $d = 2$ (in this case, $\beta$ must be a Pisot number).
	\item\label{it:pisot:rec-2} $d = 3$ and $\beta$ is a Pisot number with no real Galois conjugate ({\it i.e.}, the discriminant of $p$ is negative). 
	\item\label{it:pisot:rec-3} $\beta$ is a Salem number.
	\end{enumerate}
\end{theorem}
\begin{proof}
{ Case \ref{it:pisot:rec-1} is covered by \cite[Thm.{} B]{ByszewskiKonieczny-2018-TAMS}, using an argument which relies on best rational approximations to 
quadratic irrationals. Case \ref{it:pisot:rec-2} likewise follows from \cite[Thm.{} B]{ByszewskiKonieczny-2018-TAMS} under mild additional assumptions.  
Below in Theorem \ref{prop:pisot} we give a complete proof of Case \ref{it:pisot:rec-2} which has a distinctly algebraic flavour, in contrast to the argument in 
\cite{ByszewskiKonieczny-2018-TAMS} relying on Diophantine approximation. 
Finally, case \ref{it:pisot:rec-3} is covered in the upcoming preprint \cite{ByszewskiKonieczny-upcoming} using methods analogous to those used in Theorem \ref{prop:pisot}. 
}
\end{proof}

In the remainder of this section, we will mostly speak of {\gp} subsets of $\NN_0$ rather than {\gpwords}. We recall that, as defined in Section \ref{sec:def}, these terms are closely related, and the connection between them if further elucidated in Corollary \ref{cor:constr:fibre}.
	
\subsection{Cubic Pisot units}

The main ingredient in the proof of Theorem \ref{prop:pisot:rec}\ref{it:pisot:rec-2} is the following result. 
In fact, the two results are equivalent due to a reduction obtained in \cite[Prop.{} 5.1]{ByszewskiKonieczny-2018-TAMS}. 

\begin{theorem}\label{prop:pisot}
Let $\beta > 1$ be a cubic Pisot unit with a pair of complex Galois conjugates $\a,\bar\a$ and let $E = \set{ \nint{\b^i} }{ i \in \NN_0}$. Then $\bb{1}_E$ is a {\gpword}.
\end{theorem}

\begin{proof}
We will devise a procedure that verifies if, for a given integer $n \in \ZZ$, we have
\begin{equation}\label{eq:n=b^i}
n = \nint{\b^i} \text{ for some } i \in \NN_0\,.
\end{equation}
Later, we will explain how this procedure can be encoded using a generalised polynomial formula. 
Throughout, we assume that $\abs{n}$ is sufficiently large, which we may do because {\gp} sets are closed under finite modifications.
Set $L = \QQ(\b,\a,\bar\a)$ and $K = \QQ(\b)$, and let $p(x) = x^3 - ax^2 -bx-1$ denote the minimal polynomial of $\b$.

Suppose for a moment that \eqref{eq:n=b^i} holds, and hence in particular $n > 0$. Our first goal is to compute $\beta^i$ as a {\gp} function of $n$. Since $\b^i + \a^i + \bar\a^i$ is an integer and $\abs{\a^i} = \beta^{-i/2} \ll 1/\sqrt{n}$, 
we see that 
\begin{equation}\label{eq:374:00}
 n = \b^i + \a^i + \bar\a^i \,.
\end{equation}
Similarly, for every $j \in \NN$, we have 
\[
	\b^j n = {\b^{i+j} + \a^{i+j} + \bar\a^{i+j}} + O(\b^j/\sqrt{n})\,,
\]
and, as a consequence, bearing in mind that $n$ is large enough, we obtain
\begin{align}
\label{eq:374:01}
	\nint{\b n} &= \b^{i+1} + \a^{i+1} + \bar\a^{i+1}\\
\label{eq:374:02}
	\nint{\b^2 n} &= \b^{i+2} + \a^{i+2} + \bar\a^{i+2}\,.
\end{align}
Thus, we have expressed $(n,\nint{\beta n},\nint{\beta^2 n})$ as a linear function of $(\beta^i,\alpha^i,\bar\alpha^i)$. As a consequence, we can compute $(\beta^i,\alpha^i,\bar\alpha^i)$ by solving a system of linear equations involving $(n,\nint{\beta n},\nint{\beta^2 n})$ and algebraic coefficients belonging to $L$.

For each $n \in \ZZ$, let $g(n)$, $h(n)$, and $h^*(n)$ be the solution to the system of equations
\begin{align}
\label{eq:374:10}	 g(n) + h(n) + h^*(n) &= n\\
\label{eq:374:11}	g(n)\b + h(n)\a + h^*(n)\bar\a  &= \nint{\b n}\\
\label{eq:374:12}	 g(n)\b^2 + h(n)\a^2 + h^*(n)\bar\a^2 &= \nint{\b^2 n}\, . 
\end{align}
Note that \eqref{eq:374:10}--\eqref{eq:374:12} is non-singular, so $g(n)$, $h(n)$, and $h^*(n)$ are well-defined and unique. 
In fact, one can explicitly compute that 
\begin{align}
\label{eq:374:20}	 g(n) &= \frac{\nint{\b n}(\a+\bar\a)-\a\bar\a n-\nint{\b^2 n}}{(\b-\a)(\b-\bar\a)}
= \frac{\nint{\b n}(a-\b)-n/\b-\nint{\b^2 n}}{2\b^2 -a\b+1/\b} \\
\label{eq:374:21}	 h(n) &= \frac{\nint{\b n}(\b+\bar\a)-\b\bar\a n-\nint{\b^2 n}}{(\a-\bar\a)(\a-\b)}
= \frac{\nint{\b n}(a-\a)-n/\a-\nint{\b^2 n}}{2\a^2 -a\a+1/\a} \\
\label{eq:374:22}	 h^*(n) &= \frac{\nint{\b n}(\a+\b)-\a\b n-\nint{\b^2 n}}{(\bar\a-\a)(\bar\a-\b)}
= \frac{\nint{\b n}(a-\bar\a)-n/\bar\a-\nint{\b^2 n}}{2\bar\a^2 -a\bar\a+1/\bar\a}\,\cdot
\end{align}
The key reason for our interest in $g,h,h^*$ is that, if \eqref{eq:n=b^i} holds, then it follows from the discussion above that $g(n) = \beta^i$, $h(n) = \a^i$ and $h^*(n) = \bar\a^i$. Indeed, $(\beta^i,\alpha^i,\bar\alpha^i)$ is a solution to \eqref{eq:374:10}--\eqref{eq:374:12}, and the solution to \eqref{eq:374:10}--\eqref{eq:374:12} is unique.

It is apparent from formulae \eqref{eq:374:20}--\eqref{eq:374:22} (or from the symmetries of \eqref{eq:374:10}--\eqref{eq:374:12}) 
that $h^*(n) = \bar{h(n)}$ and $g(n) \in K = \QQ(\b)$, $h(n) \in \QQ(\a)$ for all $n \in \ZZ$. We also see that $g(n)$ is a linear combination 
of $n$, $\nint{\b n}$, and $\nint{\b^2 n}$. Hence $g$ is a {\gp} map. For the same reason, $\Re(h)$ and $\Im(h)$ are {\gp} maps\footnote{Note that we have not introduced the notion of a {\gp} map $\NN_0 \to \CC$. However, if one identifies $\CC$ with $\RR^2$ in the standard way, then $h$ is a {\gp} map under this identification.}.

Our next goal is to express condition \eqref{eq:374:00} in terms of the maps $g,h,h^*$.  By the Dirichlet unit theorem, we know that the group of units of $\cO_K$ has rank $1$. Since $\beta$ is a unit, it generates a group that has finite index in the group of all units of $\cO_K$. Assume for now that $\beta$ is a fundamental unit, meaning that all units in $\cO_K$ take the form $\pm \beta^i$ for $i \in \ZZ$. We address the general case at the end of the proof. 

Recall that if \eqref{eq:n=b^i} holds for some $n \in \ZZ$, then $g(n) = \b^i$ and $h(n) = \a^i$. Thus, $g(n) \in \cO_K$ and $g(n)h(n)h^*(n) = 1$. If, additionally, $\abs{n}$ is sufficiently large, then $g(n) \in [n-1,n+1)$.  
Suppose, conversely, that for some $n \in \ZZ$ we have $g(n) \in \cO_K$ and $g(n)h(n) h^*(n) = 1$. 
Then $g(n)$ is a unit, since its norm $N_{L/\QQ}(g(n)) = g(n)h(n)h^*(n)$ is equal to $1$. As a consequence, we have $g(n) = \pm \b^i$ for some $i \in \ZZ$. 
If additionally $g(n) \geq n-1$ then $g(n) = \b^i$ for some $i \in \NN$. Thus, for all but finitely many $n \in \ZZ$, \eqref{eq:n=b^i} is equivalent to
\begin{align}\label{cond:521:2} 
 g(n) \in \cO_K\,,\quad g(n)h(n)h^*(n) &= 1\,, &&\text{and}& g(n) &\in [n-1,n+1)\,.
\end{align}

Our final goal is to express the conditions in \eqref{cond:521:2} in terms of generalised polynomials. As a first step in this direction, we will need a more precise description of $g,h,h^*$. 
Since $g(n) \in \QQ(\b)$, for each $n \in \ZZ$, there exists a decomposition
\begin{equation}\label{eq:374:30}
g(n) = u(n) + v(n)\b + w(n)\b^2\,,
\end{equation}
where $u(n)$, $v(n)$, and $w(n) \in \QQ$. If $\sigma \in \mathrm{Gal}(L/\QQ)$ is an automorphism with $\sigma(\b) = \a$, then 
$\sigma(g(n)) = h(n)$, and as a consequence we also have
\begin{align}\label{eq:374:31}
h(n) &= u(n) + v(n)\a + w(n)\a^2\\
h^*(n) &= u(n) + v(n)\bar\a + w(n)\bar\a^2\,.
\end{align}

Arguing along similar lines as above, we can express $u(n)$, $v(n)$, and $w(n)$ as linear combinations of $g(n)$, $h(n)$, and $h^*(n)$, 
and hence also as a linear combination of $n$, $\nint{\b n}$, and $\nint{\b^2 n}$. For instance,
\begin{align*}
	u(n) = \frac{1}{\Delta}\Big(& 
	\bra{-2 a^3+a^2 b^2-10 a b+4 b^3-9}n   
	\\&+\bra{a^3 b-a^2+4 a b^2+6 b}\nint{\b n} +  
	\bra{-a^2 b +3 a-4 b^2}\nint{\b^2 n} 
	\Big)\,,
\end{align*}
where $\Delta = -4 a^3+a^2 b^2-18 a b+4 b^3-27$ is the discriminant of $p$.
In particular, $u$, $v$, and $w$ are {\gp} maps.

Since $g$, $\Re(h)$ and $\Im(h)$ are generalised polynomials, 
it follows that
$$
E_1 = \set{n \in \ZZ}{g(n)h(n)h^*(n) = 1} \text{ and } 
E_2 = \set{n \in \ZZ}{g(n)-n \in [-1,1)}
$$
are {\gp} sets. Let us consider the set 
\[ 
\Lambda = \set{(x,y,z) \in \QQ^3}{ x + y\b + z\b^2 \in \cO_K} \,.
\] 
Clearly, $\Lambda$ is a lattice. It follows that $\Lambda$ is a {\gp} subset of $\RR^3$. Indeed, if $\Lambda = A \ZZ^3$ for a matrix $A \in \mathrm{GL}(3;\RR)$, 
then 
\[
	\Lambda = \set{x \in \RR^3}{ \fp{ \bra{A^{-1}x}_1} + \fp{ \bra{A^{-1}x}_2} + \fp{ \bra{A^{-1}x}_3} = 0}\, .
\]
Recalling that $u$, $v$, and $w$ are {\gp} maps, we conclude that
\[
	E_3 = \set{n \in \ZZ}{ g(n) \in \cO_K} = \set{n \in \NN_0}{ (u(n),v(n),w(n)) \in \Lambda }
\]
is a {\gp} set. Hence, $E_1 \cap E_2 \cap E_3$ is a {\gp} set  by Proposition \ref{prop:cl:sets}, and it is precisely the set of all $n \in \ZZ$ 
which satisfy \eqref{cond:521:2}, as needed.

Let us now consider the case where the fundamental unit of $K$ is some $\tilde\beta \neq \pm \beta$. 
We may assume that $\tilde \beta > 0$. Since $\beta$ is a unit, we have $\beta = \tilde\beta^k$ for some $k \in \NN$. 
It is straightforward to verify that $\tilde\beta$ is again a cubic Pisot unit with a pair of complex conjugates, say $(\gamma, \bar{\gamma})$, 
so the above discussion applies with $\tilde\beta$ in place of $\beta$. Set $\tilde n_i = \tilde\b^i + \gamma^i + \bar{\gamma}^i$, for every $i \in \NN_0$. 
Then $(\tilde n_i)_{i=0}^\infty$ satisfies a linear recurrence relation, and, for all sufficiently large $i$, we have 
$\tilde{n_i} = \nintnormal{\tilde\b^i}$ and $\tilde n_{i+1} = \nintnormal{\tilde\beta \tilde n_i}$. Furthermore, 
$\tilde E = \set{ \tilde n_i}{i \in \NN_0}$ is a {\gp} subset of $\NN_0$. 
One can find an integer $m$ such that the sequence $(\tilde n_i \bmod{m})_{i=0}^\infty$ is periodic with minimal period $\ell$ that is a 
multiple of $k$ \cite{Ward-1933}
\footnote{More precisely, by \cite[Thm.\ 1]{Ward-1933}, it is enough to consider the case where $k$ is a  power of a prime $p$. By \cite[Thm.\ 10.2 \& Cor.\ 3]{Ward-1933}, the period of $(\tilde n_i \bmod{p^N})_i$ is a multiple of the order of $\tilde \beta$ in $\mathbb{F}_q[\tilde\beta]^{\times}$, where $q = p^{N-C}$ for a constant $C$. In particular, the period of $(\tilde n_i \bmod{p^{N}})_i$ is a multiple of $k$ for sufficiently large $N$.}
. Note that, for each sufficiently large $i$ and for each $j$, $0 \leq j < \ell$, we have
\[
	\nint{\tilde \b^j \tilde n_i} \equiv \tilde n_{i+j\bmod \ell} \bmod m \,.
\]
This allows us to identify, for each residue class $r$, $0 \leq r < \ell$, the elements $n \in \tilde E$ which take the form 
$n = \tilde n_i$ with $i \equiv r \bmod \ell$.  With finitely many exceptions, these are precisely those $n \in \tilde E$ for which 
$\nint{\tilde \b^j \tilde n} \equiv \tilde n_{r+j\bmod \ell} \bmod m$ for all $j$, $0 \leq j < \ell$. It follows that the set $\set{\nint{\b^i}}{i \in \NN_0}$ 
differs by finitely many elements from 
\begin{equation}\label{eq:399:1}
	\set{ n \in \tilde E}{ (\exists\, 0 \leq a < \ell/k) \ (\forall\, 0 \leq j < \ell)\ \nintnormal{\tilde\b^j n} \equiv \tilde n_{ak+j} \bmod{m} }\, .
\end{equation}
Since the set in \eqref{eq:399:1} is a {\gp} set, and since the property of being a {\gp} set is preserved under finite modifications, 
we conclude that the set $\set{\nint{\b^i}}{i \in \NN_0}$ is also a {\gp} set.
\end{proof}

\begin{proof}[Proof of Theorem \ref{prop:pisot:rec}(ii)]
Follows immediately from Theorem \ref{prop:pisot} and \cite[Prop.{} 5.1]{ByszewskiKonieczny-2018-TAMS}.
\end{proof}

\subsection{Concluding remarks} 
Let us now briefly discuss potential generalisations of Theorem \ref{prop:pisot:rec}. Let $p(x) = x^d - \sum_{i=1}^d a_ix^{d-i} \in \ZZ[x]$ be a monic irreducible polynomial of degree $d \in \NN$ with a root $\beta > 1$ which is either a Pisot unit or a Salem number, and let $\a_1,\a_2,\dots,\a_{d-1}$ denote the remaining roots. 

Much of the reasoning in Theorem \ref{prop:pisot} carries through to this more general context, except that the group of units of $\cO_{K}$ now no longer has rank $1$ (here, $K = \QQ(\beta)$). Adapting the argument, we can hope to show that 
\[
	F = \set{ \operatorname{Tr}_{L/\QQ}(\mu) }{ \mu \text{ is a unit of } \cO_K} 
\]
is a {\gp} subset of $\ZZ$, where $L = \QQ(\b,\a_1,\dots,\a_{d-1})$ is the splitting field of $p$. 

Consider also the sequence given by $n_i = \mathrm{Tr}(\beta^i) = \beta^i + \sum_{k=1}^{d-1} \a_k^i \in \ZZ$, and note that $n_i$ obeys the linear recurrence
$$
	n_{i+d} = \sum_{j=1}^d a_j n_{i+d-j}, \qquad i \in \NN_0\,.
$$
Set $E = \set{ n_i}{i \in \NN_0} \cap \NN$. Note that the assumption on $\beta$ ensures that $n_i = \nintnormal{\b^i} + O(1)$ for all $i \in \NN_0$. 
In analogy with Theorem \ref{prop:pisot:rec}, one can ask if $E$ is a {\gp} subset of $\NN$. 

In the case where the group of units of $\cO_{K}$ has rank $1$, the set $F$ is essentially equal to $E$, which is one of the key observations behind the 
proof of Theorem \ref{prop:pisot}. In general, we have the inclusion $E \supset F$, but $E$ will usually be a proper subset of $F$ and there is no clear way of 
describing $E$ inside of $F$ with a {\gp} formula. 

In the case where $\beta$ is a Pisot number but not a unit, we have a negative result for $d = 1$: if $k \geq 2$ is and integer, then $E = \set{k^i}{i \in \NN_0}$ 
is not a {\gp} set. (We revisit this example in Section \ref{sec:neg}.) When $d \geq 2$, nothing is known, but, based on the previously mentioned result, 
it seems reasonable to expect a negative result. In the case where $\beta$ is neither a Pisot nor a Salem number ({\it i.e.}, either $\beta$ is transcendental or 
$\beta$ has a  Galois conjugate with absolute value greater than $1$) nothing is known, and the techniques discussed in this section appear not to 
be applicable. We end this section with the following two general problems. 

\begin{problem}
\begin{enumerate} 
\item  
Classify all sequences of positive integers $(n_i)_{i=1}^\infty$ satisfying a linear recurrence such that the characteristic word $\bb{1}_E$ of the set $E = \set{n_i}{i \in \NN_0}$ is a {\gpword}. 
\item 
Classify all real numbers $\beta > 1$ such that the characteristic word $\bb{1}_E$ of the set $E = \set{ \nint{\b^i} }{ i \in \NN_0}$ is a {\gpword}.
\end{enumerate}
\end{problem}

\section{Negative results}\label{sec:neg}

In this section we discuss the problem of proving that a given infinite word cannot be described by a generalised polynomial formula. 
We remark that, as a general principle, it is often harder to verify that a sequence does not admit a representation of a specified form than it is to find such a representation when it exists. This phenomenon is particularly prevalent from a computational point of view. For context, we also note that Allouche, Shallit, and Yassawi 
\cite{AlloucheShallitYassawi-2021} recently survey ways in which one can show that a sequence is not automatic.

\subsection{General conditions}

Let us now review criteria that can be used to show that a given infinite word $\bb a$ 
is not a {\gpword}. In principle, each ``positive'' result about {\gpwords} gives rise to such a criterion, and \emph{vice versa}. 
For each property $P$ which can apply to finitely-valued sequences, the implication ``\textit{If $\bb a$ is a {\gpword} then $P(\bb a)$}'' is 
tantamount to ``\textit{If $\neg P(\bb a)$ then $\bb a$ is not a {\gpword}}''. In practice, depending on the aesthetic appeal of the property $P(\bb a)$ and 
on the ease of verifying $\neg P(\bb a)$, one of these implications is more interesting than the other. 
The following proposition gathers several such properties that arise from positive results discussed earlier.

\begin{proposition}\label{prop:neg:long-list}
Let $\bb a$ be an infinite word defined over a finite alphabet $\Sigma$. Then $\bb a$ is not a {\gpword}
if one of the following conditions hold. 
\begin{enumerate}
\item\label{it:neg:long-list:0} There is $w \in \Sigma^*$ such that $\mathrm{freq}(\bb a,w)$ does not exist.
\item\label{it:neg:long-list:1} There is $\fword{w} \in \Sigma^*$ such that $\mathrm{freq}(\bb a,\fword{w}) > 0$ but $\mathrm{rec}(\bb a,\fword{w}) = \infty$.
\item\label{it:neg:long-list:2} The subword complexity of $\bb a$ satisfies $\limsup_{n\to \infty}\log(p_{\bb a}(N))/\log N = \infty$.
\item\label{it:neg:long-list:3} There exists $x \in \Sigma$ with $\freq(\bb a,x) < 1$ and $t \colon \NN_0 \to \NN_0$ that is good for equidistribution 
on nilsystems (cf.{} Sec.{} \ref{ssec:facts:equidist}) such that $a_{t(n)} = x$ for all $n \in \NN_0$.
\end{enumerate}
\end{proposition}

\begin{proof}
These are immediate consequences of Theorem \ref{prop:facts:freq-exists}, Corollary \ref{cor:facts:freq-exists}, 
Theorem \ref{thm:A}, and Proposition \ref{prop:nil:equidistribution}.
\end{proof}

Another way to discriminate bracket words that we want to mention is related to periodicity. 
Following \cite{ByszewskiKonieczny-2020-CJM}, we say that an infinite word 
$\bb a$ is \emph{weakly periodic} if for every infinite arithmetic progression $\bra{kn+r}_{n =0}^\infty$ ($k \in \NN,\ r \in \NN_0$) 
there exist two distinct sub-progressions with equal steps $\bra{k'n+r'}_{n =0}^\infty$ and $\bra{k'n+r''}_{n =0}^\infty$ 
($k' \in \NN$ and $k \mid k'$, $r'$ and $r''$ belong to $\NN_0$, $r'\equiv r'' \equiv r \bmod k$, and $r' \neq r''$) 
such that the corresponding restrictions of $\bb a$ are the same: $ a_{k'n +r'} = a_{k'n+r''}$ for all $n \in \NN_0$. 
For instance, all automatic sequences are weakly periodic, all Toeplitz sequences are weakly periodic, and the characteristic sequence of 
the square-free integers is weakly periodic.\footnote{A sequence $\bb a$ is a Toeplitz sequence if for each $n$ there exists a period $q \in \NN$ 
such that $a_{n+qm} = a_{n}$ for all $m \in \NN$; 
see, for instance, \cite{Downarowicz-2005}  for background. An integer $n$ is square-free if, for every prime $p$, one has $p^2 \nmid n$.} 
We will also say that $\bb a$ is \emph{almost everywhere periodic} if there exists an infinite periodic word $\bb a'$ such that $a_{n} = a'_{n}$ for almost all $n$, that is, $d\bra{\set{n \in \NN_0}{a_{n} \neq a'_{n}}} = 0$. The following result is a rephrasing of {\cite[{Thm.{} 2.6}]{ByszewskiKonieczny-2020-CJM}}.

\begin{proposition}[{\cite[{Thm.{} 2.6}]{ByszewskiKonieczny-2020-CJM}}]\label{prop:neg:weakly-periodic}
	Let $\bb a$ be an infinite word that is weakly periodic but not almost everywhere periodic. Then $\bb a$ is not a {\gpword}.	
\end{proposition}

The proof relies on Theorem \ref{thm:BL-mini} and the fact that each nilsystem $(X,T)$ can be partitioned into a finite number of components, 
$X = X_1 \cup X_2 \cup \dots \cup X_d$, such that $(X_i,T^d)$ is a totally minimal\footnote{A topological dynamical system $(X,T)$ is said to be totally minimal if 
$(X,T^d)$ is minimal for every $d\geq 1$.} dynamical system for all $i$, $1 \leq i \leq d$. 
Next, we observe that, for a totally minimal dynamical system $(X,T)$, $x \in X$, open $S \subset X$, $k \in \NN$, $r,r' \in \NN_0$ with $r \neq r'$, 
if for all $n \in \NN_0$ we have the implication 
\[T^{kn+r}(x) \in S \Rightarrow T^{kn+r'}(x) \in S\,,\] then $S$ is either empty or dense in $X$ (see \cite[Lem.{} 2.4]{ByszewskiKonieczny-2020-CJM}).

\begin{example}
	The indicator function of the square-free integers is not a {\gpword} (which can also be derived from item (iii) of Proposition \ref{prop:neg:long-list}). 
	A non-periodic Toeplitz sequence is not a {\gpword} (note that a Toeplitz sequence is periodic if and only if it is almost everywhere periodic). 
	It was shown in \cite{CassaigneKarhumaki-1997} that a large class of Toeplitz sequences have polynomial subword complexity, so item (iii) of Proposition \ref{prop:neg:long-list} 
	cannot be applied in this case, nor can item (ii) Proposition \ref{prop:neg:long-list}.
\end{example}

As a counterpart to Theorem \ref{thm:nil:IP*-rec},  
it is shown in \cite{ByszewskiKonieczny-2018-TAMS}  that any {\gp} subset $E \subset \NN_0$ with $d(E) = 0$ is very poor in terms of additive structure. 
Here, we cite a slightly weaker (but more succinct) variant of {\cite[{Thm.{} A}]{ByszewskiKonieczny-2018-TAMS}}. 

\begin{proposition}[{\cite[{Thm.{} A}]{ByszewskiKonieczny-2018-TAMS}}]\label{prop:neg:sparse-IP}
	Let $E \subset \NN_0$ be a set. Suppose that $d(E) = 0$ and $E$ contains an $\IP$ set. Then $E$ is not a {\gp} set.	
\end{proposition}

\begin{corollary}\label{cor:neg:sparse-IP}
	Let $\bb a = (a_n)_{n=0}^\infty$ be an infinite word over a finite alphabet $\Sigma$. Suppose that there exists $x \in \Sigma$ and an $\IP$ set $E \subset \NN$ such that $\freq(\bb a, x) = 0$ and $a_n = 0$ for all $n \in E$. Then $\bb a$ is not a {\gpword}.
\end{corollary}

\begin{example}
	The characteristic word of the set $ \set{\sum_{n \in I} n! }{ I \subset \NN,\ \abs{I} < \infty}$ is not a {\gpword}. 
\end{example}

Finally, we cite a result recently obtained by the second-named author \cite[{Thm.{} A}]{Konieczny-2021-JLM}, as a final ingredient needed to finish the classification of automatic {\gpwords}, cf.\ Theorem \ref{thm:BK}. 
Recall that a set $E \subset \NN_0$ is \emph{thick} if its complement is not syndetic, or equivalently if for each $\ell$ there exists $n$ with $n,n+1,\dots, n+\ell \in E$. 

\begin{theorem}[{\cite[{Thm.{} A}]{Konieczny-2021-JLM}}]\label{thm:neg:powers}
	Let $k \geq 2$ be an integer, let $E \subset \NN_0$, and put
	\[
		F = \set{ m \in \NN}{ m k^n \in E \text{ for infinitely many $n \in \NN_0$}}. 
	\]
	If $\set{k^n}{n \in \NN_0} \subset E$ and $\NN \setminus F$ is thick then $E$ is not a {\gp} set.
\end{theorem}

\begin{example}
	Let $k \geq 2$ be an integer. Then the characteristic word of the set $\set{k^n}{n \in \NN_0}$ is not a {\gpword}.
\end{example}

For a set $E \subset \NN_0$ and $k \in \NN$, we set $E/k = \set{n \in \NN_0}{kn \in E}$. Restricting our attention to sets $E \subset \NN_0$ with $E/k = E$, we obtain a cleaner statement, which can be phrased either in terms of {\gp} sets or {\gpwords}.

\begin{corollary}\label{cor:neg:powers}
	Let $k \geq 2$ and $E \subset \NN$. If $E \neq \emptyset$, $d(E) = 0$, and $E/k = E$, then $E$ is not a {\gp} set.
\end{corollary}
\begin{proof}
	Pick $n_0 \in E$ and let $E' = E/n_0$. Then $d(E') = 0$, $k^n \in E'$ for all $n$, and 
	\[ 
		F' = \set{ m \in \NN}{ m k^n \in E' \text{ for infinitely many $n \in \NN_0$}} = E'. 
	\]
	Since $d(F') = 0$, the set $\NN \setminus F'$ is thick and we can apply Theorem \ref{thm:neg:powers}.
\end{proof}

\begin{corollary}\label{cor:neg:powers-set}
	Let $\bb a = (a_n)_{n=0}^\infty$ be an infinite word over a finite alphabet $\Sigma$, let $k \geq 2$ be an integer and let $x \in \Sigma$. Suppose that
\begin{inparaenum}[(i)]
\item $a_{kn} = a_{n}$ for all $n \in \NN$,
\item $\freq(\bb a; x) = 0$, and 
\item $a_n = x$ for at least one $n \in \NN$.
\end{inparaenum}
Then $\bb a$ is not a {\gpword}.
\end{corollary}

\subsection{Primes and squares}

Let us consider two standard examples: the primes $\cP = \{2,3,5,7,\dots\}$ and the squares $\cS = \{0,1,4,9,\dots\}$. 
In both cases, we can show that the corresponding characteristic word  is not a {\gpword}. 
The arguments are short and rely heavily on theorems from nilpotent dynamics.

\begin{proposition}\label{prop:neg:primes-not-gp}
	The characteristic word ${\bb 1}_{\cP}$ of the set of primes is not a {\gpword}.
\end{proposition}
\begin{proof}
	As a special case of \cite[Theorem 5.2]{GreenTao-2012-Mobius}, for each {\gp} map $g \colon \NN_0 \to [0,1]$ and for all sufficiently large integers $N$, 
	we have
	\begin{equation}\label{eq:863:1}
		\abs{\EEE_{n < N} \mu(n) g(n)} = O( \log^{-2} N)\,,
	\end{equation}
	where we let $\mu(n)$ denote the M\"{o}bius function ({\it i.e.}, $\mu(n) = 0$ if $n$ is divisible by a square and $\mu(n) = (-1)^k$ if $n$ is 
	the product of $k$ different primes), and the implicit constant is allowed to depend on $g$. 
	If $\bb{1}_{\cP}$ was a {\gpword}, then \eqref{eq:863:1} would imply that 
	\begin{equation}\label{eq:863:2}
		{\abs{\cP \cap [N]}} =O( N/\log^{2} N)\, ,
	\end{equation}
	which would contradict the Prime Number Theorem.
\end{proof}

\begin{remark}
	 Note, in particular, that the cited theorem implies that $\mu$ is not a {\gpword}. The same applies, with virtually the same proof, 
	 to the Liouville function, which is defined by $\lambda(n) = (-1)^k$ if $n$ is the product of $k$ primes (counting multiplicities). 
	 Results from \cite{FrantzikinakisHost-2017} allow one to extend this observation to a wider class of multiplicative functions. 
	 In a different direction, using the results in \cite{GreenTao-2012-Mobius}, one should be able to strengthen the result of 
	 Proposition \ref{prop:neg:primes-not-gp} by showing that for each bounded {\gp} map $g \colon \NN_0 \to \RR$, the averages 
	 $\EE_{p \in \cP \cap [N]} g(p)$ converge to the same limit as the averages $\EE_{n \in [N]} g(n)$ as $N \to \infty$.
\end{remark}

We next consider the set of squares. 

\begin{proposition}\label{prop:neg:squares-not-gp}
	The characteristic word ${\bb 1}_{\cS}$ of the set of squares is not a {\gpword}.
\end{proposition}

This follows immediately from the following more general result.

\begin{proposition}
	Let $p \colon \ZZ \to \ZZ$ be a polynomial with $p(\NN_0) \subset \NN_0$ and $\deg p \geq 2$. 
	Then the set $ \set{p(n)}{ n \in \NN_0}$ is not a {\gp} subset of $\NN_0$.
\end{proposition}

\begin{proof}
	We will prove marginally more, namely that if $p_1,p_2,\dots,p_s \colon \ZZ \to \ZZ$ are polynomials 
	with $p_j(\NN_0) \subset \NN_0$ and $\deg p_j \geq 2$ for all $j$, $1 \leq j \leq r$, then the set
	\begin{equation}\label{eq:254:1}
		E = \textstyle \bigcup_{j=1}^r p_j(\NN_0)
	\end{equation}
	is not a {\gp} subset of $\NN_0$. For the sake of contradiction, suppose that the converse is true. 
	Throughout the proof, we use terminology introduced in Appendix \ref{app:nil}.
	
	Applying Theorem \ref{thm:BL-poly}, we can find a nilmanifold $X = G/\Gamma$ with $G$ connected and simply connected, a
	 semialgebraic subset $S \subset X$, and a polynomial sequence $g \colon \ZZ \to G$ such that $E = \set{n \in \NN_0}{g(n)\Gamma \in S}$.
	
	Let us first consider the case where $(g(n)\Gamma)_{n=0}^\infty$ is equidistributed in $X$. Since $d(E) = 0$, we have $\mu_{X}(S) = 0$. 
	The polynomial sequence $\bra{ \pi_{\mathrm{ab}}(g(n)\Gamma) }_{n=0}^\infty$ is equidistributed in $X_{\mathrm{ab}}$. 
	Hence, we can conclude from Weyl's equidistribution theorem that $\bra{ \pi_{\mathrm{ab}}(g(p(n)\Gamma)) }_{n=0}^\infty$ is 
	equidistributed in $X_{\ab}$ for each non-constant polynomial map $p \colon \NN_0 \to \NN_0$. 
	It now follows from Theorem \ref{thm:Leib-nil} (that is,  \cite[Thm.{} C]{Leibman-2005}) that $\bra{ g(p(n)\Gamma) }_{n=0}^\infty$ is 
	equidistributed in $X$. In particular $g(p(n))\Gamma \not \in S$ for almost all $n \in \NN_0$. Taking $p = p_1$ we reach a contradiction.	
	
	 Suppose next that $(g(n)\Gamma)_{n=0}^\infty$ is not equidistributed in $X$. Then, it follows from \cite[Thm.{} B]{Leibman-2005} 
	 that there exists $q \in \NN$ and sub-nilmanifolds $Y_0,Y_1,\dots,Y_{q-1} \subset X$ such that $\bra{g(qn+i)\Gamma}_{n=0}^\infty$ is 
	 equidistributed in $Y_i$ for each $i \in [q]$. It remains to apply the previously considered special case to the sets $(E-i)/q$ for each $i \in [q]$ 
	 (note that these sets again take the form \eqref{eq:254:1}, possibly for some larger $r$). 
\end{proof}

\subsection{Number-theoretical functions}

Number theory provides a plentiful source of examples of finitely-valued sequences, for which one can inquire into the existence of a 
generalised polynomial representation. Multiplicative functions constitute a particularly interesting and well-studied class of sequences, with many applications to other problems. 
Recall that a sequence $f \colon \NN \to \CC$ is \emph{multiplicative} if $f(nm) = f(n) f(m)$ 
for all $n,m \in \NN$ with $\gcd(n,m) = 1$. For instance, the M\"{o}bius function, mentioned in the previous sub-section, is often used in the study of the prime numbers. The Prime Number Theorem is equivalent to $\frac{1}{N} \sum_{n=1}^N \mu(n) \to 0$ as $N \to \infty$, while the more quantitative bound $\abs{\sum_{n=1}^N \mu(n)} = O(N^{1/2+\e})$ for each $\e > 0$ is equivalent to the Riemann hypothesis. The results from \cite{GreenTao-2012-Mobius} on correlations of the M\"{o}bius function were a crucial ingredient in the work of Green and Tao on linear patterns in the primes \cite{GreenTao-2010b}. As we pointed out earlier, the M\"{o}bius function is not a {\gpword}, and the same applies to the closely related Liouville function.

While the M\"{o}bius function and the Liouville function are bounded, there are many interesting examples of unbounded integer-valued multiplicative functions. In order to obtain finitely-valued sequences, we use reduction modulo a fixed integer. 
Thus, in this section we investigate infinite words of the form $(\bra{f(n)\bmod{q}}_{n=1}^\infty$, 
where $q \geq 2$ is an integer and $f \colon \NN \to \ZZ$ is multiplicative. 
Analogous questions, with automatic sequences in place of {\gpwords}, were investigated in \cite{Yazdani-2001}.
As an illustrative example, we begin with the Euler totient function  $\phi$, given by
\[
	\phi(n) = \abs{ (\ZZ/n\ZZ)^*} = n\prod_{\mathcal P\ni p \mid n}\bra{1-\frac{1}{p}} \,, \qquad n \in \NN\,.
\]

\begin{proposition}
	Let $q \geq 3$ be an integer. Then $\bra{\phi(n) \bmod q}_{n=1}^\infty$ is not a {\gpword}.
\end{proposition}

\begin{proof}
Let us assume by contradiction that $\bra{\phi(n) \bmod q}_{n=1}^\infty$ is  a {\gpword}.
Replacing $q$ with a divisor, we may assume that $q = 4$ or $q$ is a prime. 

We first assume that $q=4$. By direct inspection we have:
\[
	\phi(n) \bmod 4 =
	\begin{cases}
		1 & \text{if } n \in \{1,2\};\\
		2 & \text{if $n = 4$ or $n \in \{p,2p\}$ for a prime $p \equiv 3 \bmod 4$} ;\\
		0 & \text{otherwise}.
	\end{cases}
\]
Since $\bra{\phi(n) \bmod 4}_{n=1}^\infty$ is a {\gpword}, we deduce that $\cP \cap (4\ZZ+3)$ is a {\gp} set. 
This is impossible by direct repetition of the argument in Proposition \ref{prop:neg:primes-not-gp}. 

Suppose now that $q$ is a prime and set 
\[
	E = \set{ n \in \NN}{ \phi(n) \not \equiv 0 \bmod q}\, .
\]
 Alternatively, $E$ is described by the condition: $n \in E$ if and only if $q^2 \nmid n$ and $n$ has no prime divisors congruent to $1$ modulo $q$. 
 Hence, applying the prime number theorem in arithmetic progressions, we conclude that $d(E) = 0$. By assumption, $E$ is a {\gp} set. 
Pick a prime $p > q$ with $q \nmid p-1$. Then $E/p = \set{n \in \NN}{pn \in E} = E$, and hence we infer from Corollary \ref{cor:neg:powers} 
that $E$ is not a  {\gp} set. This provides a contradiction.
\end{proof}

Applying the same ideas, we can obtain the following result, which can be used to deal with most other ``naturally occurring'' multiplicative sequences. 

\begin{proposition}\label{prop:neg:multi}
	Let $f \colon \NN \to \ZZ$ be a multiplicative sequence and let $q \in \NN$. Suppose that the two following properties hold. 
\begin{enumerate}
\item\label{it:138:2}  There exists $p\in \cP$ such that the sequence $\bra{f(p^n) \bmod q}_{n=1}^\infty$ is eventually periodic but not eventually zero.
\item\label{it:138:3} There exist infinitely many  $p \in \cP$ such that $f(p^n) \equiv 0 \bmod q$ for some  $n \in \NN$.
\end{enumerate}
Then $\bra{f(n) \bmod q}_{n=1}^\infty$ is not a {\gpword}.
\end{proposition}

\begin{proof}
We argue by contradiction, assuming that $\bra{f(n) \bmod q}_{n=1}^\infty$ is a {\gpword}. 
	Using Lemma \ref{lem:cl:code}, we deduce that $\bra{f(n) \bmod p}_{n=1}^\infty$ is also a {\gpword} for every $p \mid q$. 
	Replacing $q$ with a prime divisor if necessary, we may thus assume that $q$ is prime.
	 Consider the {\gp} set 
	 \[
	E = \set{ n \in \NN}{ f(n) \not \equiv 0 \bmod q}\, .
	\]
Let $p$ be a prime satisfying the conditions in \ref{it:138:2}. Let $c \in \NN_0$ and $d \in \NN$ be such that 
$f(p^{n+d}) \equiv f(p^{n}) \bmod q$ for all $n \geq c$ and $f(p^c) \not \equiv 0 \bmod q$. 
Let $E'$ be the {\gp} set defined by $E' = E/p^c = \set{n \in \NN}{p^c n \in E}$. Then $E'/p^d = E'$. 

Next, we show that $d(E) = 0$. Suppose, conversely, that $d(E) > 0$. Then, by Theorem \ref{thm:facts:dens-exists}, 
$E$ is syndetic, meaning that we can find $N$ such that $E$ intersects any interval $[n,n+N)$ ($n \in \NN$). 
By \ref{it:138:3}, we can pick pairwise coprime integers $r_0,r_1,\dots,r_{N-1}$ with $f(r_i) \equiv 0 \bmod q$ for all $i \in [N]$. 
By the Chinese remainder theorem, we can find $n \in \NN$ such that $n+i \equiv r_i \bmod r_i^{2}$ for all $i \in [N]$. 
Then we have $f(n+i) \equiv 0 \bmod q$ and $n+i \not \in E$ for all $i \in [N]$, contradicting the defining condition of $N$. 
Hence $d(E) = 0$. 

Since $d(E) = 0$, we also have $d(E') = 0$. Thus, we infer from Corollary \ref{cor:neg:powers} that $E'$ is not a {\gp} set, which contradicts 
previous observations.
\end{proof}

\begin{example}
Given an integer $k \geq 0$, we let $\sigma_k \colon \NN \to \NN$ denote the $k$-th power divisor-sum function defined by
\[
	\sigma_k(n) = \sum_{d \mid n} d^k = \prod_{p\mid n} \frac{p^{k(\nu_p(n)+1)}-1}{p^k-1}\,, \qquad n \in \NN\,,
\]
where the sum runs over all positive integers $d$ that divide $n$. Note that $\sigma_0(n) = d(n)$ is the number of divisors of $n$ 
and $\sigma_1(n) = \sigma(n)$ is the sum of divisors. For each integer $q \geq 2$, the assumptions of Proposition \ref{prop:neg:multi} are satisfied: 
if $k \geq 1$ then in \ref{it:138:2} we can take $p=q$, and  in \ref{it:138:3} we can take any $p \equiv 1 \bmod{q}$ and $n = 1$ 
(infinitude of such primes follows from Dirichlet's theorem). If $k = 0$, both \ref{it:138:2} and \ref{it:138:3} hold for all primes $p$. 
Hence, $\bra{\sigma_k(n) \bmod q}_{n=1}^\infty$ is not a {\gpword}.
\end{example}

\begin{example}
Given an integer $k \geq 1$, we let $\tau_k \colon \NN \to \NN$ denote the number of representations as the product of $k$ positive integers, that is 
\[
	\tau_k(n) = \abs{ \set{(d_1,d_2,\dots,d_k) \in \NN^k}{ \prod_{i=1}^k d_i = n} } = \prod_{p\mid n} \binom{k+\nu_p(n)-1}{k-1}\,, \qquad n \in \NN\,.
\]
For each integer $q \geq 2$, the assumptions of Proposition \ref{prop:neg:multi} are satisfied: both \ref{it:138:2} and \ref{it:138:3} hold for all primes $p$.
Hence, $\bra{\tau_k(n) \bmod q}_{n=1}^\infty$ is not a {\gpword}.
\end{example}

\begin{example}
	Let $\tau \colon \NN \to \ZZ$ denote the Ramanujan function, and let $q$ be an integer with $\gcd(q,6) = 1$. For each prime $p$, the sequence $\bra{\tau(p^n)}_{n=1}^\infty$ satisfies the recurrence
	\[
		\tau(p^{n+2}) = \tau(p) \tau(p^{n+1}) - p^{11}\tau(p^n)\,,
	\] 
	and hence the sequence $\bra{\tau(p^n) \bmod q}_{n=1}^\infty$ is eventually periodic. In particular, since $\tau(2) = -24 \not\equiv 0 \bmod q$, in Proposition \ref{prop:neg:multi} assumption \ref{it:138:2} is satisfied for $p = 2$. It is known that the set of $n \in \NN$ with $\tau(n) \not\equiv 0 \bmod q$ has asymptotic density $0$ \cite{Ramanujan}. As a consequence, assumption \ref{prop:neg:multi}\ref{it:138:3} is also satisfied (otherwise, there would exist $P \in \NN$ such that $\tau(n) \not \equiv 0 \bmod q$ for all $n \in P\NN + 1$). Hence, $\bra{\tau(n) \bmod q}_{n=1}^\infty$ is not a {\gpword}.
\end{example}

We close this section on multiplicative functions with another result, Corollary \ref{cor:prop:neg:fin-primes-multi}, which applies to multiplicative sequences that exhibit 
non-trivial behaviour only for a finite set of primes. In particular, these sequences do not satisfy condition (ii) of Proposition \ref{prop:neg:multi}. In fact, we can state a somewhat more general result. For a prime $p$ and an integer $n$, we let $\nu_p(n)$ denote the $p$-adic valuation of $n$.  

\begin{proposition}\label{prop:neg:fin-primes}
	Let $\ell \in \NN$, let $p_1,p_2,\dots,p_{\ell}$ be primes, let $\Sigma$ be a finite alphabet, and let $F \colon \NN_0^\ell \to \Sigma$. Let $\bb a$ be the infinite word defined 
	over $\Sigma$ by
	\[ a_n = F\bra{\nu_{p_1}(n),\nu_{p_2}(n), \dots,\nu_{p_{\ell}}(n)} , \qquad n \in \NN. \]
Suppose that $\bb a$ is not almost everywhere periodic. Then $\bb a$ is not a {\gpword}.
\end{proposition}
\begin{proof}
	By Proposition \ref{prop:neg:weakly-periodic}, it will suffice to check that $\bb a$ is weakly periodic. Pick any $k \in \NN,r \in \NN_0$. We can find $k'$ with $k \mid k'$ such that $\nu_{p_i}(k') > \nu_{p_i}(r)$ for all $1 \leq i \leq \ell$. Thus, $a_{k'n + r} = a_{k'n + k'+r}$ for all $n \in \NN_0$, so, with notation as in the definition of weak periodicity, 
	we may take $r' = r$ and $r'' = r+k'$.
\end{proof}

\begin{corollary}\label{cor:prop:neg:fin-primes-multi}
	Let $f \colon \NN \to \ZZ$ be a multiplicative sequence and let $q \in \cP$. 
	Suppose that $(f(n) \bmod q)_{n=1}^\infty$ is not periodic and that there exists $p_0 \in \cP$ such that $f(p^n) \equiv 1 \bmod q$ 
	for all primes $p \geq p_0$ and all $n \in \NN$. Then $\bra{f(n) \bmod q}_{n=1}^\infty$ is not a {\gpword}.
\end{corollary}

\begin{proof}
	Note that the sequence $(f(n) \bmod q)_{n=1}^\infty$ is Toeplitz, since, for every positive integer $n$, one has $f(n+dm) = f(n)$ for all $m \in \NN$, for any choice of  $d \in \NN$ 
	such that $\nu_p(d) > \nu_p(n)$ for all $p < p_0$. Note also that $(f(n) \bmod q)_{n=1}^\infty$ is not periodic, and thus it is also not almost everywhere periodic. It remains to apply Proposition \ref{prop:neg:fin-primes}.
\end{proof}

\begin{example}
	The sequence $\bra{ (-1)^{\nu_2(n) + \nu_3(n) }}_{n=1}^\infty$ is not a {\gpword}.
\end{example}

\subsection{Automatic sequences}\label{ssec:Negative-Automatic} 
Let us now consider the problem of classifying automatic sequences which are bracket words. 
To begin with, we very briefly recall the definition of an automatic sequence; for extensive background, see \cite{AlloucheShallit-book}. 

Let $k \in \NN$, let $\Sigma_k = \{0,1,\dots,k-1\}$ denote the base-$k$ alphabet, and let $\Omega$ be a finite set. 
A sequence $\bb a$ over $\Omega$ is $k$-automatic if there exists a finite $k$-automaton which computes $\bb a$. More explicitly, 
this means that there exist a finite set of states $Q$, a distinguished state $q_0 \in Q$, a transition function 
$\delta \colon Q \times \Sigma_k\to Q$, and  an output map $\tau \colon Q \to \Omega$ 
such that for each integer $n \in \NN_0$ with base-$k$ expansion $n = n_0 + kn_1 + k^2 n_2 + \dots + k^\ell n_{\ell}$ 
we have
\[
	a_n = 
	\tau\left(\delta\Big(
		\delta \big(
			\dots
			\delta( 	
				\delta(q_0, n_0)
			, n_1) 
			\dots
		, n_{\ell-1} \big)
	, n_{\ell} \Big)\right)\,.
\]

It is relatively easy to show that if an automatic sequence $\bb a = (a_n)_{n=0}^\infty$ coincides with a bracket word $\bb b = (b_n)_{n=0}^\infty$ almost everywhere, that is, 
if $$d\bra{\set{n \in \NN_0}{a_n \neq b_n}} = 0 \,,$$ 
then $\bb a$ must also be periodic almost everywhere ({\it i.e.}, there exists a periodic word $\bb c$ 
such that $d\bra{\set{n \in \NN_0}{a_n \neq c_n}} = 0$). Indeed, this result follows from Proposition \ref{prop:neg:weakly-periodic} combined with the fact that 
a $k$-automatic sequence has a finite $k$-kernel; see \cite{ByszewskiKonieczny-2020-CJM} for details. 
Here, the $k$-kernel of a sequence $\bb a = (a_n)_{n=0}^\infty$ is defined as 
the set of subsequences $\set{ \bra{a_{k^i n+r}}_{n=0}^\infty}{r,i \in \NN_0,\ r < k^i}$. 

As a consequence of this ``density $1$'' result, in order to classify all automatic sequences which are bracket words, 
it is enough to consider sequences $\bb a$ over $\{0,1\}$, with $\mathrm{\freq}(\bb a,1) = 0$. Equivalently, we let $\bb a = \bb 1_{E}$, 
where $E \subset \NN_0$ and $d(E) = 0$. In \cite{ByszewskiKonieczny-2018-TAMS}, we investigated such ``sparse'' bracket words and showed 
a slightly stronger variant of Proposition \ref{prop:neg:sparse-IP}. Combining it with facts concerning additive richness of automatic sets, 
a full classification was obtained in \cite{ByszewskiKonieczny-2020-CJM}, conditional on the conjecture that for $k \geq 2$ 
the set $\set{k^i}{i \in \NN_0}$ of powers of $k$ is not a {\gp} subset of $\NN_0$. This conjecture was finally proved in \cite{Konieczny-2021-JLM}, 
leading to the following result.

\begin{theorem}[{\cite[Thm.{} B]{Konieczny-2021-JLM}}]\label{thm:BK}
	Let $\bb a$ be an automatic sequence that is not eventually periodic. Then $\bb a$ is not a {\gpword}.
\end{theorem}

 Note that this theorem extends the classical result claiming  that a Sturmian word cannot be automatic.

\section{Subword complexity}\label{sec:subword}

Since generalised polynomials are defined by relatively simple formulae, it is natural to inquire if they also have low complexity when viewed 
from different perspectives. In particular, we investigate \emph{subword complexity}, that is, the number $p_{\bb a}(N)$ of different length-$N$ factors 
(or subwords) of an infinite word $\bb a$.  An overview of various results concerning subword complexity can be found in a number of surveys, 
such as \cite{CassaigneNicolas-2010} or \cite{Allouche-1994}, and connections with the theory of dynamical systems are discussed in \cite{Ferenczi-1999}.

\subsection{Background}

For an infinite word $\bb a=(a_n)_{n=0}^\infty$ defined over an alphabet $\Sigma$, the complexity function of $\bb a$ is the function that associates  
with each positive integer $n$ the positive integer  
\[ p_{\bb a}(N) = \abs{\set{ a_{n}a_{n+1}\cdots a_{n+N-1}}{n \in \NN_0}} \,.\]
The simplest sequences, in terms of subword complexity, are the eventually periodic ones: If $\bb a$ is eventually periodic then $p_{\bb a}(N)$ is bounded. 
Conversely, it was shown by Morse and Hedlund that if there exists $N$ with $p_{\bb a}(N) \leq N$ then $\bb a$ must be eventually periodic 
(see, for instance, \cite{MorseHedlund-1938} or \cite[Prop.{} 2]{Ferenczi-1999}). Thus, for any sequence $\bb a$, which is not eventually periodic, and every $N$ 
we have $p_{\bb a}(N) \geq N+1$. The Sturmian words, discussed in  Section \ref{ssec:nil:sturm}, are characterised by the property that $p_{\bb a}(N) = N+1$ for all $N$.
Recall that Sturmian words are, in particular, {\gpwords}. 
At the other extreme, if $\vert \Sigma\vert=k$, we have the upper bound $p_{\bb a}(N) \leq k^N$ for all $N$, and there exist sequences for which equality holds. (In fact, this is the case for almost all sequences with respect to the natural probability measure on $\Sigma^{\NN_0}$). 
Because of the elementary inequality $p_{\bb a}(N+M) \leq p_{\bb a}(N) p_{\bb a}(M)$, the limit
\[
	h(\bb a) = \lim_{N \to \infty} \frac{\log p_{\bb a}(N)}{N}
\]
exists for any finitely-valued sequence $\bb a$; its value is called the \emph{entropy} of $\bb a$, and is closely connected to the notion of 
topological entropy coming from the theory of dynamical systems.

\subsection{New result}

Because {\gpwords} can be represented using nilsystems as in Theorem \ref{thm:BL-word} and because nilsystems have zero entropy, 
one could show\footnote{Since we are about to prove a more precise estimate, we do not go into the details of the argument. 
If in Theorem \ref{thm:BL-word} the sets $S_i$ were open, then the conclusion would be an immediate consequence of standard facts about topological entropy. 
In general, the sets $S_i$ are not open, but their boundaries $\partial S_i$ have zero $\mu_{G/\Gamma}$-measure, which is sufficient for this application.} 
that $h(\bb a) = 0$ for each {\gpword} $\bb a$, meaning that $p_{\bb a}(N) = \exp({o(N)})$ as $N \to \infty$. Our main new result concerning the complexity 
of bracket words is a polynomial bound, given in Theorem \ref{thm:A} and restated below for the reader's convenience. 

\thmsubword*

\begin{remark}
We note that related results concerning topological complexity of nilsystems were obtained in 
\cite{HostKraMaass-2014}, but are not directly applicable in our context (the crucial obstacle is the fact that the representation in 
Theorem \ref{thm:BL-word} involves a partition into semialgebraic sets $S_i$ which are neither open nor closed, while the results 
in \cite[Sec.{} 3]{HostKraMaass-2014} are applicable to open covers).
\end{remark}

The proof of this result is carried out in Sections \ref{sec:sc-induction}, \ref{sec:sc-auxiliary}, and \ref{sec:sc-proof}. 

\subsection{Complementary results}

\newcommand{\C}{\lambda} 
In view of Theorem \ref{thm:A}, it is natural to inquire into more precise asymptotics for 
$p_{\bb a}(N)$, where $\bb a$ is a {\gpword}. Let
\[
	\C(\bb a) = \limsup_{N \to \infty} \frac{\log p_{\bb a}(N)}{\log N}
\]
be the smallest exponent such that $p_{\bb a}(N) \leq N^{\C(\bb a)+o(1)}$ as $N \to \infty$. 
Theorem \ref{thm:A} asserts that $\C(\bb a)$ is finite for all {\gpwords} $\bb a$. 
The exact value of $\C(\bb a)$ is known only in the simplest examples, such as the Sturmian words. 
Already in relatively simple cases, such as $a_n = \braif{ \fp{ \sqrt{2}n \ip{\sqrt{3} n}} > 1/4}$, it would be interesting to compute $\C(\bb a)$ exactly. 

Theorem 3.4 of \cite{HostKraMaass-2014} asserts that the topological complexity $S(\e,N)$ of a minimal nilsystem $(X,T)$ obeys the bounds 
$C(\e) N^c \leq  S(\e,N) \leq C'(\e) N^c$, where $c$ is the total commutator dimension of $X$ (see the original paper for the relevant definitions). 
Let $\bb a$ be a {\gpword} represented on the nilsystem $(X,T)$ as in Theorem \ref{thm:BL-word}. It seems plausible that, under mild additional assumptions, 
similar estimates might hold for $p_{\bb a}(N)$, implying in particular that $\C(\bb a) = c$. 

We prove now that $\lambda(\bb a)$ can be arbitrarily large. In fact, for each $d \in \NN$, we can construct a {\gpword} $\bb a$ with $\lambda(\bb a) = d$.

\begin{proposition}\label{prop: comp1}
	Fix $d \in \NN$, and $\a_1,\a_2, \dots, \a_d \in \RR$ such that $1,\a_1,\a_2,\dots,\a_d$ are $\QQ$-linearly independent. 
	For every $i$, $1 \leq i \leq d$, let $\bb a^{(i)}$ be the Sturmian word defined by $a^{(i)}_n = \ip{n\a_i} - \ip{(n-1) \a_i}$ for $n \in \NN_0$, and set 
	$\bb a = \prod_{i=1}^d \bb a^{(i)}$. Then $\bb a$ is a bracket word with $\lambda(\bb a) = d$.
\end{proposition}

\begin{proof} We first observe that, according to Corollary \ref{cor:cl:prod}, and since Sturmian words are bracket words, the word $\bb a$ is a bracket word. 

	For each $i$, $1 \leq i \leq d$, we have $p_{\bb a^{(i)}}(N) = N+1$. As a general fact about products, we have 
	$p_{\bb a}(N) \leq \prod_{i=1}^d p_{\bb a^{(i)}}(N)$. It  follows that $\lambda(\bb a) \leq \sum_{i=1}^d \lambda\bra{ \bb a^{(i)}} = d$. 
	
	Conversely, for each $i$, $1 \leq i \leq d$, and each length-$N$ subword $w^{(i)}$ of $\bb a^{(i)}$, there exists an interval (not degenerate to a point) 
	$I \subset \RR/\ZZ$ such that, for $n \in \NN_0$, $\bb a^{(i)}|_{[n,n+N)} = w^{(i)}$ if and only if $n \a_{i} \bmod \ZZ \in I$. Since by assumption the numbers 
	$1,\a_1,\a_2,\dots,\a_d$ are $\QQ$-linearly independent, the sequence $(n\a_1,n\a_2,\dots,n\a_d) \bmod \ZZ^d$ is equidistributed in $\RR^d/\ZZ^d$. 
	It follows that for any $d$-tuple $w^{(1)},w^{(2)},\dots,w^{(d)}$ of factors of $\bb a^{(1)}, \bb a^{(2)},\dots, \bb a^{(d)}$ respectively, 
	$\prod_{i=1}^d w^{(i)}$ is a factor of $\bb a$. Hence,
	$p_{\bb a}(N) \geq \prod_{i=1}^d p_{\bb a^{(i)}}(N)$, and consequently $\lambda(\bb a) \geq d$. This ends the proof. 
\end{proof}

We end this section with an explicit example showing that $\C(\bb a)$ takes on arbitrarily large values for {\gpwords} $\bb a$, 
even when the size of the alphabet of $\bb a$ remains bounded.

\begin{proposition}\label{prop: comp2}
	Fix $d \in \NN$. Let $\a \in \RR \setminus \QQ$ and let $\bb a$ be the {\gpword} over $\{0,1,\dots,9\}$ given by 
	$a_n = \ip{ 10 \fp{\a n^d} }$. Then $p_{\bb a}(N) \gg_d N^{d(d-1)/2}$ for all $N \in \NN$.
\end{proposition}

\begin{proof}
	For each pair $(m,n) \in \ZZ^2$, we can compute that $a_{n+m} = \ip{10\fp{h_m(n)}}$, where 
	\[
		 h_m(n) =  \sum_{i=0}^d \a_i(m) n^{d-i} \qquad \text{ and } \qquad \a_i(m) = \fp{ \a m^i \binom{d}{i} }\,, \qquad 0 \leq i \leq d\,.
	\]
	Note that $\a_0(m) = \fp{\a}$ for all $m$. Let $\vec \a(m) = \bra{\vec \a_i(m)}_{i=1}^{d} \in [0,1]^{d}$, and let $A = \set{ \vec\a(m)}{ m \in \ZZ}$. 
	By Weyl's equidistribution theorem, $A$ is dense in $[0,1]^{d}$. 
	
	Let $N$ be a sufficiently large integer. Suppose that for some $m,m' \in \NN_0$,  $\bb a|_{[m,m+N)} = \bb a|_{[m',m'+N)}$. 
	Then 
	\begin{equation}\label{eq:245:1}
		\fpa{h_m(n) - h_{m'}(n)}  \leq 1/5\,,  \;\;\;\; \forall n \in [N]\, .
	\end{equation}
 In general, if $\beta = (\beta_0,\beta_1,\dots,\beta_{d-1}) \in [0,1)^{d}$ is such that $\fpa{\sum_{i=0}^{d-1} \b_i n^i} \leq 1/5$ for all $n \in [N]$, then it follows from the quantitative version of the Weyl equidistribution 
	theorem that there exists a non-zero integer $q = O_d(1)$ such that $\fpa{q\beta_i} = O_d (N^{-(d-i)})$ for all $i$, $0 \leq i \leq d$. 
	(This can be derived using \cite[Lemma 2.4]{HardyLittlewood-book}, or as a special case of a much more general statement in Theorem \ref{thm:GT}.) Applying this observation with $\b = \a(m) - \a(m')$, we see that there exists a set $R \subset \RR^d/\ZZ^d$, which is a union of $O_d(1)$ 
	rectangles with side lengths $O_d (N^{-(d-i)})$, $1 \leq i \leq d$, such that \eqref{eq:245:1} implies $\a(m)-\a(m') \bmod \ZZ^d \in R$. Note that the measure of $R$ is 
	\[ 0< \lambda\bra{R} = O_d\bra{N^{-(d-1)}  N^{-(d-2)} \cdots  N^0} = O_d\bra{ N^{-d(d-1)/2}}\,.\]
	Put $r = \ip{1/\lambda(R)}$. We can inductively construct an increasing sequence of integers $m_1,m_2,\dots,m_r$ with the property that $\a(m_j)-\a(m_i) \bmod \ZZ^d \not \in R$ for all $1 \leq i < j \leq r$. Indeed, if $m_1,m_2,\dots,m_{j-1}$ have already been constructed ($1 \leq j \leq r$) then it is enough to pick any $m_{j}$ with 
	\[ \a(m_{j}) \bmod \ZZ^d \not \in \textstyle\bigcup_{i=1}^{j-1} \bra{R + \a(m_i)},\]
	which is possible because the measure of the union above is strictly less than $1$ and $A$ is dense in $[0,1]^d$. By construction, all factors $\bb a|_{[m_i,m_i+N)}$ with $1 \leq i \leq r$ are distinct. It follows that \( p_{\bb a}(N) \geq r \gg_{d} N^{d(d-1)/2}.\)
\end{proof}
 
\subsection{Outline of the proof of Theorem \ref{thm:A}} 
We illustrate our strategy of the proof of Theorem \ref{thm:A} with a specific example, which already employs most of the tools used in the proof of the general case. Since the subsequent discussion serves only as illustration and motivation, we skip some of the technical details.

Let $\bb a$ be the {\gpword} over the alphabet $\{0,1\}$ given by 
\[
	a_n = \braif{\fp{\sqrt{2}n \ip{\sqrt{3}n}} < \frac{1}{2}}.
\]
(Above, we use Iverson bracket, introduced in Section \ref{sec:Notation}.) Our goal is to obtain a polynomial bound for $p_{\bb a}(N)$. We may write $a_n = c(g(n))$, where 
\[
	g(n) = \ip{2\fp{\sqrt{2}n \ip{\sqrt{3}n}}},
\]
and $c \colon \{0,1\} \to \{0,1\}$ is given by $c(0) = 1$, $c(1) = 0$. Of course, the subword complexity of $\bb a$ is the same as the subword complexity of $\bra{g(n)}_{n=0}^\infty$.

Our first step (cf.\ Example \ref{ex:par-gp}) is to note that for each $m \in \NN_0$ there exist $\a,\b,\gamma,\delta \in [0,1)$ such that for all $n \in \NN_0$ we have
\begin{align*}
	g(n+m) &= \ip{ 2\fp{ (\sqrt{2}n+\a) { \ip{\sqrt{3} n + \b}} + \gamma  n + \delta } } =  \tilde g_{\a,\b,\gamma,\delta}(n).
\end{align*}
Thus, instead of estimating the number of subwords of $\bra{g(n)}_{n=0}^\infty$ of a given length $N$, it will suffice to estimate the number of sequences $\bra{\tilde g_{\a,\b,\gamma,\delta}(n)}_{n=0}^{N-1}$, where $(\a,\b,\gamma,\delta)$ varies over $[0,1)^4$. This step is the key reason why we introduce the notion of a parametric {\gp} map in Section \ref{ssec:sc-para-gp}. 

For the purpose of counting, it will be more convenient to work with the operation $\ip{\cdot}$ rather than $\fp{\cdot}$. We can remove $\fp{\cdot}$ from the definition of $\tilde g$ by writing 
\begin{align*}
\tilde g_{\a,\b,\gamma,\delta}(n) &= 
2\ip{ (\sqrt{2}n+\a) { \ip{\sqrt{3} n + \b}} + \gamma  n + \delta }  \\&+
 \ip{ 2\bra{ (\sqrt{2}n+\a) { \ip{\sqrt{3} n + \b}} + \gamma  n + \delta }}
 =  2 g'_{\a,\b,\gamma,\delta}(n) - g''_{\a,\b,\gamma,\delta}(n).
\end{align*}
Thus, to obtain a polynomial bound for $p_{\bb a}(N)$, it will suffice to obtain polynomial bounds for the number of sequences $\bra{ g'_{\a,\b,\gamma,\delta}(n)}_{n=0}^{N-1}$ and $\bra{ g''_{\a,\b,\gamma,\delta}(n)}_{n=0}^{N-1}$, where $(\a,\b,\gamma,\delta)$ varies over $[0,1)^4$. We only consider the first of these two bounds, the second one being analogous. (Actually, the two classes of sequences are related by $g'_{\a,\b,\gamma,\delta}(n) = \ip{g''_{\a,\b,\gamma,\delta}(n)/2}$, but we do not need this fact.)

Expanding the brackets inside $\ip{\cdot}$, we may bring $g'_{\a,\b,\gamma,\delta}(n)$ into a more convenient form
\begin{align*}
	g'_{\a,\b,\gamma,\delta}(n)
	&= \ip{ 
		\sqrt{2} n \ip{\sqrt{3} n + \b} + \a \ip{\sqrt{3} n + \b} + \gamma  n + \delta 
	} 
	\\ &= \ip{ 
		\sqrt{2} h^{(1)}_{\b}(n) + \a h^{(2)}_{\b}(n) + \gamma  h^{(3)}(n) + \delta h^{(4)}(n)
	},
\end{align*} 
where $ h^{(1)}_{\b}(n) = n \ip{\sqrt{3} n + \b}$, $h^{(2)}_{\b}(n) = \ip{\sqrt{3} n + \b}$, $h^{(3)}(n) = n$ and $h^{(4)}(n) = 1$ are integer-valued {\gp} maps, which are strictly simpler than $g'_{\a,\b,\gamma,\delta}(n)$ in a sense that is made precise in Section \ref{ssec:sc-height}.  Note also that for each of the maps $h^{(i)}$ and each $n \in [N]$ we have $0 \leq h^{(i)}(n) < 10 N^2$.

Suppose that we have already proved a polynomial bound on the number of sequences $\brabig{h^{(1)}_{\b}(n)}_{n=0}^{N-1}$ and $\brabig{h^{(2)}_{\b}(n)}_{n=0}^{N-1}$ as $\beta$ varies over $[0,1)$. (Our proof of Theorem \ref{thm:A} proceeds by induction, cf.\ Section \ref{ssec:sc-induction}. Here, we omit the discussion of relatively simple and uninteresting cases.) Since the only dependence of $g'_{\a,\b,\gamma,\delta}(n)$ on $d\b$ is through $h^{(1)}_{\b}$ and $h^{(2)}_{\b}$, it will suffice to obtain for each $\b \in [0,1)$ a polynomial bound for the number of sequences $\bra{ g'_{\a,\b,\gamma,\delta}(n)}_{n=0}^{N-1}$, where $(\a,\gamma,\delta)$ varies over $[0,1)^3$. Thus, have reduced the number of parameters from $4$ to $3$. Let us fix the choice of $\b \in [0,1)$, and let $h^{(1)}(n) = h^{(1)}_{\b}(n)$, $h^{(2)}(n) = h^{(2)}_{\b}(n)$. We point out that from this point we will no longer need any information about the maps $h^{(i)}$ ($1 \leq i \leq 4$) other than that they map $[N]$ to $[10 N^2]$; in particular, we will not use the fact that they are {\gp} (cf.\ Proposition \ref{prop:inductive}).

We are thus left with the task of obtaining a polynomial bound on the number of sequences
\[
\bra{ g'_{\a,\b,\gamma,\delta}(n)}_{n=0}^{N-1} = 
\bra{
\ip{ 
		\sqrt{2} h^{(1)}(n) + \a h^{(2)}(n) + \gamma  h^{(3)}(n) + \delta h^{(4)}(n)
}}_{n=0}^{N-1}
\]
as $\a,\gamma,\delta$ varies over $[0,1)^3$. A natural approach at this point is to approximate $\a,\gamma,\delta$ by rational numbers $\a^*$, $\gamma^*$, $\delta^*$ with denominators $Q$ (to be optimised in the course of the argument) and (say) $0 \leq \a - \a^*,\ \gamma - \gamma^*, \delta-\delta^* \leq 1/Q$. Indeed, as long as we have a polynomial bound $Q = N^{O(1)}$, the choice of $\a^*$, $\gamma^*$, $\delta^*$ will only contribute a polynomial factor to our bound on the subword complexity of $\bb a$. At the same time, for $n \in [N]$ we expect (cf.\ eq.\ \eqref{eq:521:1}), at least heuristically, that
\begin{align*}
	g'_{\a,\b,\gamma,\delta}(n) &= \ip{ 
		\sqrt{2} h^{(1)}(n) + \a h^{(2)}(n) + \gamma  h^{(3)}(n) + \delta h^{(4)}(n)}
	\\
	&= 	 
	 \ip{ 
		\sqrt{2} h^{(1)}(n) + \a^* h^{(2)}(n) + \gamma^*  h^{(3)}(n) + \delta^* h^{(4)}(n) 
} = g'_{\a^*,\b,\gamma^*,\delta^*}(n)
.
\end{align*}
Let us make that last point somewhat more precise. We have 
\[
	\abs{ (\a-\a^*) h^{(2)}(n) + (\gamma-\gamma^*)  h^{(3)}(n) + (\delta-\delta^*) h^{(4)}(n) } \leq 30 N^2/Q.
\]
Thus, (cf.\ eq.\ \eqref{eq:521:2})
\(
	g'_{\a,\b,\gamma,\delta}(n) = g'_{\a^*,\b,\gamma^*,\delta^*}(n)
\)
 as long as we have
\begin{equation}\label{eq:sc:ex:001}
\fpa{ 
		\sqrt{2} h^{(1)}(n) + \a h^{(2)}(n) + \gamma  h^{(3)}(n) + \delta h^{(4)}(n)} > 30N^2/Q.
\end{equation}
Next, we will need to better understand the situation where \eqref{eq:sc:ex:001} does not hold. It will be convenient to define (cf.\ eq.\ \eqref{eq:521:15})
\[
h^{(0)}(n) =
-\nint{ 
		\sqrt{2} h^{(1)}(n) + \a^* h^{(2)}(n) + \gamma^*  h^{(3)}(n) + \delta^* h^{(4)}(n).
}
\]
Then we have the bound $0 \geq h^{0}(n) > -50 N^2$, and we may express \eqref{eq:sc:ex:001} in a slightly more convenient form
\begin{align*}
	\abs{ h^{(0)}(n) + \sqrt{2} h^{(1)}(n) + \a h^{(2)}(n) + \gamma  h^{(3)}(n) + \delta h^{(4)}(n) } > 30N^2/Q.
\end{align*}
Let $B \subset \ZZ^5$ denote the set of vectors $m = (m_0,m_1,m_2,m_3,m_4)$ with $0 \geq m_0 > -50N^2$ and $0 \leq m_1,m_2,m_3,m_4 < 10N^2$ (thus, $\bra{ h^{(0)}(n), h^{(1)}(n),\dots, h^{(4)}(n) } \in B$ for $n \in [N]$). 
The set of vectors $m = (m_0,m_1,m_2,m_3,m_4)$ such that
\begin{align}\label{eq:sc:ex:002}
	\abs{ m_0 + \sqrt{2} m_1 + \a m_2 + \gamma  m_3 + \delta m_4 } \leq 30N^2/Q
\end{align}
can be thought of as a discrete approximation of a hyperplane, and thus we expect it to be additively structured. Indeed, we show in Section \ref{sec:sc-auxiliary} (Proposition \ref{prop:lattice-approx}) that there exists a lattice $\Lambda$ with the following properties:
\begin{itemize}
\item If $m \in B$ and \eqref{eq:sc:ex:002} holds then $m \in \Lambda$.
\item If $m \in B \cap \Lambda$ then 
\begin{align}\label{eq:sc:ex:003}
	\abs{ m_0 + \sqrt{2} m_1 + \a m_2 + \gamma  m_3 + \delta m_4 } \leq C N^{12}/Q
\end{align}
for some absolute constant $C > 0$.
\end{itemize}

Consider $n \in [N]$ such that $\bra{ h^{(0)}(n), h^{(1)}(n),\dots, h^{(4)}(n) } \in \Lambda$. Then (cf.\ eq.\ \eqref{eq:521:3})
\begin{align*}
g'_{\a,\b,\gamma,\delta}(n) &= - h^{(0)}(n) + \ip{ 
		 h^{(0)}(n) + \sqrt{2} h^{(1)}(n) + \a h^{(2)}(n) + \gamma  h^{(3)}(n) + \delta h^{(4)}(n)},
\end{align*}
where the expression under the $\ip{\cdot}$ is bounded in absolute value by $C N^{12}/Q$. We will take $Q > CN^{12}$, meaning that
\begin{align}\label{eq:sc:ex:004}
g'_{\a,\b,\gamma,\delta}(n) &= -h^{(0)}(n),\ &&\text{or}& g'_{\a,\b,\gamma,\delta}(n) = -h^{(0)}(n)-1. 
\end{align}
Next, we need to develop a better understanding of when each of the two possibilities mentioned above occurs. Put
\begin{align*}
	\Lambda^+ &= \set{ m \in \Lambda }{m_0 + \sqrt{2} m_1 + \a m_2 + \gamma  m_3 + \delta m_4 \geq 0},\\
	\Lambda^- &= \set{ m \in \Lambda }{m_0 + \sqrt{2} m_1 + \a m_2 + \gamma  m_3 + \delta m_4 < 0}.
\end{align*}
Then $\Lambda = \Lambda^+ \cup \Lambda^-$ is a partition obtained by cutting $\Lambda$ with a hyperplane. We can now make \eqref{eq:sc:ex:004} more precise (cf.\ eq.\ \eqref{eq:521:4}):
\begin{equation}\label{eq:sc:ex:005}
g'_{\a,\b,\gamma,\delta}(n) =
\begin{cases}
	-h^{(0)}(n) & \text{if } \bra{ h^{(0)}(n), h^{(1)}(n),\dots, h^{(4)}(n) } \in \Lambda^+,\\
	-h^{(0)}(n)-1
	& \text{if } \bra{ h^{(0)}(n), h^{(1)}(n),\dots, h^{(4)}(n) } \in \Lambda^-.\\
\end{cases}
\end{equation}

Recall from \eqref{eq:sc:ex:001} that if $ \bra{ h^{(0)}(n), h^{(1)}(n),\dots, h^{(4)}(n) } \not \in \Lambda$ then $g'_{\a,\b,\gamma,\delta}(n) = g'_{\a^*,\b,\gamma^*,\delta^*}(n)$. Combining this observation with \eqref{eq:sc:ex:005}, we see that the sequence $\bra{ g'_{\a,\b,\gamma,\delta}(n)}_{n=0}^{N-1}$ is completely determined by the following data:
\begin{itemize}
\item the rational approximations $\a^*,\gamma^*,\delta^*$;
\item the restrictions of the sequences $h^{(1)}_{\b}$ and $h^{(2)}_{\b}$ to $[N]$;
\item the intersection of the lattice $\Lambda$ with the box $B$;
\item the hyperplane partition $\Lambda \cap B = (\Lambda^+ \cap B) \cup (\Lambda^- \cap B)$.
\end{itemize}
As we have seen, we may take $Q$ of the form $Q = C'N^{12}$ for a large constant $C'$. Then the number of choices of $\a^*,\gamma^*,\delta^*$ is $O(N^{36})$, and hence polynomial in $N$. The contribution from $h^{(1)}_{\b}$ and $h^{(2)}_{\b}$ is polynomial by the inductive assumption. Thus, it remains to estimate the number of ways in which the set $B$ can be partitioned into the components $B = (B \setminus \Lambda) \cup (\Lambda^+ \cap B) \cup (\Lambda^- \cap B)$, as described above. A polynomial estimate on the number of such partitions is obtained in Section \ref{sec:sc-auxiliary} (Proposition \ref{prop:half-lattice-count}) using the techniques of additive geometry.  
 
\section{Proof of Theorem \ref{thm:A}: notation and induction}\label{sec:sc-induction}

In this section we set up the inductive scheme and introduce notation which will be used in the proof of Theorem \ref{thm:A}.

\subsection{Height}\label{ssec:sc-height}
One of several measures of complexity of a generalised polynomial is the \emph{height}, that is, the number of nested instances of the floor function. 
Let $d \in \NN$, and let $\mathrm{GP}_0$ denote the polynomial maps from $\RR^d$ to $\RR$.
Inductively, for each $i \in \NN$, let $\mathrm{GP}_{i}$ be the smallest class of maps from $\RR^d$ to $\RR$ that is closed under sums and products, 
and which contains $g$ and $\ip{g}$ for each $g \in \mathrm{GP}_{i-1}$. By definition, for each {\gp} map $g \colon \RR^d \to \RR$, 
there exists an integer $i \in \NN_0$ such that $g \in \mathrm{GP}_i$. 
The height of a {\gp} map $g \colon \RR^d \to \RR$, 
denoted by $\cmp(g)$, is the least of such integers $i \in \NN_0$. 

 More generally, if $\Omega \subset \RR^d$ and $g \colon \Omega \to \RR$ is a {\gp} map then $\cmp(g)$ is defined as the least possible value 
 of $\cmp(\tilde g)$ where $\tilde g \colon \RR^d \to \RR$ is a {\gp} map and $\tilde g|_{\Omega} = g$. For instance, if $g$ is the {\gp} map 
 $\ZZ \to \ZZ \subset \RR$ given by
\( g(n) = \ip{ \sqrt{2} n \ipnormal{\sqrt{3} n}+\sqrt{5}n^2}\)
then $\cmp(g) \leq 2$. For the following result, we recall that by convention we treat the empty product $\prod_{j=1}^{0}(\cdots)$ as being identically equal to $1$.

\begin{lemma}\label{lem:gp-sum-rep}
	Let $d \in \NN$ and let $g \colon \RR^d \to \RR$ be a {\gp} map. Then $g$ can be written as
	\begin{equation}\label{eq:gp-sum-rep}
		g = \sum_{i=1}^s p_i \prod_{j=1}^{r_i} \ip{ h_{i,j} },
	\end{equation}
	where $s \in \NN$, $p_i \colon \RR^d \to \RR$ ($1 \leq i \leq s$) are polynomials, and for each $1 \leq i \leq s$, $r_i \in \NN_0$ and  $h_{i,j} \colon \RR^d \to \RR$ ($1 \leq j \leq r_i$) are a GP maps with $\cmp(h_{i,j}) < \cmp(g)$.
\end{lemma}

\begin{proof}
	We proceed by induction on $\cmp(g)$. If $g$ is a polynomial, there is nothing to prove, since we can take $s = 1$ and $p_1 = g$. 
	If $g = \ip{h}$ for some $h$ with $\cmp(h) < \cmp(g)$ then we can take $s = 1$, $r_1 = 1$, $p_1 = 1$, and $h_{1,1} = h$. 
	If $g = h + h'$ or $g = h \cdot h'$ for some $h$ and $h'$ with $\max\{\cmp(h),\cmp(h')\} < \cmp(g)$, then a representation of $g$ of the form 
	\eqref{eq:gp-sum-rep} can be obtained from the analogous representations of $h$ and $h'$.
\end{proof}

\subsection{Parametric generalised polynomials}\label{ssec:sc-para-gp}

It will be convenient to state some of our results in terms of \emph{parametric {\gp} maps}, by which we mean  families of {\gp} 
maps which include real-valued parameters as coefficients. For instance, the formula 
\[ g_{\a,\b}(n) = \ip{ \a n \ip{\b n}+\sqrt{2}n^2}\]
defines, for each $\a,\b \in \RR$, a {\gp} map from $\ZZ$ to $\RR$, and we will refer to $g_{\bullet}$ as a parametric {\gp} 
map\footnote{We use ``$\bullet$'' as a placeholder for a variable. Thus, in the discussion above, we let $g_{\bullet}$ denote 
the map $\RR^2 \to \RR^{\ZZ}$, $(\a,\b) \mapsto g_{\a,\b}$, which can be identified with a map $\RR^2 \times \ZZ \to \RR$ in a natural way.}  
$\ZZ \to \RR$. We make this notion precise in the following definition. Below and elsewhere, 
if $I \subset J$ are finite sets and $\vec\beta \in \RR^J$, we let $\vec\beta|I \in \RR^I$ denote the restriction of $\beta$ to $I$, 
that is, $(\beta|I)_i = \beta_i$ for all $i \in I$.

\begin{definition}
Let $I$ be a finite set. A \emph{{\pgp} map} $g_{\bullet} \colon \ZZ \to \RR$ \emph{with index set I} is a map $\RR^I \to \RR^{\ZZ}$, $\a \mapsto g_{\a}$, 
such that the combined map $\RR^I \times \ZZ \to \RR$, $(\a,n) \mapsto g_{\a}(n)$ is a {\gp} map.
 The height of the {\pgp} map $g_\bullet$, denoted by $\cmp(g_\bullet)$, is the height of the corresponding {\gp} map $(\vec\a,n) \mapsto g_{\vec\a}(n)$. 
\end{definition}

\begin{definition} 
Let $g_\bullet,h_\bullet \colon \ZZ \to \RR$ be two {\pgp} maps with index sets ${I}$ and ${J}$ respectively. We define the sum 
$ g_\bullet + h_\bullet $ to be the {\pgp} map with index set $I \cup J$ given by $(g+h)_{\vec\a}(n) = g_{\vec\a|I}(n) + h_{\vec\a|J}(n)$, $\vec\a \in \RR^{I\cup J}$, $n \in \ZZ$.  
  The product $g_\bullet \cdot h_\bullet$ is defined accordingly by $(g \cdot h)_{\vec\a}(n) = g_{\vec\a|I}(n) \cdot h_{\vec\a|J}(n)$, $\vec\a \in \RR^{I\cup J}$, $n \in \ZZ$. 
\end{definition}

Our interest in {\pgp} maps stems largely from the following lemma, which allows us to replace the study of subwords in {\gp} 
sequences by the study of prefixes in {\pgp} sequences.

\begin{lemma}\label{lem:subword->prefix}
	Let $g \colon \ZZ \to \RR$ be a bounded {\gp} map. Then there exists a {\pgp} map 
	$\tilde g_{\bullet}$ with index set $I$, such that, for each $m \in \ZZ$, there exists $\vec\a \in [0,1)^I$ 
	such that $g(n+m) = \tilde g_{\vec\a}(n)$ for all $n \in \ZZ$.
\end{lemma}
\begin{proof}
This is just a rephrasing of  Lemma \ref{lem:nil:shift}.
\end{proof}

In the examples below, we use, as in Section \ref{sec:constr}, the Iverson bracket convention: 
$\braif{\varphi} = 1$ if $\varphi$ is a true sentence and  $\braif{\varphi} = 0$ otherwise. 

\begin{example}
	Let $d \in \NN$ and $g(n) = \braif{ \fp{ \sqrt{2}n^d } < \frac{1}{2} }$. Pick any $m \in \ZZ$. 
	Then $g(n+m) = \tilde g_{\a}(n)$ for all $n \in \ZZ$, where $\a = (\a_i)_{i=0}^d$ is given by 
	$\alpha_i = \fp{ \sqrt{2} m^i \binom{d}{i} }$ ($0 \leq i \leq d$) and $\tilde g_{\bullet}$ is given by
	\[
		\tilde g_{\a}(n) = \braif{ \fp{ \sum_{i=0}^d \alpha_i n^{d-i}} < \frac{1}{2} }\,, \qquad \a \in \RR^{d+1},\ n \in \ZZ\,.
	\]
\end{example}

\begin{example}\label{ex:par-gp}	Let $g(n) = \braif{ \fp{ \sqrt{2}n \ip{\sqrt{3} n} } < \frac{1}{2} }$. Pick any $m \in \ZZ$ and 
let $a= \ip{\sqrt{2}m}$, $\a ={\sqrt{2}m}$, $b = \ip{\sqrt{3}m}$, $\b = \fp{\sqrt{3}m}$, $\gamma = \fp{\sqrt{2}b}$, and $\delta = \fp{\a b}$. Then
	\begin{align*}
	g(n+m) &= \braif{ \fp{ (\sqrt{2}n +a+\a) \bra{ \ip{\sqrt{3} n + \b} +b} } < \frac{1}{2} } \\
	& = \braif{ \fp{ (\sqrt{2}n+\a) { \ip{\sqrt{3} n + \b}} + \sqrt{2}b n + \a b } < \frac{1}{2} }  \\
	& = \braif{ \fp{ (\sqrt{2}n+\a) { \ip{\sqrt{3} n + \b}} + \gamma  n + \delta } < \frac{1}{2} } \\ &= \tilde g_{\a,\b,\gamma,\delta}(n)\,,	
	\end{align*}
where $\tilde g_{\bullet}$ is the parametric {\gp} map with index set $\{1,2,3,4\}$ given by
\[
	\tilde g_{\a_1,\a_2,\a_3,\a_4}(n) = \braif{ \fp{ (\sqrt{2}n+\a_1) { \ip{\sqrt{3} n + \a_2}} + \a_3  n + \a_4 } < \frac{1}{2} }\, .
\] 
\end{example}

When the dependence of a {\pgp} map on the parameters becomes too complicated, it is often more convenient to instead work 
with a {\pgp} map which has more parameters but depends on them in a simpler way. 
For instance, if $g_{\bullet}$ is a {\pgp} map of the form
\begin{equation}\label{eq:295:1}
	g_{\vec \a}(n) = f_1(\vec \a) h_{\vec \a}^{(1)}(n) + f_2(\vec \a) h_{\vec \a}^{(2)}(n) + f_3(\vec \a) h_{\vec \a}^{(3)}(n)\, ,
\end{equation}
where $f_1$, $f_2$, and $f_3$ are {\gp} maps and $h^{(1)}_\bullet$, $h^{(2)}_\bullet$, and $h^{(3)}_\bullet$ 
are {\pgp} maps, 
then it might be preferable to instead work with the {\pgp} map 
\begin{equation}\label{eq:295:2}
	g_{\vec \a,\vec \b}'(n) = \b_1 h_{\vec \a}^{(1)}(n) + \b_2 h_{\vec \a}^{(2)}(n) + \b_3 h_{\vec \a}^{(3)}(n)\,.
\end{equation}
We make this idea precise in the following definition.

\begin{definition}\label{def:ind:succ}
	Let $g_\bullet$ and $h_\bullet$ be two {\pgp} maps with index sets $I$ and $J$ respectively. 
	Then we say that $h_\bullet$ \emph{extends} $g_\bullet$, denoted $h_\bullet \succeq g_\bullet$, if there exists a {\gp} map 
	$\varphi \colon \RR^I \to \RR^J$ such that $g_{\vec \a} = h_{\varphi(\vec\a)}$ for all $\vec\a \in \RR^I$. 
\end{definition}

It is routine to check that the relation $\succeq$ defined above is a partial order. 

\begin{example}
	If $g_\bullet$ and $g'_\bullet$ are respectively given by \eqref{eq:295:1} and \eqref{eq:295:2}, 
	then $g'_\bullet \succeq g_\bullet$. One can take $\varphi(\vec\a) = (\vec\a, f_1(\vec\a),f_2(\vec\a),f_3(\vec\a))$.
\end{example}

\subsection{Induction scheme}\label{ssec:sc-induction}
Using the terminology introduced above, we are ready to explain the induction scheme that will be used in the proof of Theorem \ref{thm:A}. 
It can be construed as an analogue of the inductive definition of generalised polynomials, but restricted to $\ZZ$-valued sequences. 
Note that {\gp} maps from $\ZZ$ to $\RR$ can be identified with {\pgp} maps with an empty index set.

\begin{proposition}\label{prop:gen-poly-induction}
	Let $\cG$ be a family of {\pgp} maps from $\ZZ$ to $\ZZ$ with index sets contained in $\NN_0$. Suppose that $\cG$ has the following closure properties. 
\begin{enumerate}
\item\label{it:A-1} All {\gp} maps $\ZZ \to \ZZ$ belong to $\cG$.
\item\label{it:A-2} For every  $g_{\bullet}$ and $h_{\bullet} \in \cG$, it holds that $g_{\bullet}+h_{\bullet} \in \cG$ and $g_{\bullet} \cdot h_{\bullet} \in \cG$.
\item\label{it:A-4} For every $g_\bullet \in \cG$, $\cG$ contains all the {\pgp} maps $g'_\bullet \colon \ZZ \to \ZZ$ satisfying $g_\bullet \succeq g'_\bullet$.
\item\label{it:A-3} For every pair of disjoint finite sets $I \subset \NN$, $J \subset \NN$, and every sequence of {\pgp} maps $h^{(i)}_{\bullet} \in \cG$, $i \in I$, 
with index set $J$, $\cG$ contains the {\pgp} map $g_\bullet$ defined by
\[ g_{\vec \a,\vec\b}(n) = \ip{\sum_{i\in I} \a_i h^{(i)}_{\vec\b}(n)}\,, \qquad n \in \ZZ\,,\ \vec\a \in \RR^{I}\,,\ \vec\b \in \RR^J\,.\]
\end{enumerate}
Then $\cG$ contains all {\pgp} maps $\ZZ \to \ZZ$ with index sets contained in $\NN_0$.
\end{proposition}

\begin{proof}
	Since each {\pgp} map from $\ZZ$ to $\ZZ$ takes the form $\ip{g_\bullet}$ for some {\pgp} map $g_\bullet \colon \ZZ \to \RR$, 
	it suffices to show that $\ip{g_\bullet} \in \cG$ for each {\pgp} $g_\bullet$ with index set $I \subset \NN_0$. 
	We proceed by induction on $\cmp(g_\bullet)$. 
	
Suppose first that $\cmp(g_\bullet) = 0$, that is, $g_\bullet$ is a polynomial. Expanding, we can write
\[
	g_{\vec\a}(n) = \sum_{i=0}^d q_i(\vec \a) n^i\,, \qquad n \in \ZZ\,,\ \vec\a \in \RR^I\,,
\]
where $d \in \NN_0$ and $q_i \colon \RR^I \to \RR$ are polynomials.
It follows from \ref{it:A-1} and  \ref{it:A-3} that $\cG$ also contains the {\pgp} map with index set $\{0,1,\dots,d\}$ given by
\[
	g'_{\vec\a}(n) = \ip{\sum_{i=0}^d \a_i n^i}\,, \qquad n \in \ZZ,\ \vec\a \in \RR^{d+1}\, .
\]
Since $g'_\bullet \succeq \ip{g_\bullet}$, it follows from $\ref{it:A-4}$ that $g_\bullet \in \cG$.

Suppose next that $\cmp(g_\bullet) \geq 1$. Using Lemma \ref{lem:gp-sum-rep} and expanding out the polynomial contributions, 
we can represent $g_{\bullet}$ in the form
\[
	g_{\vec\b}(n) = \sum_{i=1}^s q_i(\vec \b)  \tilde h^{(i)}_{\vec\b}(n)\,, \qquad n \in \ZZ\,,\ \vec\b \in \RR^J\,,
\]
where for every $i$, $1 \leq i \leq s$, $q_i$ is a polynomial, and  $\tilde h^{(i)}_{\bullet}$ is a {\pgp} map of the form 
\[ \tilde h^{(i)}_{\vec\b}(n) =  n^{d_i} \prod_{j=1}^{r_i} \ip{h^{(i,j)}_{\vec\b}(n)} \qquad n \in \ZZ\,,\ \vec\b \in \RR^J\,, \]
where $d_i \in \NN_0$, $r_i \in \NN_0$, and $\cmp( h^{(i,j)}_\bullet) < \cmp(g_\bullet)$ for all $j$, $1 \leq j \leq r_i$. 
By the inductive assumption, for each pair $(i,j)$, with $1 \leq i \leq s$ and $1 \leq j \leq r_i$, we have $\ip{h^{(i,j)}_\bullet }\in \cG$. 
Consequently, applying \ref{it:A-1} and \ref{it:A-2}, we conclude that also $\tilde h^{(i)}_{\bullet} \in \cG$. 

For notational convenience, assume that $\min J > s$ and put $I = \{1,\dots,s\}$ 
(since the ordering of the parameters does not play any role, this does not decrease the level of generality). 
It follows from \ref{it:A-3} that $\cG$ also contains the {\pgp} map with index set ${I \cup J}$ that is defined by
\[
	g'_{\vec\a,\vec\b}(n) = \ip{\sum_{i=1}^s \a_i \tilde h^{(i)}_{\vec\b}(n) }\,, \qquad n \in \ZZ\,,\ \vec\a \in \RR^{I}\,,\ \vec\b \in \RR^{J}\,.
\]
Since $g'_\bullet \succeq \ip{g_\bullet}$, it follows from \ref{it:A-4} that $\ip{g_\bullet} \in \cG$. 
This ends the proof. 
\end{proof}

\section{Proof of Theorem \ref{thm:A}: auxiliary results}\label{sec:sc-auxiliary}

In this section we discuss some results in additive combinatorics and Diophantine approximation which will be used 
in the course of the proof of Theorem \ref{thm:A}.

\subsection{Additive geometry}

Let $r\in \NN$. A \emph{symmetric generalised arithmetic progression} of rank $r$ in an abelian group $Z$ with steps $x_1,x_2,\dots,x_r \in Z$ and side lengths 
$\ell_1,\ell_2,\dots, \ell_r \in \NN_0$ is defined as the set
\[
	\fS(x_1, x_2,\dots x_r; \ell_1,\ell_2,\dots, \ell_r) = 
		\set{ \sum_{i=1}^r n_i x_i}{ - \ell_i < n_i < \ell_i \text{ for all } i,\,1 \leq i \leq r }\,.
\]

Let $d \in \NN$. By a \emph{lattice} in $\RR^d$ we mean a discrete subgroup of $\RR^d$.
In particular, we do not require $\Lambda$ to have full rank, which is slightly non-standard but consistent (in particular with \cite{TaoVu-book}). 
If $\Lambda < \RR^d$ is a lattice of full rank, we let $\covol \Lambda = \vol(\RR^d/\Lambda)$ denote its covolume 
(that is, the measure of a fundamental domain of $\Lambda$).
By a \emph{convex body} in $\RR^d$, we mean a convex, open, non-empty, and bounded set, and by a \emph{half-space},  
we always mean a closed half-space, that is, a set of the form
\[
\fH(\vec x, t) = \set{\vec y \in \RR^d}{ \vec x \cdot \vec y \geq t}
\]
for some $\vec x \in \RR^d$ and $t \in \RR$.

The classical John's theorem asserts that each symmetric convex body $B$ in $\RR^d$ can be efficiently approximated by an ellipsoid $E$, 
in the sense that $E \subset B \subset \sqrt{d} E$. We will use a discrete analogue of this result.

\begin{theorem}[Discrete John's Theorem, {\cite[Thm.{} 3.36]{TaoVu-book}}]\label{thm:John}
	Let $d \in \NN$, $B \subset \RR^d$ be a symmetric convex body, and $\Lambda < \RR^d$ be a lattice of rank $r$.  
	Then there exist vectors $\vec v_1,\vec v_2,\dots, \vec v_r \in \RR^d$ and integers $\ell_1,\ell_2,\dots,\ell_r \in \NN$  such that
\begin{align*}
\fS(\vec v_1,\vec v_2,\dots, \vec v_r; \ell_1,\ell_2,\dots,\ell_r) 
	& \subset B \cap \Lambda  \\
        &\subset \fS(\vec v_1,\vec v_2,\dots, \vec v_r; r^{2r}\ell_1,r^{2r}\ell_2,\dots,r^{2r}\ell_r)\,.
\end{align*}
\end{theorem}

We also mention an estimate on the number of partitions of a finite set using a hyperplane.

\begin{theorem}[{\cite[Thm.\ 1]{Harding-1966}}]\label{thm:partition-cell}
	Let $d$ and $n$ be two positive integers,  and let $S \subset \RR^d$ be a set with $\abs{S} = n$. 
	Then the number of sets of the form $S \cap H$, where $H$ is a half-space, is at most $2\sum_{i=0}^d \binom{n-1}{i}$.
\end{theorem}

In the course of the argument, we will need estimates on the number of sets that can be obtained from a given convex body in $\RR^d$ 
by intersecting it with a sub-lattice of $\ZZ^d$ and a half-space. 

\begin{proposition}\label{prop:half-lattice-count}
	Let $d \in \NN$,  $\Gamma < \RR^d$ be a lattice, and  $B \subset \RR^d$ be a symmetric convex body. 
	Then the number of pairs of sets of the form $\bra{ \Lambda \cap B, \Lambda \cap B \cap H }$, 
	where $\Lambda < \Gamma$ is a lattice and $H$ is a half-space, belongs to $O_d(\abs{B \cap \Gamma}^{2d})$.
\end{proposition}

\begin{proof}

Put $M = \abs{B \cap \Gamma}$. Let $\Lambda < \Gamma$ be a sub-lattice. 
It follows form Theorem \ref{thm:John} that there exist vectors $\vec v_1,\dots,\vec v_d \in \Lambda \cap B$ such that 
	\(
	 \Lambda \cap B \subset \spanZ\{\vec v_1, \dots, \vec v_d\}.
	\) 
As a consequence,
	\[
	 \Lambda \cap B = \spanZ\{\vec v_1, \dots, \vec v_d\} \cap B\,.
	\]
In particular, the set $\Lambda \cap B$ is completely determined by the vectors $\vec v_1, \dots, \vec v_d$ which belong to $B \cap \Gamma$, and 
hence can be chosen in at most $M^d$ ways.

For a fixed choice of $\Lambda$, we infer from Theorem \ref{thm:partition-cell} that the number of sets of the form $\Lambda \cap B  \cap H$, 
where $H$ is a half-space, belongs to $O_d(M^d)$. Combining these two estimates yields the claim.
\end{proof}

\newcommand{\Q}{q}
\renewcommand{\P}{p}	

\subsection{Diophantine approximation}

In this subsection, we briefly discuss approximate linear relations with integer coefficients. 
Let $d \in \NN$. Given a vector $\a = (\a_i)_{i=1}^d \in \RR^d$, $\e > 0$, and $N \in \NN$, we set
\begin{equation}\label{eq:def-of-R}
\cR_N(\vec \a, \e) = \set{ \vec n \in \ZZ^d }{\norm{n}_{\infty} < N, \ \abs{\sum_{i=1}^d n_i \a_i } < \e}\,.
\end{equation} 
These sets naturally appear in the theory of Diophantine approximation. For instance, a classical theorem of Dirichlet asserts that for all 
$\alpha \in [0,1)^d$, the set $\cR_{N}(\a,N^{-1/(d-1)})$ is non-empty, and the exponent ${-1/(d-1)}$ cannot be improved in general \cite{Cassels-1955}.  
In our application, we will be interested in the case where $\e$ is considerably smaller. In this regime, 
the sets $\cR_N(\vec\a, \e)$ can be approximated by lattices in a sense that is made precise by the following proposition.

\begin{proposition}\label{prop:lattice-approx}
	Let $d \in \NN$. There exists a positive real number $C_d$ such that for each $\vec\a = (\a_i)_{i=1}^d \in \RR^d$, $\e > 0$, and $N \in \NN$, 
	there exists a lattice $\Lambda  = \Lambda(\vec\a, \e, N)$ satisfying
\begin{equation}\label{eq:def-of-Lambda}
	\cR_N(\vec\a,\e) \subseteq \Lambda \cap (-N,N)^d \subseteq \cR_N(\vec\a, C_d N^d \e)\,.
\end{equation}
\end{proposition}

For later reference, we let $\Lambda(\vec\a, \e, N)$ denote a lattice satisfying \eqref{eq:def-of-Lambda}. 
The remaining part of this section is devoted to the proof of Proposition \ref{prop:lattice-approx}.  
We begin with a lemma concerning finite groups, where the situation is simpler. 
Recall that, for an abelian group $G$ and $g \in G$, the order $\ord(g)$ is the least $q \in \NN$ with $qg = e_G$, the order of $G$ is the cardinality of $G$, and the exponent of $G$ is the least common multiple of $\ord(g)$ for $g \in G$. For a set $S \subset G$ an and $k \in \NN$, we let $kS$ denote the $k$-fold sumset 
$S+S+ \dots + S = \set{a_1+a_2+\dots+a_k}{ a_i \in S}$.

\begin{lemma}\label{lem:fill-group}
	Let $G$ be a finite abelian group with order $M \geq 2$ and exponent $\Q$, and let $S \subset G$ be a generating set with $e_G \in S$. 
	Then there exists $k \leq {\Q \log M}/{\log \Q}$ such that $kS = G$.
\end{lemma}

\begin{proof}
	We proceed by induction on $M$. Pick any $g \in S \setminus \{e_G\}$ and let $\ell = \ord(g)$. 
	Since $G$ is non-trivial, we may assume that $\ell \geq 2$. Put $f(x) = {x}/{\log x}$ ($x > 1$) and note that $f(n) \leq f(m)$ for all pairs of integers 
	$(n,m)$ with $m\geq n\geq 2$ and $n \mid m$. Our goal is to find $k \leq f(q)\log M$ with $kS = G$.
	
	Let $\bar G = G/\left< g \right>$, and let $\pi \colon G \to \bar G$ be the quotient map. 
	Then $G$ has order $\bar M = M/\ell$ and exponent $\bar \Q$ which divides $\Q$. 
	The set $\bar S = \pi(S)$ generates $\bar G$. If $\bar G$ is trivial then $G$ is the cyclic group generated by $g$, 
	and hence $\Q = M$ and $k S = G$ for $k = f(\Q) \log M = M$, as needed. (Note that in this step we use the assumption that $e_G \in S$.) 
	Suppose next that $\bar G$ is non-trivial, meaning that $\bar M \geq 2$. By the inductive assumption, there exists $\bar k \leq f(\bar \Q) \log \bar M$ 
	with $\bar k \bar S = \bar G$. It follows that $k S = G$, where $k = \bar k + \ell$. We can estimate
	\[ k \leq f(\bar \Q) \log \bar M + \ell = f(\Q) \log M - \bra{f(\Q) -  f(\bar \Q)}\log M + \bra{ f(\ell) - f(\bar \Q)} \log \ell.\]
	 Thus, it is enough to show that
	 \[ \bra{f(\Q) - f(\bar \Q)}\log M \geq \bra{ f(\ell) - f(\bar \Q)} \log \ell \,\]
	which follows from the observations that $M \geq \ell$, $\ell\mid \Q$, and $\bar \Q \mid \Q$.
\end{proof}

\begin{lemma}\label{lem:span-Z-quant}
	Let $d \in \NN$, $B \subset \RR^d$ be symmetric convex body, and $S \subset B$ be a symmetric set with $0 \in S$. 
	Put $\Lambda = \spanZ S$ and assume that $\Lambda$ is a lattice.
Then there exists $k = O_d(\abs{\Lambda \cap B})$ such that
	 \( kS \cap B = \Lambda \cap B.\)
\end{lemma}

\begin{proof}
	Put $M= \abs{\Lambda \cap B}$.
	We proceed by induction on $d$, including the degenerate case $d = 0$ for which one can take $k = 1$. 
	Assume that $d \geq 1$ and that the claim is proved for $d-1$. We may also assume without loss of generality that $\Lambda$ has full rank, 
	since otherwise we could replace $\RR^d$ with the subspace spanned by $\Lambda$. 
		
	Applying Theorem \ref{thm:John}, we conclude that there exist vectors $\vec w_1,\vec w_2, \dots, \vec w_d \in \Lambda$ 
	and side lengths $\ell_1, \ell_2,\dots, \ell_d \in \NN$ such that
	\[
		\fS(\vec w_1, \vec w_2, \dots, \vec w_d; \ell_1,\ell_2,\dots,\ell_d) \subset B \cap \Lambda 
		\subset \fS(\vec w_1, \vec w_2, \dots, \vec w_d; D\ell_1,D\ell_2,\dots,D\ell_d)\,,
	\] 
	where $D = d^{2d} = O_d(1)$. Since $\Lambda$ has rank $d$, the vectors $w_1,w_2,\dots,w_d$ form a basis of $\RR^d$. 
	Note that $\prod_{i=1}^d \ell_i \leq M$. Applying a change of basis, we may assume that $\vec w_i = \vec e_i$, 
	the $i$-th standard basis vector, for all $i$, $1 \leq i \leq d$. In particular, $\Lambda = \ZZ^d$.
	
Pick any $d$-tuple of linearly independent vectors $\vec v_1,\vec v_2,\dots,\vec v_d \in S$. 
Let $A \in \RR^{d \times d}$ be the matrix satisfying $A \vec e_i = \vec v_i$, and put $\Q = \det A = \covol\bra{ \spanZ(\vec v_1, \vec v_2,\dots, \vec v_d)}$. 
Then $\Q \neq 0$ and the Leibniz formula gives the estimate $\abs{\Q} \leq d! D^d \prod_{i=1}^d \ell_i = O_d(M)$. 
Using Cramer's rule to find the inverse of $A$, we observe that the matrix $\Q A^{-1}$ has integer entries and satisfies 
$\abs{ \Q ( A^{-1})_{i,j} } = O_d(M/\ell_i)$ for each pair $(i,j)$ with $1 \leq i,j \leq d$. 
Since $\Q e_i = \sum_{j=1}^d \Q (A^{-1})_{i,j} v_j$ and since $0 \in S$, we conclude that there exists $\P = O_d(M)$ 
such that $\Q \vec e_i \in \floor{\P/\ell_i} S$ for every $i$, $1 \leq i \leq d$.
	
	Pick any $\vec u \in \ZZ^d \cap B$. Since $\vec u \in \ZZ^d$, by Lemma \ref{lem:fill-group} applied to the group $G = \ZZ^d/\Q\ZZ^d$ 
	and the set\footnote{Note that $\spanZ S=\Lambda=\mathbb Z^d$. Hence it makes sense to consider the set $S\bmod q$. Furthermore, since by assumption 
	$0\in S$, it follows that $0\in S\bmod q$ as needed.} $\bra{ S \bmod \Q} \subset G$, there exists $\pvec u' = (u_i')_{i=1}^d \in \ZZ^d$ such that $\vec u \in \Q\pvec u' + n S$ for some $n \leq d\Q$. Then for every $i$, $1 \leq i \leq d$, we have $\abs{u_i'} \leq D \ell_i$, so $\Q \pvec u' \in m S$, where
	\[
		m = \sum_{i=1}^d \abs{u_i'} \floor{\P/\ell_i} \leq d D \P = O_{d}(M)\,.
	\]
	It follows that we can take $k = n + m = O_d(M)$.	
\end{proof}

\begin{example}
	Pick any $d,N \in \NN$. Let $S = \{\vec 0, \pm \vec v_1, \pm \vec v_2, \dots, \pm \vec v_d\}$, where $\vec v_1 = \vec e_1$ and 
	$\vec v_i = \vec e_i - N \vec e_{i-1}$ for every $i$, $2 \leq i \leq d$. Thus, $\spanZ(S) = \ZZ^d$. We can compute that 
	\[ 
		\vec e_d = \vec v_{d} + N \vec v_{d-1} + \dots + N^{d-1} \vec v_1\,. 
	\]
	Taking $B = (-2N,2N)^d$, we see that $\abs{B \cap \ZZ^d} \leq (4N)^d$ and if $k S \supset B \cap \ZZ^d$ then $k \geq N^d$. Hence, the bound in Lemma \ref{lem:span-Z-quant} is tight up to a constant factor.
\end{example}

\begin{proof}[Proof of Proposition \ref{prop:lattice-approx}]
	Apply Lemma \ref{lem:span-Z-quant} to $B = (-N,N)^d$ and $S = \cR_{N}(\vec\a,\e)$. 
	It is immediate from the definition that $\cR_N(\a,\e) \subset \Lambda \cap B$, where $\Lambda = \spanZ S$ is a sub-lattice of $\ZZ^d$. 
	For the other inclusion, we note that $\Lambda \cap B = kS = k \cR_N(\a,\e) \subset \cR_N(\a,k\e)$, where $k = O_d(\abs{\Lambda \cap B}) = O_d(N^d)$.
\end{proof}  

\section{Proof of Theorem \ref{thm:A}: combining the ingredients}\label{sec:sc-proof}

We are now ready to prove a proposition which serves as the inductive step in the proof Theorem \ref{thm:A}. 
In the argument, we will use the following elementary fact.

\begin{lemma}\label{lem:sc-pr:tech}
	Let $x$ and $x^*$ be two real numbers and assume that $\abs{x-x^*} < \min\bra{ \abs{x-\nint{x^*}}, 1/2}$. Then $\ip{x} = \ip{x^*}$.
\end{lemma}
\begin{proof}

If $\fp{x^*} \in [0,1/2)$, then $\abs{x-\nint{x^*}} > \abs{x-x^*}$ implies that $x > \nint{x^*} = \ip{x^*}$, and $\abs{x-x^*} < 1/2$ 
implies that $x < \ip{x^*} + 1$. Thus, $\ip{x} = \ip{x^*}$. The case where $\fp{x^*} \in [1/2,1)$ is analogous.
\end{proof}

\begin{proposition}\label{prop:inductive}
	Let $d \in \NN$, let $h_1,h_2,\dots,h_d \colon [N] \to \ZZ$ be sequences, with $\norm{h_i}_{\infty} \leq H$ for every $i$,  $1 \leq i \leq d$. 
	For $\vec\a \in \RR^d$, we let $g_{\vec\a} \colon [N] \to \ZZ$ denote the sequence defined by 
	\[ g_{\vec\a}(n) = \ip{\sum_{i=1}^d \alpha_i h_i(n)}\,,\quad n \in [N]\,.\]
Then $\abs{\set{ g_{\vec\a}}{ \vec\a \in [-R,R)^d}}=O_d(R^dH^{3d^2})$.
\end{proposition}

\begin{proof}
	Note that $g_{\vec\a}(n) = g_{\fp{\vec\a}}(n) + g_{\ip{\vec\a}}(n)$, where $\fp{\vec\a} \in [0,1)^d$ and 
	$\ip{\vec\a} \in \{-R,-R+1,\dots,R-2,R-1\}^d$, which is a set whose cardinality is in $O_d(R^d)$.  
	Thus, it remains to show that the cardinality of the set $\set{ g_{\vec\a}}{ \vec\a \in [0,1)^d}$ is in $O_d(H^{3d^2})$.
	 
	Pick any $\vec\a \in [0,1)^d$. Our plan is to construct an alternative description of $g_{\vec\a}$ that, instead of $\vec\a$, includes a finite number 
	of parameters chosen from sets whose cardinality can be estimated in an easier way. 

	Put $\e = 1/\bra{100C_d(dH)^d}$, where $C_d$ is the constant from Proposition \ref{prop:lattice-approx}. 
	Let $\pvec\a^* = (\a_i^*)_{i=1}^d \in [0,1)^d$ be any vector such that $\abs{\a_i - \a^*_i} < \e/dH$ and $\a_i^* \in (\e/d H)\ZZ$ for every $i$, 
	$1 \leq i \leq d$. This ensures that 
	\begin{equation}\label{eq:521:1}
		\abs{\sum_{i=1}^d \a_i h_i(n) - \sum_{i=1}^d \a_i^* h_i(n) } < \e\,, \qquad \forall\ n \in [N]\,.	
	\end{equation}
	 Let $h_0 \colon [N] \to \ZZ$ be the map given by 
	 \begin{equation}\label{eq:521:15}
	 h_0(n) = - \nint{\sum_{i=1}^d \a_i^* h_i(n) }	 
	 \end{equation}
	and let $\vec h = (h_0,h_1,h_2,\dots,h_d) \colon [N] \to \ZZ^{d+1}$. 
	Note that $\norm{h_0}_\infty \leq dH$ and, since $\e < 1/2$, that $h_0(n) + g_{\vec\a}(n) \in \{0,-1\}$ for all $n \in [N]$. 
	Put $B = [-dH,dH] \times [-H,H]^d$. Let also $\Lambda = \Lambda(1\vec\a, \e, d H)$ be the lattice constructed in 
	Proposition \ref{prop:lattice-approx} and set 
	\[\Lambda^+ = \set{\vec m \in \Lambda}{ 1\vec\a \cdot \vec m  \geq 0}\,.\] 
	Here and elsewhere, we use the shorthand $1\vec\a = (1,\a_1,\a_2,\dots,\a_d)$. Hence $(1\a)_0 = 1$ and $(1\a)_i = \a_i$ for every $i \geq 1$. 
 	
	We have now introduced all the objects needed to obtain a more well-behaved formula for $g_{\vec\a}(n)$. 
	We consider two cases, depending on whether $\vec h(n) \in \Lambda$ or not. If $\vec h(n) \in B \setminus \Lambda$, 
	then $\abs{1\vec\a \cdot \vec h(n)} > \e$ and hence, by Lemma \ref{lem:sc-pr:tech}, we find
\begin{align}\label{eq:521:2}
	g_{\vec\a}(n) = \ip{\sum_{i=1}^d \a_i h_i(n)} =  \ip{\sum_{i=1}^d \a_i^* h_i(n)}\,.
\end{align} 	
Next, if $\vec h(n) \in B \cap \Lambda$, then $\abs{1\vec\a \cdot \vec h(n)} < C_d(dH)^{d} \e < 1/2$. Since $h_0(n) \in \ZZ$, we have
\begin{align}\label{eq:521:3}
	g_{\vec\a}(n) = -h_0(n) + \ip{1\vec\a \cdot \vec h(n)} \,.
\end{align} 
If $n \in \Lambda^+$, then $1\vec\a \cdot \vec h(n) \in [0,\frac{1}{2})$ and hence  $\ip{1\vec\a \cdot \vec h(n)} = 0$. 
Similarly, if $\vec h(n) \in \Lambda \setminus \Lambda^+$, then $1\vec\a \cdot \vec h(n) \in (-\frac{1}{2},0)$ and 
hence $\ip{1\vec\a \cdot \vec h(n)} = -1$.
Combining \eqref{eq:521:2} and \eqref{eq:521:3}, and expanding out the definition of $h_0$, we conclude that
	\begin{align}\label{eq:521:4}
	g_{\vec\a}(n) = 
\begin{dcases}
\ip{\sum_{i=1}^d \a_i^* h_i(n)} & \text{ if } \vec h(n) \in B \setminus \Lambda\,, \\
\ip{\sum_{i=1}^d \a_i^* h_i(n)+\frac{1}{2}} & \text{ if } \vec h(n) \in B  \cap \Lambda^+ \,,\\
\ip{\sum_{i=1}^d \a_i^* h_i(n)-\frac{1}{2}} & \text{ if } \vec h(n) \in B  \cap \Lambda \setminus \Lambda^+.
\end{dcases}
\end{align} 

	It follows from \eqref{eq:521:4} that, in order to determine $g_{\vec\a}|_{[N]}$, it is sufficient to know the following data: 
	$\pvec\a^*$, $B \cap \Lambda$, and $B  \cap \Lambda^+$. The number of possible choices of $\pvec\a^*$ is in $O_d(H^{d^2})$ directly by the definition. 
	Furthermore, it follows from Proposition \ref{prop:half-lattice-count}, applied with  $\Gamma = \ZZ^d$, that the number of possible choices for $B \cap \Lambda$ 
	and $B \cap \Lambda^+$ are both in $O_{d}\bra{ H^{2d^2} }$. Consequently, the number of distinct sequences of the form $g_{\vec\a}|_{[N]}$ is in 
	$O_{d}\bra{H^{3d^2}}$. This ends the proof. 
\end{proof}

We have now collected all the components needed to prove our main result.

\begin{proposition}\label{thm:main}
	Let $g_\bullet:\ZZ \to \ZZ$ be a {\pgp} map  with index set $I \subset \NN_0$. 
	Then there exist two positive real numbers $B = B(g_\bullet)$ and $C = C(g_\bullet)$ such that 
\begin{equation}\label{eq:469:1}
	\abs{ \set{g_{\vec\a}|_{[N]}}{\a \in \RR^I,\ \norm{\vec\a}_{\infty} \leq R }} = O_g\bra{R^B N^{C}}\,.
\end{equation}	
\end{proposition}

\begin{proof}
	We proceed by structural induction with respect to $g_\bullet$, using the scheme introduced by Proposition \ref{prop:gen-poly-induction}. 
	Let $\cG'$ denote the set of all {\pgp} maps $g_\bullet$ such that \eqref{eq:469:1} holds. 
	
\begin{enumerate}
\item If $g \colon \ZZ \to \ZZ$ is a {\gp} map (viewed as a {\pgp} map with empty index set), then the set in \eqref{eq:469:1} has only one element 
and thus $g \in \cG$ holds trivially. 
\item For each $g_\bullet,h_\bullet \in \cG'$ with index sets $I,J \subset \NN_0$ respectively, and for each $ N \in \NN$, we note that 
$(g + h)_{\gamma}|_{[N]}$, $\gamma \in \RR^{I\cup J}$, is uniquely determined by $g_{\a}|_{[N]}$ and $h_{\b}|_{[N]}$, 
where $\a = \gamma|I$ and $\b = \gamma|J$. As a consequence, $g_\bullet + h_{\bullet} \in \cG'$ and we can take
	\begin{equation}\label{eq:469:2}
		B(g_\bullet + h_{\bullet}) \leq B(g_\bullet) + B(h_{\bullet}) \;\;\mbox{ and  }\;\;
		 C(g_\bullet + h_{\bullet})  \leq C(g_\bullet) + C(h_{\bullet})\,.
	\end{equation}
	The same reasoning also applies to the product  $g_\bullet \cdot h_{\bullet} $, with the same bounds. 
		
\item	Next, let $g_\bullet$ and $g_\bullet'$ be two {\pgp} maps with index sets $I,J \subset \NN_0$ respectively, 
and suppose that $g_\bullet \in \cG'$ and $g_\bullet \succeq g'_\bullet$. By Definition \ref{def:ind:succ}, 
there exists a {\gp} map $\varphi \colon \RR^J \to \RR^I$ such that $g_{\vec\b}' = g_{\varphi(\vec\b)}$ for all $\b \in \RR^J$. 
Since $\varphi$ is a {\gp} map, there exists a constant $D = D(\varphi)$ such that $\norm{\varphi(\b)}_\infty \leq D R^D$ 
for all $\b \in \RR^J$ with $\norm{\b}_\infty \leq R$. Hence, for every pair of positive integers $(N,R)$, we have
	\begin{align*}
	\abs{\set{ g'_{\vec\b}|_{[N]} }{\b \in \RR^J,\ \normbig{\vec\b}_{\infty} \leq R}}
	& \leq
		\abs{ \set{ g_{\vec\a}|_{[N]} }{\a \in \RR^I,\ \norm{\vec\a}_{\infty} \leq D R^D } },
	\end{align*}
	which implies that $g'_\bullet \in \cG'$ and that we can take
	\begin{equation}\label{eq:469:3}
		B(g_\bullet') \leq D B(g_\bullet) \;\;\mbox{ and }\;\;  
		 C(g'_\bullet) \leq C(g_\bullet)\,.
	\end{equation}

\item	Finally, let $I \subset \NN_0$ be a finite set, and, for every $i\in I$, let $h^{(i)}_\bullet \in \cG'$ be a {\pgp} map with index sets $J^{(i)} \subset \NN_0 \setminus I$. 
Set $J = \bigcup_{i\in I} J^{(i)}$. Our aim is to show that $\cG'$ also contains the {\pgp} map $g_\bullet$ with index set $I \cup J$ defined by
	\[ 
		g_{\vec\alpha, \vec\b}(n) = \ip{ \sum_{i \in I} \alpha_i h^{(i)}_{\vec\b} (n)}\,,\qquad n \in \ZZ\,,\ \vec \alpha \in \RR^I\,,\ \vec\b \in \RR^J\,.
	\]
	
	Given any $N,R \in \NN$, we set 
	\[ H = \max_{i \in I} \max_{n \in [N]} \sup_{\normsml{ \vec\b }_{\infty} \leq R } \abs{h^{(i)}_{\vec\b}(n)}\,.\]
	Since the $h^{(i)}_\bullet$'s are {\pgp} maps, we have $H = O_g(R^D N^D)$ for some constant $D = D(g_\bullet)$. 
	For every $i \in I$, set $B_i = B(h^{(i)}_\bullet)$, $C_i = C(h^{(i)}_\bullet)$, and
\[
	\cH_i =  \set{h^{(i)}_{\vec\b}|_{[N]}}{\b \in \RR^J,\ \norm{\vec\b}_{\infty} \leq R }\,.
\]	
 
 By Proposition \ref{prop:inductive}, for each $h_1 \in \cH_1, h_2 \in \cH_2, \dots, h_d \in \cH_d$, 
 the number of sequences of the form $\ip{\sum_{i=1}^d \alpha_i h_i}$, where $\vec\a \in \RR^d$ and $\norm{\vec\a}_\infty \leq R$, 
 belongs to $O_d(R^d H^{3d^2})$. 
 It follows that the number of distinct sequences of the form $g_{\vec\a,\vec\b}|_{[N]}$, where 
 $\a \in \RR^I$, $\b \in \RR^J$, and $\max\{\norm{\a}_\infty,\norm{\b}_\infty\} \leq R$, is at most
	\[
		\prod_{i=1}^d \abs{\cH_i} \cdot O_d\bra{R^d H^{3d^2}} = O_g(R^B N^C)\,,
	\]
	where 
	\begin{equation*}\label{eq:469:4}
		B = \sum_{i \in I} B_i + d + 3d^2 D \;\;\mbox{ and }\;\;
		 C = \sum_{i \in I} C_i + 3d^2 D\,.
	\end{equation*}
	In particular, $g_{\bullet} \in \cG'$ and we can take $B(g_{\bullet})\leq B$ and $C(g_{\bullet})\leq C$.
\end{enumerate}
	
	It now follows from Proposition \ref{prop:gen-poly-induction} that $\cG'$ contains all {\pgp} maps $\ZZ \to \ZZ$ 
	with index sets contained in $\NN_0$. This ends the proof.	
\end{proof}

\begin{proof}[Proof of Theorem \ref{thm:A}]
	Let $\bb a$ be a bracket word defined over an alphabet $\Sigma$. We infer from Corollary \ref{cor:constr:gpword-nice} that 
	there exist a finitely-valued {\gp} map $g \colon \NN_0 \to \NN$ and 
	$c \colon g(\NN_0) \to \Sigma$ such that $\bb a = \brabig{ c(g(n)) }_{n=0}^\infty$.
	 Let $\bb g = (g(n))_{n=0}^\infty$. It follows from Proposition \ref{thm:main} and Lemma \ref{lem:subword->prefix} that $p_{\bb g}(N) = O_g(N^C)$ for some $C > 0$. 
	 Since $p_{\bb a}(N) \leq p_{\bb g}(N)$ for all $N \in \NN$, this ends the proof.
\end{proof}

\color{black} \appendix 

\section{Nilpotent dynamics}\label{app:nil}

\subsection{Connectedness}

A technical issue which often leads to complications is that there is in general no guarantee that the nilpotent Lie group $G$ 
and the nilmanifold $X$ in Theorem \ref{thm:BL-mini} should be connected. 

As for connectedness of $X$, in most application this is not a major problem. In this case, it follows from \cite[Thm.{} B]{Leibman-2005} that there exists $q \in \NN$ and connected sub-nilmanifolds $Y_0,Y_1,\dots,Y_{q-1} \subset X$ such that $\bra{g(qn+i)\Gamma}_{n=0}^\infty$ is equidistributed in $Y_i$ for each $i \in [q]$. Passing to an arithmetic progression of the form $q \NN_0 +i$, one can then work with the connected nilsystem $(Y_i,T^q)$.

Connectedness of $G$ is a more fundamental issue. For instance, for $\a \in \RR \setminus \QQ$, the {\gp} map $g(n) = \fp{n^2 \a}$ can be represented using the nilrotation $(G/\Gamma,T_h)$, where
\[
	G =
	\begin{bmatrix}
	1 & \ZZ & \RR \\
	0 & 1 & \RR \\
	0 & 0 & 1
	\end{bmatrix}\,, \qquad
	\Gamma =
	\begin{bmatrix}
	1 & \ZZ & \ZZ \\
	0 & 1 & \ZZ \\
	0 & 0 & 1
	\end{bmatrix}\,,  \;\;\mbox{ and } \;\;
	h =	\begin{bmatrix}
	1 & 1 & -\a \\
	0 & 1 & 2\a \\
	0 & 0 & 1
	\end{bmatrix}.
\]
Indeed, a simple computation shows that
\[
	g^n \Gamma =
	\begin{bmatrix}
	1 & 0 & \fp{n^2\a} \\
	0 & 1 & \fp{2n\a} \\
	0 & 0 & 1
	\end{bmatrix}\Gamma.
\]
However, if we pass to the connected component of $G$, then we obtain a much simpler nilsystem $G^{\circ}/\Gamma \cap G^{\circ} \simeq \RR^2/\ZZ^2$, 
where it is no more possible to represent $g$. Fortunately, this problem disappears if instead of sequences $g^n\Gamma$ we consider polynomial maps from $\ZZ$ to $G$. 
Given a connected and simply connected Lie group $G$, a \emph{polynomial map} $p \colon \ZZ \to G$ is a sequence of the form
\[
	p(n) = g_1^{p_1(n)}g_2^{p_2(n)} \cdots g_r^{p_r(n)}\,,
\]
where $g_i \in G$, $r\in \NN$, and $p_i \colon \RR \to \RR$ are polynomials. 

\begin{theorem}[{\cite[Thm.{} A$^*$]{BergelsonLeibman-2007}}]\label{thm:BL-poly}
	Any {\gp} map $g \colon \ZZ \to [0,1)$ has a representation $g(n) = \tau_j(p(n)\Gamma)$, where $X = G/\Gamma$ is a nilmanifold 
	with $G$ connected and simply connected, $\tau = (\tau_1,\tau_2,\dots, \tau_{\dim G}) \colon G/\Gamma \to [0,1)^{\dim G}$ are Mal'cev coordinates, 
	$p \colon \ZZ \to G$ is a polynomial map, and $1 \leq j \leq \dim G$. Conversely, any map $g \colon \ZZ \to \RR$ of the aforementioned form is a {\gp} map.
\end{theorem}

\subsection{Quantitative equidistribution}

Let $X = G/\Gamma$ be a nilmanifold, with $G$ connected and simply connected. Throughout, we assume that $X$ is equipped 
with a Mal'cev coordinates $\tau$ as well as a metric $d_X$. 

 Recall that $X$ is equipped with a natural choice of measure, namely the Haar measure $\mu_X$. 
 A sequence $(x_n)_{n=0}^\infty$ is equidistributed in $X$ if for each continuous map $f \colon X \to \RR$,
\[
	\lim_{N \to \infty} \frac{1}{N} \sum_{n=0}^{N-1} f(x_n) = \int_X f d \mu_X\,.
\]
We will also need a quantitative variant of this property. For $\delta > 0$ and $N \in \NN$, a sequence $(x_n)_{n=0}^{N-1}$ is 
said to be $\delta$-equidistributed in $X$ if for every Lipschitz map $f \colon X \to \RR$,
\[
	\abs{ \frac{1}{N} \sum_{n=0}^{N-1} f(x_n) - \int_X f d \mu_X} \leq \delta \normLip{f}\, ,
\]
where $\normLip{f} = \sup_{x \in X} \abs{f} + \sup_{x,y \in X} \abs{f(x)-f(y)}/d_{X}(x,y)$.

We let $X_{\mathrm{ab}} = G/[G,G]\Gamma \simeq (\RR/\ZZ)^{d_{\ab}}$ denote the so-called \emph{horizontal torus}, 
which is the largest torus that is a factor of $X$, and we let $\pi_{\ab} \colon X \to X_{\ab}$ denote the corresponding projection. 
We cite a criterion for equidistribution obtained by Leibman.

\begin{theorem}[{\cite[Thm.\ C]{Leibman-2005}}]\label{thm:Leib-nil}
	Let $g \colon \ZZ \to G$ be a polynomial map. Then $\bra{g(n)\Gamma}_{n=0}^\infty$ is equidistributed in $X$ 
	if and only if $\bra{\pi_{\ab}(g(n)\Gamma)}_{n=0}^\infty$ is equidistributed in $X_{\ab}$.
\end{theorem}

A \emph{horizontal character} $\eta \colon G \to \RR/\ZZ$ is a continuous homomorphism with $\Gamma \subset \ker(\eta)$. 
Note that $\eta$ vanishes on $[G,G]\Gamma$ and hence induces a continuous homomorphism $X_{\ab} \to \RR/\ZZ$. 
Identifying $X_{\ab}$ with $(\RR/\ZZ)^{d_{\ab}}$ allows us to identify $\eta$ with a vector $k \in \ZZ^{d_{\ab}}$, and 
we set $\abs{\eta} = \norm{k}_\infty$. The identification of $X_{\ab}$ with $(\RR/\ZZ)^{d_{\ab}}$ is generally not unique, but Mal'cev coordinates 
provide one distinguished choice. We cite a simplified variant of the quantitative version of Leibman's theorem.

\begin{theorem}[{\cite[Thm.\ 1.16]{GreenTao-2012}}]\label{thm:GT}
	Let $g \colon \ZZ \to G$ be a polynomial map. Then there exists a positive real number $C$, which depends on $X$, $\tau$, and $g$,  
	such that the following holds. For any $\delta \in (0,1/2)$ and $N \in \NN$ such that the sequence $(g(n)\Gamma)_{n=0}^{N-1}$ is 
	not $\delta$-equidistributed, there exists a horizontal character $\eta \colon G \to \RR/\ZZ$ such that $0 < \abs{\eta} \leq 1/\delta^{C}$ 
	and $\norm{\eta \circ g}_{C^\infty[N]} \leq 1/\delta^C N$. 
\end{theorem}

We record the following consequence of Theorem \ref{thm:GT}, combined with the Schmidt subspace theorem.  
Below, we say that a polynomial map $g \colon \ZZ \to G$ has algebraic coefficients if the corresponding map $\tilde \tau \circ g$ 
is a polynomial with algebraic coefficients (cf.\ eq.\ \eqref{eq:Malcev-G}).
	
\newcommand{\temptau}{\delta}	
\begin{lemma}\label{lem:quant-equi-algebraic}
Let $g \colon \ZZ \to G$ be a polynomial map with algebraic coefficients. Suppose that the sequence $(g(n)\Gamma)_{n=0}^\infty$ is equidistributed. 
Then there exists a constant $c > 0$ (dependent on $G,\Gamma,\tau$, and $g$) such that for each $N \in \NN$ the sequence 
$\bra{ g(n)\Gamma }_{n=0}^{N-1}$ is $O(N^{-c})$-equidistributed in $X$.
\end{lemma}

\begin{proof}
We argue by contradiction, assuming that the sequence $\bra{ g(n)\Gamma }_{n=0}^{N-1}$ is not $N^{-\temptau}$-equidistributed in $X$ 
for some small $\temptau > 0$, to be determined in the course of the argument. It follows from Theorem \ref{thm:GT} that there exists a real number $C > 0$ 
and a horizontal character $\eta \colon G \to \RR/\ZZ$ such that $\abs{\eta} \leq N^{C \temptau}$ and $\norm{\eta \circ g}_{C^\infty[N]} \leq N^{C\temptau}$. 
 
Letting $\pi_{\ab} \colon X \to X_{\ab}$ denote the projection to the abelian torus and identifying $X_{\ab}$ with $(\RR/\ZZ)^{d_{\ab}}$, 
we can expand
\begin{equation}\label{eq:923:1}
	\pi_{\ab}(g(n)) = \sum_{j=0}^r \a^{(j)} n^j\,,
\end{equation}
where for every $j$,  $0 \leq j \leq r$, $\a^{(j)} = (\a^{(j)}_i)_{i=1}^{d_{\ab}} \in X_{\ab}$. 
Identifying $\eta$ with a vector $k = (k_1,k_2,\dots,k_{d_{\ab}}) \in \ZZ^{d_{\ab}}$, we have
\begin{equation}\label{eq:923:2}
	\fpa{ k \cdot \a^{(j)} } = \fpa{\sum_{i=1}^{d_{\ab}} k_i \a^{(j)}_i } \leq N^{C\temptau - j} \leq N^{C\temptau - 1} \,,
\end{equation}
 for every $j$,  $0 \leq j \leq r$.
We have not guarantee that $1,\a^{(j)}_1,\a^{(j)}_2,\dots,\a^{(j)}_{d_{\ab}}$ are $\QQ$-linearly independent for some $j$.   
However, we can find a linear combination $\b = (\b_i)_{i=1}^{d_{\ab}}=  \sum_j w_j \a^{(j)}$, with $w_j \in \ZZ$, 
such that $1,\b_1,\b_2,\dots,\b_{d_{\ab}}$ are $\QQ$-linearly independent. Then, we infer from \eqref{eq:923:2} that 
\begin{equation}\label{eq:923:3}
	\fpa{ k \cdot \b } = \fpa{\sum_{i=1}^{d_{\ab}} k_i \b_i } \ll N^{C\temptau - 1}\, .
\end{equation}
On the other hand, it follows from the subspace theorem (see \cite{Schmidt-1972}) that 
\begin{equation}\label{eq:923:4}
	\fpa{ k \cdot \b } = \fpa{\sum_{i=1}^{d_{\ab}} k_i \b_i } \gg \bra{ \max_{1 \leq i \leq d_{\ab}} k_i}^{-{d_{\ab}}-1} \gg N^{-C\temptau\bra{{d_{\ab}}+1}} \, .
\end{equation}	
Provided that $N$ is sufficiently large and $\temptau$ is sufficiently small, \eqref{eq:923:3} and \eqref{eq:923:4} are contradictory.
\end{proof}

\subsection{Orbit closures}

Lastly, we discuss the behaviour of bounded {\gp} sequences under shifts. First, we record a consequence of Theorem \ref{thm:BL-mini} 
and other results in \cite{BergelsonLeibman-2007}. Roughly speaking, it asserts that the class of bounded {\gp} maps representable by 
a formula of a given ``shape'' is closed under shifts. (See also Section \ref{sec:sc-induction} for an alternative formulation.)

\begin{lemma}\label{lem:nil:shift}
	Let $g \colon \ZZ \to \RR$ be a bounded {\gp} map. Then there exists $d \in \NN_0$ and a {\gp} map $h \colon \RR^d \times \ZZ \to \RR$ 
	such that, for every $m \in \ZZ$, there exists $\vec \a \in [0,1)^d$ such that $g(n+m) = h(\vec \a,n)$ for all $n \in \ZZ$.
\end{lemma}

\begin{proof}
	Let $g(n) = f(T^n(x))$ be the representation of $g$ coming from Theorem \ref{thm:BL-mini} (using the notation therein).  
	Assume first that $X$ is connected and let $\tau \colon X \to [0,1)^d$ be some Mal'cev coordinates on $X$. 
	Set $h({\vec \a},n) = f(T^n(\tau^{-1}(\vec \a)))$. Then $h$ is a {\gp} map by \cite[Sec.{} 1.15]{BergelsonLeibman-2007}. 
	For $m \in \ZZ$, let $\vec \a (m) = \tau(T^m(x))$. Then $h({\vec \a(m)},n) = g(n+m)$ for all $n \in \ZZ$. 
	If $X$ is not connected, then we can reduce to the connected case by passing to an arithmetic progression.
\end{proof}

Finally, we mention a closure property of bounded {\gp} maps from $\NN$ to $\RR$. We stress that the same conclusion does not hold after replacing $\NN$ with $\ZZ$. 

\begin{lemma}\label{lem:cl:limit}
	Let $d \in \NN$, $h \colon \RR^d \times \NN \to \RR$ be a {\gp} map, and $(x_i)_{i=0}^\infty$ be a bounded sequence with values in $\RR^d$. 
	Let us assume that the limit 
	\[
		g(n) = \lim_{i \to \infty} h(x_i,n) 
	\]
	exists for all $n \in \NN$. Then $g$ is a {\gp} map.
\end{lemma}
\begin{proof}
	This is a special case of \cite[Prop.{} 2.16]{Konieczny-2021-JLM}. Note that the assumption that the limit above exists allows us to avoid the use of limits along ultrafilters, which appear in \cite{Konieczny-2021-JLM}.
\end{proof}

\bibliographystyle{alphaabbr}
\bibliography{bibliography}

\end{document}